\documentclass[11pt]{amsart}

\theoremstyle{plain}
\newtheorem{theorem}{Theorem}[section] 
\newtheorem{lemma}[theorem]{Lemma}     
\newtheorem{corollary}[theorem]{Corollary}
\newtheorem{proposition}[theorem]{Proposition}
\newtheorem{remark}[theorem]{Remark}
\newtheorem{definition}[theorem]{Definition}
\newtheorem*{theorem*}{Theorem}
\newtheorem*{definition*}{Definition}
\newtheorem{example}{Example}[subsection]

\newcommand{\vp}{\varepsilon}

\begin{document}
	\title{Local inverse measure-theoretic entropy for endomorphisms}
	\author{Eugen Mihailescu}
		
	\author{Radu B. Munteanu}
	\maketitle

	\begin{abstract}
	We introduce a new notion of local inverse metric entropy along backward trajectories for ergodic measures preserved by endomorphisms (non-invertible maps) on a compact metric space. A second notion of  inverse measure entropy is  defined  by using measurable  partitions. Our notions have several useful applications. Inverse entropy can distinguish between isomorphism classes of endomorphisms on Lebesgue spaces, when they have the same forward measure-theoretic entropy. 
	In a general setting we prove that the local inverse entropy of an ergodic measure $\mu$ is equal to the forward entropy minus the folding entropy,  i.e $h^-_f(\mu)=h_f(\mu)- F_f(\mu)$.  The inverse entropy of hyperbolic measures on compact manifolds is explored, focusing on their negative Lyapunov exponents. 
	 We compute next the inverse entropy of the inverse SRB measure on a  hyperbolic repellor.
	 We prove  an entropy rigidity result for special Anosov endomorphisms of $\mathbb T^2$, namely that they can be classified up to smooth conjugacy by knowing  the entropy of their SRB measure and the inverse entropy of their inverse SRB measure, namely $(h_f(\mu_f^+), h_f^-(\mu_f^-))$.
	Next we study the relations between our inverse measure-theoretic entropy and the generalized topological inverse entropy on subsets of prehistories. In general we establish a Partial Variational Principle for  inverse entropy. We obtain also a Full Variational Principle for inverse entropy in the case of special TA-covering maps on tori.  In the end, several examples of endomorphisms are studied, such as fat baker transformations, fat solenoidal attractors, special Anosov endomorphisms, toral endomorphisms, 
	and the local inverse entropy is computed for their SRB measures.  
 
	\end{abstract}
	\textbf{2020 MSC}: 37A35; 37B40; 37C40; 37D25. 
	
	\textbf{Keywords}: Dynamics of non-invertible maps; Jacobian of a measure;  inverse topological entropy;  folding entropy; special endomorphisms; SRB measures; smooth rigidity; Lyapunov exponents.

	\section{Introduction}
	
	Let $(X,d)$ be a compact metric space and $\mathcal{B}$ be the $\sigma$-algebra of Borel sets.  Let $f:X\rightarrow X$ be a measurable map, which is in general non-invertible, and denote by  $M_f(X)$ the set of all $f$-invariant probability measures on $X$. The notion of entropy $h_f(\mu)$ of an $f$-invariant probability measure $\mu$ is central in Ergodic Theory (for eg \cite{KH}, \cite{Ro}, \cite{W}). In \cite{Bo} Bowen introduced a notion of topological entropy on non-compact sets and proved that the entropy $h_f(\mu)$  of an $f$-invariant measure $\mu$ is equal to the topological entropy of the set $G_\mu$ of generic points for $\mu$.

	For $\mu\in M_f(X)$, Brin and Katok introduced in \cite{B} a notion of \textbf{local entropy} and related it to the measure-theoretic entropy $h_f(\mu)$. They showed that if $\mu$ is ergodic and $B_n(x,\varepsilon)=\{y\in X  : d(f^i(x),f^i(y))<\varepsilon, 0\leq i\leq n\}$ is the $(n,\varepsilon)$-Bowen ball centered at $x$, then for $\mu$-a.e $x\in X $ we have, 
	\begin{equation*}\label{brka}
		h_f(\mu)= \lim_{\varepsilon \rightarrow 0}\liminf_{n\rightarrow \infty}\frac{-\log \mu (B_n(x,\varepsilon))}{n}=\lim_{\varepsilon \rightarrow 0}\limsup_{n\rightarrow \infty}\frac{-\log\mu (B_n(x,\varepsilon))}{n}.
	\end{equation*}
	If $(X,\mathcal{B},\mu)$ is a Lebesgue space (\cite{Ro}) and $f:X\to X$ is a countable-to-one endomorphism such that $\mu$ is an $f$-invariant probability measure on $X$, then Parry introduced and studied in \cite{P} the notion of Jacobian $J_f(\mu)$ of $\mu$. In \cite{R} Ruelle introduced the notion of folding entropy $F_f(\mu)$ for $\mu$, defined as the conditional entropy $H_\mu(\epsilon | f^{-1}\epsilon)$, where $\epsilon$ is the single point partition of $X$ and $f^{-1}\epsilon$ is the partition given by the fibers of $f$.
	
	In general, the dynamics of endomorphisms (non-invertible maps) is different from that of homeomorphisms/diffeomorphisms. For smooth endomorphisms $f:M \to M$ defined on a compact manifold $M$ which are hyperbolic on a compact subset $\Lambda \subset M$, we have that the unstable tangent spaces $E^u(\hat x)$ and the local unstable manifolds $W^u(\hat x)$ depend on the prehistory (full backward trajectory) $\hat x \in \hat \Lambda$, not only on the base points, where $\hat \Lambda$ is the inverse limit of the system $f|_\Lambda : \Lambda \to \Lambda$ (see \cite{Ru}). Also dimension theory for hyperbolic non-invertible maps is different from that of hyperbolic diffeomorphisms (\cite{Pe}), see for example \cite{Ru}, \cite{Pr}, \cite{M11}, \cite{M}, \cite{MU2}.
 Differences between the dynamics of invertible maps and that of non-invertible maps can be seen for example in the case of Anosov maps on tori.  
Any Anosov diffeomorphism $f: \mathbb T^d \to \mathbb T^d$ is topologically conjugate to its linearization $f_L : \mathbb T^d \to \mathbb T^d$, whose integer-valued matrix is determined by the induced map on the fundamental group $f_*: \pi_1(\mathbb T^d) \to \pi_1(\mathbb T^d)$, as shown in \cite{Fr} and \cite{Ma}.
	On the other hand, we have Anosov endomorphisms which are non-invertible on $\mathbb T^d$. In this case the unstable tangent space $E^u({\hat x})$ (and the unstable global set $W^u(\hat x)$) depend on the prehistory (backward trajectory) $\hat x$  from the inverse limit of $(\mathbb T^d,f)$.  In \cite{Pr} Przytycki showed that for many hyperbolic endomorphisms on tori there are infinitely many points which have infinitely many unstable tangent spaces corresponding to different prehistories. Another example of hyperbolic endomorphism (non-Anosov) with infinitely many unstable manifolds through certain points was given in \cite{M11}. The hyperbolic endomorphisms which have the property that their unstable sets depend only on the respective base point and not on the whole backward trajectory are called \textit{special}.  In \cite{AH} Aoki and Hiraide showed that any Anosov endomorphism $f: \mathbb T^d \to \mathbb T^d$ is homotopic to a linear toral endomorphism  $f_L:\mathbb T^d \to \mathbb T^d$ (called the linearization of $f$) whose integer-valued matrix is given by the induced homomorphism of $f$ on the fundamental group $f_*: \pi_1(\mathbb T^d) \to \pi_1(\mathbb T^d)$. Later Sumi proved in \cite{S} that, if $f$ is a special TA-map, then $f$ is topologically conjugate to $f_L$ (see also \cite{MT}). Moreover, special toral endomorphisms and the problem of rigidity and integrable unstable bundle were investigated in detail in \cite{An}, \cite{Mi}. For instance in \cite{An} it was proved that if $f:\mathbb T^2 \to \mathbb T^2$ is a non-invertible Anosov map with 1-dimensional stable bundle, then $f$ has integrable unstable bundle (i.e $f$ is special)  if and only if
every periodic point of $f$ admits the same Lyapunov exponent on the stable bundle. And  in \cite{Mi} there were obtained conditions when a special Anosov endomorphism on $\mathbb{T}^d$ is smoothly conjugated to its linearization.
	In  the case of an Anosov  endomorphism there exists the Sinai-Ruelle-Bowen (SRB) measure introduced and studied in \cite{Si}, \cite{Ru-TF}, \cite{Bo-carte}, which describes the distribution of forward iterates of Lebesgue-a.e point, and which satisfies the Pesin entropy formula (\cite{Y}). Also in \cite{M4} Mihailescu introduced the inverse SRB measure, which describes the distribution of sets of $n$-preimages of Lebesgue-a.e point, and which satisfies an inverse Pesin formula. 
	While the (forward) topological entropy of a continuous map $f$ takes in consideration the forward iterates of  points, there are several notions of entropy which employ preimages of points with respect to endomorphisms, studied in  \cite{CN}, \cite{H}, \cite{MU2}, \cite{Ni}, \cite{WZ2} among others. In \cite{H} Hurley introduced two preimage entropies $h_m(f)$ and $h_p(f)$ defined by 
	$$h_{p}(f)=\sup_{x\in X} \lim_{\epsilon\rightarrow 0}\limsup_{n\rightarrow \infty}\frac{1}{n}\log s(n,\epsilon, f^{-n}x), \ \ h_{m}(f)=\lim_{\epsilon\rightarrow 0}\limsup_{n\rightarrow \infty}\frac{1}{n} \log \sup_{x\in X}s(n,\epsilon, f^{-n}x),$$
	where $s(n,\epsilon, f^{-n}x)$ is the maximal cardinality of  $(n,\varepsilon)$-separated subsets of $f^{-n}x$, for $x\in X$. In \cite{CN} Cheng and Newhouse introduced another version of topological preimage entropy 
	$h_{pre}(f)=\lim_{\epsilon\rightarrow 0}\limsup_{n\rightarrow \infty}\frac{1}{n}\log\sup_{x\in X, k\geq n} s(n,\epsilon, f^{-k}x).$
	They also defined a measure theoretic \textit{preimage entropy} $h_{pre,\mu}(f)$ for an $f$-invariant measure $\mu \in M_f(X)$, $h_{pre,\mu}(f)=\sup\limits_\alpha h_\mu(\alpha|\mathcal{B}^-)$ with $\alpha$ varying over all finite partitions of $X$ and $ h_\mu(\alpha|\mathcal{B}^-)= \lim\limits_{n\rightarrow\infty}\frac{1}{n}H_\mu(\alpha^n| \mathcal{B}^-)$ 
	where $\alpha^n=\bigvee\limits_{j=0}^{n-1} f^{-j}\alpha$, $H_\mu ( \cdot  | \cdot )$ is the conditional entropy and $\mathcal{B}^-=\bigcap\limits_{n=0}^\infty f^{-n}\mathcal{B}$. Also \cite{CN} obtained a Variational Principle $h_{pre}(f)=\sup_{\mu\in M_f(X)}h_{pre, \mu}(f).$
    In \cite{WZ2} Wu and Zhu defined for a measurable partition $\alpha$ of $X$,  
	$h_{m,\mu}(f,\alpha)= \limsup\limits_{n\rightarrow\infty}\frac{1}{n}H_\mu(\alpha^n| f^{-n}\mathcal{B})$
	and then the pointwise metric preimage entropy of $f$ with respect to $\mu$ is defined by
	$h_{m,\mu}(f)=\sup_\alpha h_{m,\mu}(f,\alpha)$
	where $\alpha$ ranges over all finite partitions of $X$. Also \cite{WZ2} proved a Variational Principle  
	$h_{m}(f)=\sup_{\mu\in M_f(X)}h_{m, \mu}(f) = h_{pre}(f),$
	and showed moreover  that $$h_{m,\mu}(f)=h_{pre,\mu}(f)=F_f(\mu),$$ where $F_f(\mu)$ is the folding entropy of $\mu$. If $f$ is a homeomorphism and $\mu\in M_f(X)$, recall that  $h_{pre, \mu}(f)=h_{m,\mu}(f)=F_f(\mu)=0$.


In \cite{MU2} Mihailescu and Urba\'nski introduced an inverse topological pressure $P^-_f$ and the inverse topological entropy $h^-_f$. This inverse topological entropy is different from the preimage entropies studied in \cite{CN}, \cite{H}, \cite{WZ2}. 
In \cite{MU3} we introduced the \textit{asymptotic degree} of an invariant measure $\mu$ with respect to a hyperbolic endomorphism, and proved that it is related to the folding entropy $F_f(\mu)$.


	\

 \ \ \  \textbf{Main goals of the paper:}

	Let a compact metric space $(X, d)$ and $f:X \to X$ measurable map (in general non-invertible), and denote by $\widehat{X}_f$ (or $\widehat X$) the \textit{inverse limit} of $(X,f)$, $$\widehat{X}_f=\{\hat{x}=(x,x_{-1},x_{-2},\ldots) : \ x_{-i}\in X,\ f(x_{-i})=x_{-i+1}, i\geq 1, x_0=x\}.$$
	When the  map $f$ is clear from the context, we denote this inverse limit by $\widehat X$ to simplify notation.
	Let  $\pi:\widehat{X}\rightarrow X$ be the canonical projection defined by $\pi(\hat{x})=x$. The map $\hat{f}: \widehat{X}\rightarrow\widehat{X}$, 
	$$\hat{f}(x,x_{-1},x_{-2},\ldots)=(fx, x,x_{-1},x_{-2},\ldots), \ \hat x \in \widehat X, $$
	is bijective and bi-measurable. If $\mu$ is $f$-invariant and ergodic then there exists a unique $\hat{f}$-invariant and ergodic measure $\hat{\mu}$ on $\widehat{X}$ such that $\pi_*\hat\mu=\mu$ (see \cite{Ru}). Since $(X,d)$ is a compact metric space then $\widehat{X}$ is also a compact metric space endowed with the metric
	$\hat{d}(\hat{x},\hat{y})=\sum_{i\geq 0}\frac{d(x_{-i},y_{-i})}{2^i}, \text{ for }\hat{x}, \hat{y}\in \widehat{X}$. 
	For $\hat{x}\in\widehat{X}$, $\varepsilon>0$ and $n\geq 1$, define the \textbf{$(n,\varepsilon)$-inverse Bowen ball along $\hat{x}$} by, 
	\begin{equation}\label{invBowenball}
		B_n^-(\hat{x},\varepsilon)=\{y\in X  : \ \exists \ \hat{y}=(y, y_{-1}, \ldots)\in\widehat{X }\text{ such that }d(x_{-i},y_{-i})<\varepsilon, 0\leq i\leq n\}.
	\end{equation}
Then $B_{n}^-(\hat{x},\varepsilon)=f^{n}(B_{n}(x_{-n},\varepsilon)),$
for
$B_{n}(x_{-n},\varepsilon)=\{y\in X:  d(f^i(x_{-n}),f^i{y})<\varepsilon, 0\leq i\leq n\}$
the usual $(n,\varepsilon)$-Bowen ball.

	Our main goal is to introduce and study new notions of \textbf{local inverse entropies} for an ergodic measure $\mu$ along \textbf{individual backward trajectories} with respect to an endomorphism $f$ (non-invertible map). We introduce  a notion of local inverse metric entropy of $\mu$ using inverse Bowen balls, and then a notion of inverse entropy  using measurable partitions. In  general we prove that they are equal and that the local inverse entropy  satisfies $h^-_f(\mu)=h_f(\mu)- F_f(\mu)$. These notions are different from previous notions of entropy defined using preimages.  Our setting presents new difficulties, and we develop new methods and applications.
	 Inverse entropy can distinguish between isomorphism classes of endomorphisms, when they have the same forward measure-theoretic entropy. 
	The inverse entropy of hyperbolic measures is explored also, focusing on their negative Lyapunov exponents. 
	We compute the inverse entropy of the inverse SRB measure $\mu_f^-$ on a connected hyperbolic repellor.
	 We prove  an entropy rigidity result for special Anosov endomorphisms of $\mathbb T^2$,  namely that they can be classified up to smooth conjugacy by  the entropy of their SRB measure and the inverse entropy of their inverse SRB measure, i.e by the pair of numbers $(h_f(\mu_f^+), h_f^-(\mu_f^-))$.  
	Next, we study relations between our inverse measure-theoretic entropy and the generalized topological inverse entropy on sets of prehistories, and establish a Partial Variational Principle for inverse entropy. We obtain a Full Variational Principle for inverse entropy in the case of special TA-covering maps on tori. In the end, several classes of examples are studied, fat baker's transformations, toral endomorphisms, Tsujii endomorphisms, and the inverse entropy of SRB measures  is computed.  

\

\ \ \ \ \ \textbf{Outline of the paper:}

$\bullet$ \ In \textbf{Section 2} we define several notions of local inverse entropy for an ergodic $f$-invariant probability measure $\mu$ on $X$. First,  for $\hat{x}\in \widehat{X}$ and $\varepsilon >0$, define the quantities 
$$h^-_{f,inf,B}(\mu,\hat{x},\varepsilon)=\liminf_{n\rightarrow\infty}\frac{-\log \mu(B_n^-(\hat{x},\varepsilon))}{n}, \ \ 
h^-_{f,sup,B}(\mu,\hat{x},\varepsilon)=\limsup_{n\rightarrow\infty}\frac{-\log \mu(B_n^-(\hat{x},\varepsilon))}{n}.$$
If $0<\varepsilon_1<\varepsilon_2$, then 
$h^-_{f,inf,B}(\mu,\hat{x},\varepsilon_1) \geq h^-_{f,inf,B}(\mu,\hat{x},\varepsilon_2)$ and $ h^-_{f,sup,B}(\mu,\hat{x},\varepsilon_1) \geq h^-_{f,sup,B}(\mu,\hat{x},\varepsilon_2)$.
Thus, the following limits exists and we define the local quantities
\begin{equation}\label{h-inf-sup}
h^-_{f,inf,B}(\mu,\hat{x})=\lim_{\varepsilon \rightarrow 0}h^-_{f,inf,B}(\mu,\hat{x},\varepsilon), \ \
h^-_{f,sup,B}(\mu,\hat{x})=\lim_{\varepsilon \rightarrow 0}h^-_{f,sup,B}(\mu,\hat{x},\varepsilon),
\end{equation}
which we call the \textbf{lower, respectively upper inverse metric entropy of $\mu$ at $\hat{x}$}.
The functions 
$$
h^-_{f,inf,B}(\mu, \cdot):\widehat{X}\rightarrow \mathbb{R} , \ \ h^-_{f,sup,B}(\mu, \cdot): \widehat{X}\rightarrow \mathbb{R}$$
are called the lower, respectively upper inverse metric entropy functions of $\mu$.
Let further define the \textbf{lower and upper inverse metric entropy of }$\mu$ by 
$$ h^-_{f,inf,B}(\mu)=\int_{\widehat{X}} h^-_{f,inf,B}(\mu,\hat{x})d\hat{\mu}(\hat{x}),  \ \  h^-_{f,sup,B}(\mu)=\int_{\widehat{X}}  h^-_{f,sup,B}(\mu,\hat{x})d\hat{\mu}(\hat{x}).$$


 In Lemma \ref{mod} we prove an inverse version of Brin-Katok Theorem.  
We obtain a formula that relates the inverse metric entropy of $\mu$ to the folding entropy of $\mu$ and the forward entropy of $\mu$.
\begin{theorem}\label{D}
	Let $f: X\rightarrow X$ be a continuous and locally injective transformation of the compact metric space $X$ and let $\mu$ be a probability measure on $X$ which is $f$-invariant and ergodic. Let $J_f(\mu)$ be the Jacobian of $\mu$ with respect to $f$. If the set $D$ of discontinuities of $J_f$ is closed and has $\mu$-measure zero and if $J_f(\mu)$ is bounded, then the inverse metric entropy of $\mu$ exists and  
	$$h^-_{f,B}(\mu)=h_{f}(\mu)-F_f(\mu).$$
\end{theorem}

Next, we define the lower/upper inverse entropy of an ergodic measure $\mu$ with respect to a measurable partition. 
Let $(X, \mathcal{B}, \mu)$ be a probability space and $\mathcal{P}$ be a measurable partition of $(X, \mathcal B, \mu)$.  If 
$\widehat{X}$ is the inverse limit of $(X,f)$ and $\pi : \widehat{X} \rightarrow X$ is the canonical projection, then $\widehat{\mathcal{P}}= \{ \pi^{-1}(P) \ | P \in \mathcal{P} \}$
is a measurable partition of $\widehat{X}$. 
For $n\geq 1$, let $\mathcal{P}_{n}= \bigvee\limits_{i=0}^{n} {f}^{-i}(\mathcal{P})$. If $x\in X$, let $\mathcal{P}(x)$ (respectively ($\mathcal{P}_{n}(x)$) be the atom of $\mathcal{P}$ (respectively $\mathcal{P}_{n}$) containing $x$. For $\hat{x}=(x_{-i})_{i\geq 0}\in \widehat{X}$ with $x_0=x$, define
$$\mathcal{P}_{n}^{-}(\hat{x})=\{y\in X \ | \ \exists \  \hat{y}= (y_{-i})_{i\geq 0} \text{ with }y_0=y, \text{ and }y_{-i}\in \mathcal{P}(x_{-i}),  \text{ for }i=0,1,\ldots, n \}, \ \text{and},$$
 $$h^-_{f, inf}(\mu,\mathcal{P},\hat{x})= \liminf_{n\rightarrow\infty}\frac{-\log \mu(\mathcal{P}_n^-(\hat{x}))}{n}, \ \  
h^-_{f,sup}(\mu, \mathcal{P}, \hat{x})= \limsup_{n\rightarrow\infty}\frac{-\log \mu(\mathcal{P}_n^-(\hat{x}))}{n}, \ \text{and},$$
$$h^-_{f,inf}(\mu,\mathcal{P})= \int_{\widehat{X}}h^-_{ f,inf}(\mu,\mathcal{P}, \hat{x}) \ d \hat{\mu}(\hat{x}), \ \  h^-_{f,sup}(\mu,\mathcal{P})= \int_{\widehat{X}}h^-_{ f,sup}(\mu,\mathcal{P}, \hat{x})  \ d\hat{\mu}(\hat{x}).$$ 
We define the \textbf{lower and the upper inverse partition entropy} of $\mu$ by
$$h^-_{f,inf}(\mu)=\sup\{ h^-_{f,inf}(\mu,\mathcal{P})\  : \mathcal{P} \text{ is a measurable partition with }  H_\mu(\mathcal{P})<\infty \}, $$
$$h^-_{f,sup}(\mu)=\sup\{ h^-_{f,sup}(\mu,\mathcal{P})\  : \mathcal{P} \text{ is a measurable partition with }  H_\mu(\mathcal{P})<\infty \},$$
where $H_\mu(\mathcal{P})$ is the entropy of the partition $\mathcal{P}$ (see \cite{W}).

\begin{definition}
	In case $h^-_{f,inf}(\mu)=h^-_{f,sup}(\mu)$, the common value is called the \textbf{inverse partition entropy} of $\mu$ with respect to $f$ and is denoted by $h^-_{f}(\mu)$. In this case we say that $\mu$ has inverse partition entropy. 
\end{definition}

 \begin{definition}
  	In the above setting, if the  upper and lower inverse metric entropy $h^-_{f,sup,B}(\mu)$ and $h^-_{f,inf,B}(\mu)$ are equal, then this common value is called the \textbf{inverse metric entropy of  $\mu$} with respect to $f$, denoted by $h^-_{f,B}(\mu)$; in this case we say that $\mu$ has inverse metric entropy.
  \end{definition}

  The inverse partition entropy is an isomorphism invariant of measure preserving endomorphisms.
 \noindent
 Our notion of inverse entropy $h^-_f(\mu)$ is \textbf{significantly different} from the notions of preimage entropies $h_{pre,\mu}(f)$, $h_{m,\mu}(f)$, $h_{pre}(f)$ and $h_m(f)$ studied in \cite{CN}, \cite{H}, \cite{WZ2}. Instead of considering the tree of all $n$-preimages of a point simultaneously, we study the  behavior along individual backward trajectories (prehistories). Thus our inverse entropy  emphasizes other aspects of non-invertible dynamics than the above  notions of preimage entropy, namely the dynamics of individual local inverse branches. This introduces \textbf{new challenges} and we develop new methods. The dynamical behavior on some prehistories can be very different from the behavior on other prehistories,  for example \cite{Ru}, \cite{M11}, \cite{Ts}. Thus it is important to study the ergodic theory also  in this setting. 
The differences between our inverse entropy and preimage entropies can be observed in the following cases: 
 
 a)  If $ f$ is a homeomorphism  and $\mu$ is an $f$-invariant ergodic measure, then our notion of inverse entropy $h^-_f(\mu)$ is equal to the usual forward entropy $h_f(\mu)$; while in this case, the preimage entropies satisfy $h_{pre, \mu}(f)=h_{m,\mu}(f)=F_f(\mu)=0$. 
 
 b) If $f$ is an expanding map, then any $f$-invariant ergodic measure $\mu$  has inverse entropy equal to zero (since local inverse branches are contracting), whereas $h_{pre, \mu}(f)=h_{m,\mu}(f)=F_f(\mu)$ may be nonzero. For instance for the expanding map $f:\mathbb{S}^1\rightarrow \mathbb{S}^1$, $f(z)=z^2$ and the $f$-invariant ergodic Haar (Lebesgue)  measure $\nu$ on $\mathbb{S}^1$, we have $h^-_f(\nu) = 0$, while  $h_{pre, \nu}(f)=h_{m,\nu}(f)=F_f(\nu)=\log 2 > 0$. 
 
 c) We shall see that our inverse entropy focuses especially on the generic behavior along the \textbf{contracting directions} (which determine the negative Lyapunov exponents of $\mu$). This is applied to ergodic measures for hyperbolic endomorphisms, such as those from \cite{MF}, \cite{Mi},  \cite{M11}, \cite{MU3}, \cite{Ts}.  
 
 d) We will show that the inverse measure-theoretic entropy can be applied to \textbf{distinguish between isomorphism classes} of Lebesgue spaces; this proves especially useful when the Lebesgue spaces have the same forward (usual)  entropy.


 We do not work with the whole set of $n$-preimages of a point simultaneously, but with individual backward trajectories. Another difficulty is that a point $x$ may belong to (possibly uncountably) many inverse Bowen balls  $B_n^-(\hat x, \vp)$ for various prehistories (backward trajectories) of $x$; examples of non-invertible systems were given (see \cite{Ru}, \cite{BR}, \cite{M11}, \cite{Ts}). Yet another difficulty is that the \textbf{Jacobian} of an ergodic measure $J_f(\mu)$ is just a measurable function in general,
 and $J_{f^n}(\mu)$ cannot be controlled on  Bowen balls $B_n(x_{-n}, \vp)$. Some results on the Jacobians of Gibbs measures with respect to the iterates $f^n, n \ge 1$ were given in certain cases in \cite{MF}, \cite{MU3}, but generally there are not many results in this direction. 

In a quite general setting we show below that:

 \begin{theorem}\label{t1}
	Let $f:X\rightarrow X$  be an endomorphism and $\mu$ a probability measure which is $f$-invariant and ergodic. Assume that there exists a finite measurable partition $\mathcal{A}$ of $X$ such that $f$ is injective on every atom of $\mathcal{A}$. Then $\mu$ has inverse partition entropy and 
	$$ h^-_{f}(\mu) =  h_f(\mu)- F_{f}(\mu).$$	
\end{theorem}

Our inverse entropy emphasizes the contracting directions, i.e the directions that give the \textbf{negative Lyapunov exponents} of $\mu$. For example Corollary \ref{corz} of Theorem \ref{t1} shows that if $f: M\rightarrow M$ is a $C^r$, $r>1$ endomorphism with no critical points on a compact Riemannian manifold $M$ and if $\mu$ is an ergodic $f$-invariant measure, then $\mu$ has inverse partition entropy and 
$$h^-_f(\mu)\leq  - \sum_{i: \lambda_i(\mu)<0}\lambda_i(\mu),$$
where $\lambda_i(\mu)$ are the Lyapunov exponents of $\mu$ taken with their multiplicities.

In certain cases the inverse entropy $h^-_f(\mu)$ is  equal to the absolute value of the sum of negative Lyapunov exponents of $\mu$. Consider the \textbf{inverse SRB measure} on a hyperbolic  repellor $\Lambda$ introduced in \cite{M4}; this measure describes the distribution of $n$-preimages of Lebesgue-a.e point in a neighborhood of $\Lambda$, when $n \to \infty$. In Proposition \ref{propz} we prove that if $\Lambda$ is a connected hyperbolic repellor for a smooth endomorphism $f: M\rightarrow M$ on a compact Riemannian manifold $M$ which is $d$-to-$1$ on $\Lambda$ and if $\mu_f^-$ is the inverse SRB measure on $\Lambda$, then $h^-_f(\mu_f^-)$ and $h^-_{f,B}(\mu_f^-)$ exist, and 
$$h^-_f(\mu_f^-)=h^-_{f,B}(\mu_f^-)=-\sum_{i: \lambda_i(\mu_f^-)<0}\lambda_i(\mu_f^-),$$
where the Lyapunov exponents $\lambda_i(\mu_f^-)$ are taken with their multiplicities.

Next, we will investigate when is the inverse metric entropy $h^-_{f,B}(\mu)$ equal to the inverse partition entropy $h^-_{f}(\mu)$, for an ergodic measure $\mu$.  In general, we show in Proposition \ref{c1} that $$h^-_{f,inf, B}(\mu)\leq h^-_{f, sup, B}(\mu)\leq h_{f}(\mu)-F_f(\mu)=h^-_f(\mu).$$ We define also the \textbf{zero boundary property} for a measure  (see Definition \ref{defzerob}) and prove that under some conditions the inverse metric and inverse partition entropies are equal. 

\begin{theorem}\label{zerob}
	Let $f:X\rightarrow X$ be a continuous and locally injective endomorphism of the compact metric space $X$ and $\mu$ be a probability Borel measure on $X$ which is $f$-invariant ergodic and satisfies the zero boundary condition. If the Jacobian $J_f(\mu)$ is bounded and $h_f(\mu)<\infty$, then $\mu$ has inverse metric entropy and inverse partition entropy, and 
	$$h^-_{f,B}(\mu)=h_{f}(\mu)-F_f(\mu)=h^-_f(\mu).$$
\end{theorem}

Next we study \textbf{hyperbolic measures} on a smooth manifold $M$, meaning that all their Lyapunov exponents are non-zero. For a hyperbolic measure $\mu$ there is a set of full $\hat\mu$-measure in $\widehat M$ where there exist local stable/unstable manifolds. We employ Pesin sets $\widehat R_\vp \subset \widehat M$ of prehistories for which the local stable/unstable manifolds exist and have size $\vp>0$. The theory of hyperbolic measures is well presented in \cite{BP}. We use mostly \textbf{special} hyperbolic measures (Definition \ref{preind}), in the sense that local unstable manifolds depend only on the base point, not on full prehistories. 

\begin{theorem}\label{thpreind}
	Let $f :M \rightarrow M$ be a $\mathcal{C}^2$ endomorphism defined on a compact Riemannian manifold. Let $\mu$ be an $f$-invariant ergodic measure and assume that $\mu$ is hyperbolic and special. Then, $\mu$ has inverse metric entropy  and $h^-_{f, B}(\mu)=h_f(\mu)- F_f(\mu)$, i.e for $\hat\mu$-a.e. $\hat x = (x, x_{-1}, \ldots)\in \widehat M$,
	$$\lim_{\varepsilon\rightarrow 0}\liminf_{n\rightarrow\infty} \frac{-\log \mu(f^{n}(B_{n}(x_{-n},\varepsilon)))}{n}=\lim_{\varepsilon\rightarrow 0}\limsup_{n\rightarrow\infty} \frac{-\log \mu(f^{n}(B_{n}(x_{-n},\varepsilon)))}{n}= h_{f}(\mu)- F_f(\mu).$$
\end{theorem}

\

$\bullet$ \ In \textbf{Section \ref{specialAnosov}} we study \textbf{special Anosov endomorphisms}, namely  Anosov endomorphisms on compact Riemannian manifolds whose unstable manifolds depend only on their respective base points. 
We focus  on Anosov endomorphisms of tori. 
It was shown by Aoki and Hiraide (\cite{AH}) and Sumi (\cite{S}) that if $f:\mathbb T^d \to \mathbb T^d$ is a special TA-covering map (in particular if $f$ is a special Anosov endomorphism without critical points), then $f$ is topologically conjugate by a homeomorphism $\Phi$ to its linearization $f_L$. Recall that the linear hyperbolic endomorphism $f_L$ is determined by the integer-valued matrix given by the induced homomorphism $f_*:\pi_1(\mathbb T^d) \to \pi_1(\mathbb T^d)$ on the fundamental group of $\mathbb T^d$ (which is $\mathbb Z^d)$. 
However, the topological conjugacy $\Phi$ may not be smooth. Thus it is important to find conditions when $\Phi$ is a smooth conjugacy. This \textbf{rigidity problem} was studied for Anosov diffeomorphisms and endomorphisms, and several conditions were given for eg in \cite{RL}, \cite{An}, \cite{Mi}. 
If $f$ is an Anosov endomorphism without critical points on $\mathbb T^2$, there exist the SRB measure $\mu_f^+$ (\cite{Si}, \cite{QZ}, \cite{Y}), and the inverse SRB measure $\mu_f^-$ (\cite{M4}). The SRB measure $\mu_f^+$ describes the distribution of forward orbits of Lebesgue a.e point in $\mathbb T^2$, while the inverse SRB measure $\mu_f^-$ describes the asymptotic distribution of the sets of $n$-preimages for Lebesgue a.e point.   
If $f_L$ is linear hyperbolic, then the SRB measure, inverse SRB measure, and Haar measure $m$ coincide.

In Theorem \ref{clasAnendos} we prove \textbf{entropy rigidity}, namely that a special Anosov endomorphism $f$ on $\mathbb T^2$ is smoothly conjugated to its linearization $f_L$ if and only if $f$ and $f_L$ have the same entropy of their SRB measures, and the same inverse entropy of their inverse SRB measures. So it is enough to know just two numbers,  $$\big(h_f(\mu_f^+), \  h_f^-(\mu_f^-)\big).$$

\begin{theorem}(Entropy rigidity for special Anosov endomorphisms on $\mathbb T^2$).\label{clasAnendos}

a) Let $f, g: \mathbb T^2 \to \mathbb T^2$ be $\mathcal C^\infty$ special Anosov endomorphisms without critical points having the same linearization $f_L = g_L$. Let $\Phi:\mathbb T^2 \to \mathbb T^2$ be the topological conjugacy between $f$ and $g$ and assume that $$h^-_f(\mu^-_f) = -\int \log |Dg_s|\circ \Phi \ d\mu_f^- \  \text{and} \ h_f(\mu_f^+) = \int\log|Dg_u|\circ \Phi \ d  \mu_f^+,$$ where $Dg_s(x) := Dg|_{E^s_g(x)}, Dg_u(x) := Dg|_{E^u_g(x)}, x \in \mathbb T^2$. Then, $\Phi$ is a smooth conjugation.


b) Let $f: \mathbb T^2 \to \mathbb T^2$ be a special $\mathcal C^\infty$ Anosov endomorphism without critical points having linearization $f_L$. Then $f$ is smoothly conjugated to $f_L$ if and only if $$h_f(\mu_f^+) = h_{f_L}(m) = \log |\lambda_u| \   \ \text{and} \ \ h_f^-(\mu_f^-) =  h_{f_L}^-(m) = -\log |\lambda_s|,$$
where $\lambda_u, \lambda_s$ are the eigenvalues of the matrix of $f_L$ and $m$ is the Haar measure on $\mathbb T^2$.
\end{theorem}

\

$\bullet$ \  In \textbf{Section \ref{linkstop}} we study the links between our measure-theoretic inverse entropy and the topological inverse entropy introduced in \cite{MU2}. We define a generalization  $h^-(Y,\widehat{A})$ of this inverse topological entropy for any subset $Y\subset X$, using covers with inverse Bowen balls along an arbitrary subset of prehistories $\widehat{A}\subset \widehat{X}$. If $\widehat{Y}\subset\widehat{X}$ and $Y=\pi(\widehat{Y})$ then denote $h^-(Y, \widehat{Y})$ simply by  $h^-(\widehat{Y})$, i.e.
$$ h^-(\widehat{Y})=h^-(Y, \widehat{Y}).$$
Recall that for any $f$-invariant ergodic measure $\mu$ on $X$ there exists a unique $\hat{f}$-invariant ergodic measure $\hat{\mu}$ on $\widehat{X}$ such that $\pi_*\hat{\mu}=\mu$. We relate our \textbf{generalized inverse topological entropy} with the local inverse metric entropy, in the following:

\begin{theorem}\label{newtop}
	Let $X$ be a compact metric space, $f: X\rightarrow X$ a continuous map and $\mu$ be an $f$-invariant ergodic measure on $X$ and $\widehat{Y}\subset \widehat{X}$ be a Borel set such that $\hat{\mu}(\widehat{Y})>0$ and $h^-_{f,inf,B}(\mu,\hat{x})\geq \alpha>0$ for every $\hat{x}\in \widehat{Y}$. Then, 
	$$\lim_{\delta\rightarrow 0}\left( \sup \{h^-(\widehat{\mathcal{A}}) : \widehat{\mathcal{A}}\subset \widehat{Y},   \hat{\mu}(\widehat{Y}\setminus \widehat{\mathcal{A}})<\delta\}\right)\geq \alpha.$$ 
\end{theorem}
  
\begin{theorem}\label{t2}
	Let $f: M\rightarrow M$ be a $\mathcal{C}^2$ smooth endomorphism defined on a compact Riemannian manifold $M$. Assume that $\mu$ is a hyperbolic $f$-invariant ergodic  measure on $M$ which is special. Let $\widehat{Y}\subset \widehat{M}$ be a Borel set. If $h^{-}_{f,sup,B}(\mu,\hat{x}) \leq \alpha$ for every $\hat{x}\in \widehat{Y}$, then $$\inf \{ h^-(\widehat Z), \  \widehat Z \subset \widehat Y,  \hat \mu (\widehat Z ) =\hat \mu (\widehat Y )\}\le \alpha.$$
\end{theorem}

We introduce a notion of \textit{special hyperbolic measure} by analogy to the notion of special endomorphism, as being a hyperbolic measure for which local unstable manifolds depend only on their base point, not on the entire prehistory. Clearly, if the endomorphism $f$ is hyperbolic and special, then any $f$-invariant measure is special hyperbolic. We obtain a \textbf{Partial Variational Principle} for  inverse entropy:

\begin{theorem}\label{varpr}(Partial Variational Principle for inverse entropy).
		Let $f: M\rightarrow M$ be a $\mathcal{C}^2$ smooth endomorphism on a manifold $M$. Then
	\begin{align*}
 \sup \{ \inf \{h^-(\widehat Z), \    \hat\mu(\widehat Z)=1 \}, \  & \mu \text{ special hyperbolic ergodic measure} \} \leq \\
\leq\sup \{h^-_{f,inf, B}(\mu),   \mu \text{ ergodic} \} 
& \leq\lim_{\delta\rightarrow 0} \left(\sup\{h^-(\widehat{\mathcal{A}}), \ \hat{\mu}(\widehat{\mathcal{A}})>1-\delta, \ \mu \text{  ergodic} \}\right).
	\end{align*}
\end{theorem}	
For special TA-covering maps on tori (i.e special topologically Anosov covering maps \cite{AH}) we obtain a \textbf{Full Variational Principle} for inverse entropy:

\begin{theorem}\label{VPtor} (Full Variational Principle for special TA-covering maps on tori). 
Let $f:\mathbb T^d \to \mathbb T^d$ be a special TA-covering map. Then we have $$h^-_{f}(\widehat{\mathbb {T}^d})= \sup \{h^-_f(\mu) :  \mu \text{ ergodic} \ f-\text{invariant}\}.$$
\end{theorem}
In particular Theorem \ref{VPtor} holds for smooth Anosov endomorphisms without critical points. 

\

$\bullet$ \ In \textbf{Section \ref{classesexp}} we present several classes of \textbf{Examples}. In subsection \ref{toralendo} we compute the inverse metric entropy of the Haar (normalized Lebesgue) measure $m$ for \textbf{hyperbolic toral endomorphisms}. If $A$ is a $d\times d$ hyperbolic matrix with integer entries and det$(A)\neq 0$, with $\lambda_1, \ldots, \lambda_d$ being its eigenvalues, and $f_A$ is the associated toral endomorphism on $\mathbb T^d$, then $$h^-_{f_A, B}(m)= h^-_{f_A}(m)= -\sum_{\{i: |\lambda_i|<1\}}\log|\lambda_i|.$$  

Moreover,  we show that the inverse entropy of the Haar measure $m$ can be used to \textbf{distinguish among isomorphism classes} of toral endomorphisms, even when they have the same forward entropy.
Let us recall the following definition:
 \begin{definition}
  Let $(X_i,\mu_i)$ be a Lebesgue space with a probability measure $\mu_i$ and $f_i: (X_i,\mu_i)\rightarrow (X_i,\mu_i)$ be a measure preserving endomorphism, $i=1,2$, Then we say that the systems $(X_1,\mu_1, f_1)$ and $(X_2,\mu_2,f_2)$ are \textbf{isomorphic} if there exists $\Phi: X_1\rightarrow X_2$ which is an isomorphism of measure spaces such that $\Phi_*\mu_1=\mu_2$ and $\Phi\circ f_1=f_2\circ \Phi$ $\mu_1$-a.e.	
 \end{definition}
In Example \ref{toralex}
 we consider the matrices
\begin{equation*}
	A_1=\left(\begin{array}{ccc}
		8 & 1 & 4\\
		0 & 3 & 1\\
		0 & 2 & 1
	\end{array}
	\right) \text { \ and \ }A_2=\left(\begin{array}{ccc}
		4 & 0 & 0\\
		3 & 6 & 2\\
		5 & 4 & 2
	\end{array}
	\right),
\end{equation*}
and let $f_{A_1}$ and $f_{A_2}$ be the associated endomorphisms on $\mathbb{T}^3$, and $m$ be the Haar measure on $\mathbb T^3$. Then the systems $ (\mathbb{T}^3, f_{A_1}, m)$ and $(\mathbb{T}^3, f_{A_2}, m)$  are not isomorphic, since their inverse entropies are different, even if their forward entropies are the same.

In subsection \ref{fatbaker} we study a class of hyperbolic non-invertible maps, namely the \textbf{fat baker's transformations} introduced by Alexander and Yorke in \cite{AY}.  They are defined on $[-1,1]\times[-1,1]$ by 
$$T_\beta(x,y)=\begin{cases}
(\beta x+(1-\beta), 2y-1)&   y\geq 0 \\
(\beta x-(1-\beta), 2y+1)&   y< 0,
\end{cases}$$
where $\frac{1}{2}<\beta< 1$. 
For infinitely many values of $\beta$ (for example for $\beta=\frac{\sqrt{5}-1}{2}$), the \textbf{SRB measure} (Sinai-Ruelle-Bowen measure)  $\mu_{SRB}^\beta$ of $T_\beta$ is totally singular with respect to the Lebesgue measure. There are also infinitely many values of $\beta$ in $(\frac{1}{2},1)$ for which the SRB measure of $T_\beta$ is absolutely continuous with respect to the Lebesgue measure. The general notion of overlap number $o(\mathcal{S},\mu)$ of a contractive iterated function system $\mathcal{S}$ with respect to an invariant measure $\mu$, was introduced  in \cite{MU4}. The topological overlap number  $o(\mathcal{S})$ is the overlap number  $o(\mathcal{S},\mu_{(\frac{1}{2},\frac{1}{2})})$, where $\mu_{(\frac{1}{2},\frac{1}{2})}$ is the equidistributed Bernoulli measure on $\Sigma_2^+$.
We prove in (\ref{fatformula}) that $$h_{T_\beta}^-(\mu_{SRB}^\beta) = \log 2 -\log o(\mathcal{S}_\beta),$$ where $o(\mathcal{S}_\beta)$ is the topological overlap number of  the iterated system $\mathcal{S}_\beta=\{S_1, S_2\}$ with $S_1(x)=\beta x+(1-\beta),  S_2(x)=\beta x-(1-\beta),  x \in [-1,1]$. 


Next, in subsection \ref{tsujii} we consider the \textbf{family of Anosov endomorphisms} introduced by Tsujii in \cite{Ts}, 
\begin{equation}\label{tsj}
T: S^1\times \mathbb{R}\rightarrow  S^1\times \mathbb{R},\ \  
T(x,y)=(lx, \lambda y+f(x)),
\end{equation}
where $l\geq 2$ is an integer, $0 < \lambda < 1$ is a real number and $f: S^1\rightarrow \mathbb{R}$ is a $C^2$ map. The map $T$ is an Anosov endomorphism. Therefore $T$ has an SRB measure $\mu_{SRB}^T$. If $\lambda l<1$, then $\mu_{SRB}^T$ is totally singular with respect to the Lebesgue measure because $T$ contracts area. If $\lambda l> 1$, then for some maps $T$ the SRB measure is totally singular with respect to the Lebesgue measure, while in other cases it is absolutely continuous with respect to the Lebesgue measure. In \textbf{Theorem \ref{Tsujiestimateinv}} we estimate the \textbf{inverse entropies of SRB measure} $\mu^T_{SRB}$  for a large subclass of  endomorphisms $T$, namely we show that:
$$\frac{|\log \lambda|}{2} \leq h^-_{T,inf, B}(\mu^T_{SRB}) \leq h^-_{T,sup, B}(\mu^T_{SRB})\leq h^-_{T}(\mu^T_{SRB})=|\log \lambda|.$$

\

\section{Inverse entropies of measures}

Let $(X,d)$ be a metric space and denote by $\mathcal{B}$ the sigma algebra of Borel sets. 
Let $\mu$ be a probability measure on $X$ and $f:X\rightarrow X$ a measurable map such that $\mu$ is $f$-invariant. In general $f$ is non-invertible. Assume moreover that $\mu$ is ergodic. We introduce a notion of inverse metric entropy of $\mu$ which is defined using inverse Bowen balls. We also introduce a notion of inverse entropy of $\mu$, defined using measurable partitions. Then we study the relations between these two notions of inverse entropy for an ergodic measure, including also the case when the zero boundary condition is satisfied. The inverse entropies of $\mu$ are then compared with the difference between the usual measure entropy and the folding entropy, i.e with $h_f(\mu) - F_f(\mu)$.  
Recall that the inverse limit $\widehat{X}_f$ of $(X,f)$ is defined by
$\widehat{X}_f=\{\hat{x}=(x,x_{-1},x_{-2},\ldots) : \ x_{-i}\in X,\ f(x_{-i})=x_{-i+1}, i\geq 1, x_0=x\}.$ 
The canonical projection $\pi:\widehat{X}_f\rightarrow X$ is defined by $\pi(\hat{x})=x$ and the map $\hat{f}: \widehat{X}_f\rightarrow\widehat{X}_f$ given by 
$$\hat{f}(x,x_{-1},x_{-2},\ldots)=(fx, x,x_{-1},x_{-2},\ldots)$$
is a homeomorphism. When $f$ is clear from the context we denote $\widehat{X}_f$ by $\widehat{X}$. 
If $\mu$ is an $f$-invariant ergodic measure then there exists a unique $\hat{f}$-invariant ergodic measure $\hat{\mu}$ on $\widehat{X}$ such that $\pi_*\hat\mu=\mu$.

Assume that $f$ is measurable and positively non-singular with respect to $\mu$, i.e. $\mu(A)=0$ implies $\mu(f(A)) =0$. Assume also that $f$ is essentially countable-to-one, i.e. the fibers $f^{-1}(x)$ are countable for $\mu$-a.e. $x\in X$. Then there exists a measurable partition
$\xi=\{A_0, A_1,\ldots \}$ of $X$ such that $f$ is injective on each $A_i$. The Jacobian $J_f(\mu)$ of $f$ with respect to $\mu$ is defined as (see \cite{P}),
$$J_f(\mu)(x)=\frac{d\mu\circ f|_{A_{i}}}{d\mu}, \text{ for }\mu-\text{a.e. }x\in A_i, i\geq 0.$$
This is a well defined measurable function and $J_f(\mu)(x)\geq 1$ for $\mu$-a.e. $x\in X$, since $f$ is one-to-one on $A_i$, and positively non-singular. The folding entropy $F_f(\mu)$ of $\mu$ with respect to $f$, introduced by Ruelle in \cite{R}, is 
defined as the conditional entropy 
$$F_f(\mu)=H_\mu(\epsilon|f^{-1}\epsilon),$$
where $\epsilon$ is the partition into single points of $X$ and $f^{-1}\epsilon$ is the partition into the fibers $f^{-1}(x)$, $x\in X$. From \cite{Ro}, we can disintegrate $\mu$ into a canonical family of conditional measures $\mu_x$ on the fiber $f^{-1}(x)$ for $\mu$-a.e. $x\in X$. Hence the entropy of the conditional measure $\mu_x$ is $H(\mu_x) =-\sum _{y\in f^{-1}(x)}\mu_x(y) \log \mu_x(y)$. From \cite{P} we have $J_f(\mu)(x) =\frac{1}{\mu_{f(x)}(x)}$, for $\mu$-a.e $x$, hence
$$F_f(\mu)=\int_X \log J_f(\mu)(x)d\mu(x).$$	
Since the Jacobian satisfies the Chain Rule	we have	
$$\log J_{f^{n}}(\mu)(x)=\log J_{f}(\mu)(x)+\log J_{f}(\mu)(fx)+\cdots +\log J_{f}(\mu)(f^{n-1}x),$$
for $\mu$-a.e. $x\in X $ and every $n\geq 1$. Since $\mu$ is ergodic, by Birkhoff Ergodic Theorem, 
\begin{equation}\label{Jaco}
\lim_{n\rightarrow\infty}\frac{\log J_{f^{n}}(\mu)(x)}{n}=\int_X\log J_f(\mu)(x)d\mu(x)=F_f(\mu) \text{ for }\mu-\text{a.e. }x\in X .
\end{equation}


\subsection{Inverse metric entropies for measures.}

In the above setting, recall that for $\hat{x}\in\widehat{X}$, $\varepsilon>0$, $n\geq 1$, the $(n,\varepsilon)$-inverse Bowen ball along $\hat{x}$ is defined by 
\begin{equation}\label{invBowen}
B_n^-(\hat{x},\varepsilon)=\{y\in X  : \ \exists \ \hat{y}=(y, y_{-1}, \ldots)\in\widehat{X }\text{ with }y_0=y \text{ such that }d(x_{-i},y_{-i})<\varepsilon, 0\leq i\leq n\}.
\end{equation}
For $\hat{x}\in \widehat{X}$ and $\varepsilon >0$,
\begin{equation}\label{pointwiseentr}
h^-_{f,inf,B}(\mu,\hat{x},\varepsilon)=\liminf_{n\rightarrow\infty}\frac{-\log \mu(B_n^-(\hat{x},\varepsilon))}{n}, \ \ 
h^-_{f,sup,B}(\mu,\hat{x},\varepsilon)=\limsup_{n\rightarrow\infty}\frac{-\log \mu(B_n^-(\hat{x},\varepsilon))}{n}.
\end{equation}
The following limits exist
$$h^-_{f,inf,B}(\mu,\hat{x})=\lim_{\varepsilon \rightarrow 0}h^-_{f,inf,B}(\mu,\hat{x},\varepsilon), \ \
h^-_{f,sup,B}(\mu,\hat{x})=\lim_{\varepsilon \rightarrow 0}h^-_{f,sup,B}(\mu,\hat{x},\varepsilon),$$
and are called the \textbf{lower, respectively upper inverse metric entropy of $\mu$ at $\hat{x}$}.
The functions 
$$h^-_{f,inf,B}(\mu, \cdot):\widehat{X}\rightarrow \mathbb{R} , \ \ h^-_{f,sup,B}(\mu, \cdot): \widehat{X}\rightarrow \mathbb{R}$$ are called the \textbf{lower, respectively upper inverse metric entropy functions of $\mu$}. Then the lower and upper inverse metric entropy of $\mu$ are defined by 
$$ h^-_{f,inf,B}(\mu)=\int_{\widehat{X}} h^-_{f,inf,B}(\mu,\hat{x})d\hat{\mu}(\hat{x}),  \ \  h^-_{f,sup,B}(\mu)=\int_{\widehat{X}}  h^-_{f,sup,B}(\mu,\hat{x})d\hat{\mu}(\hat{x}).$$
If the  lower and upper inverse metric entropies $h^-_{f,inf,B}(\mu)$ and $h^-_{f,sup,B}(\mu)$ are equal, then this common value is called the \textbf{inverse metric entropy of  $\mu$} with respect to $f$, and is denoted by $h^-_{f,B}(\mu)$. In this case we say that $\mu$ has inverse metric entropy.

\begin{proposition}
	With the above notations, if $f$ is continuous, then the upper and the lower inverse metric entropy functions are $\hat{f}$-invariant. Since $\mu$ is ergodic, these functions are constant $\hat\mu$-almost everywhere and $h^-_{f,inf,B}(\mu)=h^-_{f,inf,B}(\mu,\hat{x})$ and  $h^-_{f,sup,B}(\mu)=h^-_{f,sup,B}(\mu,\hat{x})$ for $\hat{\mu}$-a.e. $\hat{x}\in \widehat{X}$. 
\end{proposition}	
\begin{proof}
	Let $\varepsilon>0$. As $X$ is compact and $f$ continuous, there exists $\delta(\varepsilon)>0$ such that if $x,y\in X, d(x,y)<\delta(\varepsilon)$, then $d(f(x),f(y))<\varepsilon$. It is easy to see that 
	$B^-_n(\hat{x},\delta(\varepsilon))\subset f^{-1}(B^-_n(\hat{f}(\hat{x}),\varepsilon))$ 
	and then $\mu (B^-_n(\hat{x},\delta(\varepsilon)))\leq \mu (f^{-1}(B^-_n(\hat{f}(\hat{x}),\varepsilon)))=\mu (B^-_n(\hat{f}(\hat{x}),\varepsilon))$. Hence $h^-_{f,inf,B}(\mu,\hat{x})\geq h^-_{f,inf,B}(\mu, \hat{f}(\hat{x}))$ and $h^-_{f,sup,B}(\mu,\hat{x})\geq h^-_{f,sup,B}(\mu, \hat{f}(\hat{x})).$
	This implies that $h^-_{f,inf,B}(\mu, \cdot)$ and $h^-_{f,sup,B}(\mu, \cdot)$ are $\hat{f}$-invariant. Then since $\hat{\mu}$ is ergodic, it follows that  $h^-_{f,inf,B}(\mu, \cdot)$ and $h^-_{f,sup,B}(\mu, \cdot)$ are constant $\hat{\mu}$-a.e. 
\end{proof}

\begin{proposition}
	Let $X$ be a compact metric space and $f:X\rightarrow X$ be as above.   Then the upper and the lower inverse metric entropy functions do not depend on the metric on $X$ (if the metrics are equivalent). 
\end{proposition}	
\begin{proof}
	Let $d$ and $d'$ be two equivalent metrics on $X$. Since $d$ and $d'$ induce the same topology on $X$, the identity map $id : (X,d')\rightarrow(X,d)$ is continuous and by the compactness of $X$ also uniformly continuous. Let $\varepsilon>0$, then by uniform continuity of $id$ there exists $\delta(\varepsilon)>0$ such that $d(x,y)<\varepsilon$ if $d'(x,y)<\delta(\varepsilon)$. Thus $B_{n}^-(\hat{x},\varepsilon) \subset B^{'-}_{n}(\hat{x},\delta(\varepsilon))$, 
where $B_n^{'-}(\hat{x},\delta)=\{y\in X : \ \exists \ \hat{y}\in\widehat{X}\text{ such that }d'(x_{-i},y_{-i})<\delta, 0\leq i\leq n\}$. As $\delta(\varepsilon)\underset{\varepsilon\rightarrow 0}\longrightarrow 0$ it follows that	
$$\lim_{\varepsilon \rightarrow 0}\liminf_{n\rightarrow\infty}\frac{-\log \mu(B_n^{'-}(\hat{x},\varepsilon))}{n}\leq \lim_{\varepsilon \rightarrow 0}\liminf_{n\rightarrow\infty}\frac{-\log \mu(B_n^{-}(\hat{x},\varepsilon))}{n}.$$
	By interchanging $d$ and $d'$ we obtain 
	$$\lim_{\varepsilon \rightarrow 0}\liminf_{n\rightarrow\infty}\frac{-\log \mu(B_n^{-}(\hat{x},\varepsilon))}{n}\leq \lim_{\varepsilon \rightarrow 0}\liminf_{n\rightarrow\infty}\frac{-\log \mu(B_n^{'-}(\hat{x},\varepsilon))}{n}.$$
	and the lemma is proved. 
\end{proof}

\begin{proposition}
	Let $f$ be a continuous map of a compact metric space $X$ and  $\mu$ a probability measure on $X$ which is $f$-invariant. If $f$ is distance-expanding and open, then $h^-_{f,B}(\mu)=0$.	
\end{proposition}	
\begin{proof}
If $f$ is continuous distance-expanding and open, then there exists $\varepsilon_0>0$ such that for every $0<\delta<\varepsilon_0$ and any $\hat x \in\widehat{X}$ and $n\geq 1$, we have $B(x,\delta)\subset B_n^-(\hat{x},\delta)$. Then $0 \le -\log\mu(B_n^-(\hat{x},\delta)) \le  -\log\mu(B(x,\delta))$, and therefore $h^-_{f,B}(\mu)=0$.
\end{proof}

\begin{proposition}\label{conjtopol}
	Let $(X,d)$ and $(X', d')$ be two compact metric spaces, $f: X\rightarrow X$ and $f':X' \rightarrow X'$ be two continuous transformations and let $\mu$ and $\mu'$ be two Borel probability measures on $X$, respectively $X'$, such that 
	$\mu$ is $f$-invariant and $\mu'$ is $f'$-invariant. Assume that $\Phi : X\rightarrow X'$ is a topological conjugacy between $(X,f)$ and $(X',f')$ such that $\mu'=\Phi_*\mu$. Then for $\hat{x}\in \widehat{X}$,  $h^-_{f,inf,B}(\mu,\hat{x})=h^-_{f',inf,B}(\mu',\hat{\Phi}\hat{x})$ and $h^-_{f,sup,B}(\mu,\hat{x})=h^-_{f',sup,B}(\mu',\hat{\Phi}\hat{x})$.
\end{proposition}

\begin{proof}
	Let $d'$ be a metric on $X'$ and let $d$ be the metric on $X$ defined by
	$$d(x,y)=d'(\Phi(x), \Phi(y)).$$
	By the previous proposition the inverse entropy functions do not depend on the metric. Since $\mu'=\Phi_*\mu$, 
	the conclusion follows. 
\end{proof}

In the sequel we will need the following result: 
\begin{lemma}(Modified Brin-Katok Theorem)\label{mod}
	With the above notation, if $\mu$ is ergodic we have 
	$$\lim_{\varepsilon\rightarrow0}\liminf_{n\rightarrow\infty}\frac{-\log \mu(B_{n}(x_{-n},\varepsilon))}{n}
	=\lim_{\varepsilon\rightarrow0}\limsup_{n\rightarrow\infty}\frac{-\log \mu(B_{n}(x_{-n},\varepsilon))}{n}=h_{f}(\mu)$$
	for  $\hat{\mu}$-a.e. $\hat{x}\in \widehat{X}$, where $h_f(\mu)$ is the entropy of $\mu$ with respect to $f$. 
\end{lemma}
\begin{proof}
	Fix $\hat{x}=(x,x_{-1},\ldots)\in\widehat{X}$ and $\varepsilon>0$. Let the $(n,\varepsilon)$-Bowen ball of $\hat{f}^{-1}$ and $\hat{x}\in \widehat{X}$, $$\widehat{B}_{n}(\hat{x},\hat{f}^{-1}, \varepsilon)=\{\hat{y}\in \widehat X : \hat{d}(\hat{f}^{-i}(\hat{x}),\hat{f}^{-i}(\hat{y}))<\varepsilon, 0\leq i\leq n\}\subset \widehat{X}.$$
	We show first that 
	$$\hat{f}^{-n}(\widehat{B}_{n}(\hat{x},\hat{f}^{-1}, \varepsilon))\subset \pi^{-1}(B_{n}(x_{-n},\varepsilon)).$$
	Let $\hat{y}\in \widehat B_{n}(\hat{x},\hat{f}^{-1}, \varepsilon)$. Thus 
	$\hat{d}(\hat{f}^{-i}(\hat{x}), \hat{f}^{-i}(\hat{y}))<\varepsilon$ for $i=0,1,\ldots,n$. This clearly implies that $d(x_{-i},y_{-i})<\varepsilon$ for $i=0,1,\ldots,n$ and then $y_{-n}\in B_{n}(x_{-n},\varepsilon)$. Hence $\pi\hat{f}^{-n}(\hat{y})\in B_{n}(x_{-n},\varepsilon)$ and therefore $\hat{f}^{-n}(\hat{y})\in \pi^{-1}(B_{n}(x_{-n},\varepsilon))$. If $M=\sup \{d(x,y) : x,y\in X\}$, let $N$ be the smallest positive integer such that 
	$\sum_{j=1}^\infty\frac{M}{2^{j+N}}<\varepsilon$. We prove now that for $n>N$ and $k(n)=n-N$ we have
	\begin{equation}\label{rel1}
	\pi^{-1}(B_{n}(x_{-n},\varepsilon))\subset \hat{f}^{-n}(\widehat{B}_{k(n)}(\hat{x},\hat{f}^{-1},3\varepsilon)).
	\end{equation}
	Let $\hat{y}\in \pi^{-1}(B_{n}(x_{-n},\varepsilon))$. Thus $y\in B_{n}(x_{-n},\varepsilon)$ and then $d(f^i(y), x_{-n+i})<\varepsilon$	for $i=0,1,\ldots n$. It is easy to see that for $0\leq i\leq k(n)$ we have
	\begin{align*}
	\hat{d}(\hat{f}^{n-i}(\hat{y}),\hat{f}^{-i}(\hat{x}))&=d(f^{n-i}(y), x_{-i})+\frac{d(f^{n-i-1}(y), x_{-i-1})}{2}+\cdots +\frac{d(f^{n-k(n)}(y), x_{-k(n)})}{2^{k(n)-i}}+\cdots\\
	&+\frac{d(y, x_{-n})}{2^{n-i}}+\sum_{j=1}^\infty\frac{d(y_{-j},x_{-n-j})}{2^{n+j-i}}<\varepsilon+\frac{\varepsilon}{2}+\cdots+\frac{\varepsilon}{2^{n-i}}+\sum_{j=1}^\infty\frac{M}{2^{n+j-i}}.
	\end{align*}
	Since $i\leq k(n)$ we have $n+j-i\geq n-k(n)+j=N+j$, for $j\geq 1$. Thus we obtain that
	$$\hat{d}(\hat{f}^{n-i}(\hat{y}),\hat{f}^{-i}(\hat{x}))<2\varepsilon+\varepsilon=3\varepsilon, \text{ for }i=0,1,\ldots, k(n).$$
	which proves (\ref{rel1}). Hence, for $n$ large enough we have
	$$\hat{f}^{-n}(\widehat B_{n}(\hat{x},\hat{f}^{-1},\varepsilon))\subset\pi^{-1}(B_{n}(x_{-n},\varepsilon))\subset \hat{f}^{-n}( \widehat B_{k(n)}(\hat{x},\hat{f}^{-1},3\varepsilon)).$$
		By Brin-Katok Theorem (see \cite{B}) applied to $\hat{f}^{-1}$ and $\hat{\mu}$ on $\widehat{X}$, we obtain
	\begin{equation*}\label{rel2}
		\lim_{\varepsilon\rightarrow0}\liminf_{n\rightarrow\infty}\frac{-\log \hat{\mu}(\widehat{B}_n(\hat{x},\hat{f}^{-1},\varepsilon))}{n}
	=\lim_{\varepsilon\rightarrow0}\limsup_{n\rightarrow\infty}\frac{-\log \hat{\mu}(\widehat{B}_n(\hat{x},\hat{f}^{-1},\varepsilon))}{n}=h_{\hat{f}}(\hat{\mu}).
	\end{equation*}
	But $\hat{f}$ is $\hat{\mu}$-measure preserving,  $\mu(B_{n}(x_{-n},\varepsilon))=\hat{\mu}(\pi^{-1}(B_{n}(x_{-n},\varepsilon)))$ and $h_f(\mu)=h_{\hat{f}}(\hat{\mu})$. Thus our Lemma follows from this. 
\end{proof}

\begin{proposition}
	(i) Let $(X,d)$ be a compact metric space, $f:X\rightarrow X$ a continuous map and $\mu$ an $f$-invariant ergodic measure on $X$ which has inverse metric entropy. If $m\in \mathbb{N}^*$, then $$h^-_{f^m,B}(\mu)=m\cdot h^-_{f,B}(\mu).$$
	(ii) Let $(X_i,d_i)$ be a compact metric space,  $f_i: X_i\rightarrow X_i$ continuous transformation on $X_i$ and let
	$\mu_i$ a Borel probability measure on $X_i$ which is $f_i$-invariant and for which there exists the inverse metric entropy, $i=1,2$. On $X_1\times X_2$ we consider the metric $d$ defined by $d((x_1,x_2),(y_1,y_2))=\max\{d(x_1,y_1), d(x_2,y_2)\}$ and the product measure $\mu_1 \times\mu_2$. Then
	$$h^-_{f_1\times f_2,B}(\mu_1 \times\mu_2)= h^-_{f_1,B}(\mu_1)+ h^-_{f_2,B}(\mu_2).$$	
\end{proposition}
\begin{proof}
	(i) As usual, denote by $\widehat{X}$ the set of prehistories with respect to $f$ and let 
	$$\widehat{X}_m=\{\hat{x}_{(m)}=(x,x_{-m},x_{-2m},\ldots) : \ x_{-im}\in    X,\ f^{m}(x_{-im})=x_{-im+m}, i\geq 1, x_0=x\}$$ 
	be the set of prehistories with respect to $f^{m}$. Note that we have a canonical bijection between $\widehat{X}$ and $\widehat{X}_m$. For every $\hat{x}_{(m)}\in \widehat{X}_m$ let $$B_n^-(\hat{x}_{(m)},\varepsilon)=\{y\in X : \ \exists \ \hat{y}_{(m)}\in\widehat{X}_m\text{ such that }d(x_{-im},y_{-im})<\varepsilon, 0\leq i\leq n\}.$$
	As $f$ is uniformly continuous for every $\varepsilon>0$ there exists $\delta(\varepsilon)>0$ such that 
	$d(x,y)<\delta(\varepsilon)$ implies  $d(f^{i}(x),f^{i}(y))<\varepsilon$  for $i=0,1,\ldots, m$.
	Note that $B_{n}^-(\hat{x}_{(m)},\delta(\varepsilon))\subset B_{mn}^-(\hat{x},\varepsilon))\subset B_{n}^-(\hat{x}_{(m)}, \varepsilon)$ and 
	then $h^-_{f^m,B}(\mu)=m\cdot h^-_{f,B}(\mu)$.
	
	(ii) Note  if $\hat{x}, \hat{y}\in \widehat{X}$ then $\hat{z}=(\hat{x},\hat{y})$ is a prehistory of $(x,y)$ and $B_n^-(\hat{x},\varepsilon)\times B_{n}^-(\hat{y},\varepsilon)=B_n^-(\hat{z},\varepsilon)$.
\end{proof}

In the sequel, we will encounter several instances when we need the Jacobian of the measure to be bounded. A case when this happens is for the equilibrium measure of a H\"{o}lder continuous potential for a hyperbolic endomorphism. In \cite{MU3} it was proved the following result: 

\begin{theorem}\label{Jac}
	(Jacobians of equilibrium measures for endomorphisms \cite{MU3}). Let $f:M\rightarrow M$ be a $\mathcal{C}^2$ smooth map on a manifold $M$ such that $f$ is a hyperbolic endomorphism on a basic set $\Lambda$ and $f$ has no critical points in $\Lambda$. Let  $\phi$ be a H\"{o}lder continuous potential on $\Lambda$ and $\mu_\phi$ be the unique equilibrium measure of $\phi$ on $\Lambda$. Then there exists a constant $C>0$ such that for $\mu_\phi$-a.e. $x\in \Lambda$ and any $m \geq 1$, the Jacobian $J_{f^m}(\mu_\phi)$ of $\mu_\phi$ with respect to $f^m$ satisfies:
	\begin{equation*}
	C^{-1}\cdot\frac{\sum_{\zeta\in f^{-m}(f^m(x))\cap\Lambda}e^{S_m\phi(\zeta)}  }{e^{S_m\phi(x)}}\leq J_{f^m}(\mu_\phi)(x)\leq C\cdot\frac{\sum_{\zeta\in f^{-m}(f^m(x))\cap\Lambda}e^{S_m\phi(\zeta)}  }{e^{S_m\phi(x)}},
	\end{equation*}
	where  $S_m\phi(x)=\phi(x)+\cdots+\phi(f^{m-1}(x))$, for $x\in\Lambda$ and $m\in \mathbb{N}$.
\end{theorem}
As a consequence, we obtain:
\begin{corollary}\label{corol}
	In the same condition as in Theorem \ref{Jac}, $J_f(\mu_\phi)$ is bounded on $\Lambda$. 
\end{corollary}
\begin{proof}
	Since $\Lambda$ is compact and $f$ is locally injective, there exists $d\geq 1$ such that $\text{card}(f^{-1}(x)) \leq d$ for every $x\in\Lambda$. As $\phi$ is bounded, it follows from Theorem \ref{Jac} applied to $m=1$ that $J_f(\mu_\phi)$ is bounded.  
\end{proof}
Theorem \ref{D} assumes that the Jacobian of a measure is bounded. Such examples of measures with bounded Jacobian are given by Corollary \ref{corol}. Now we prove Theorem \ref{D} stated in Introduction:

\

\textbf{Proof of Theorem \ref{D}.}
	Since $f$ is locally injective and $X$ is compact, there exists $\delta_0$ such that for every $x\in X$, $f$ is injective on $B(x,\delta_0)$ and the boundary of $B(x,\delta_0)$ has $\mu$-measure zero. Consider a finite cover of $X$ with such balls and denote by $C$ the union of all boundaries of the balls from this cover. Recall that $D$ is the set of discontinuities of $J_f(\mu)$. Let $A=D\cup C$. 
	
	Let $\varphi :\widehat{X}\rightarrow\mathbb{R}$ defined by $\varphi(\hat{x})=\log J_f(\mu)(\pi(\hat{x}))$. Then Birkhoff's Ergodic Theorem, applied to $\varphi$ with respect to $\mu$ and $\hat{f}^{-1}:\widehat{X}\rightarrow\widehat{X}$ implies that for $\hat{\mu}$-a.e. $\hat{x}$, we have
	$
	\frac{1}{n}\sum_{i=1}^n \log J_f(\mu)(\pi(\hat{f}^{-i}\hat{x}))\underset{n\rightarrow\infty}\longrightarrow \int_{\widehat{X}} \log J_{f}(\mu)(\pi(\hat{x}))d\hat{\mu}(\hat{x}),
	$
	and therefore for $\hat{\mu}$-a.e. $\hat{x}\in \widehat X,$
	\begin{equation}\label{eq11}
	\frac{1}{n}\sum_{i=1}^n \log J_f(\mu)(x_{-i})  \underset{n\rightarrow\infty}\longrightarrow \int_{X} \log J_{f}(\mu)(x)d\mu(x)= F_f(\mu).
	\end{equation}
	Fix $\varepsilon>0$. For $0<\delta<\frac{\delta_0}{2}$ let \begin{align}\label{Kdelta}
	K_{\varepsilon,\delta}& =\{ x\in X \ : \ \log J_f(\mu)(x)-\varepsilon<  \log J_f(\mu)(y)<\log J_f(\mu)(x)+\varepsilon, \ \forall y\in B(x,\delta)   \}\\
	&= \{ x\in X \ : \ J_f(\mu)(x)e^{-\varepsilon}<  J_f(\mu)(y)<J_f(\mu)(x)e^{\varepsilon}, \ \forall y\in B(x,\delta)\} \nonumber.
	\end{align}
	Since the set $D$ of discontinuities of $J_f(\mu)$ is assumed to be closed, it follows that $X\setminus A$ is open. As $J_f(\mu)$ is continuous on $X\setminus A$ and $A=D\cup C$ has measure zero, then there is $\delta(\varepsilon)>0$ so that $$1-\varepsilon<\mu(K_{\varepsilon, \delta(\varepsilon)})\leq \mu(K_{\varepsilon,\delta}), \text{ for all }\delta<\delta(\varepsilon).$$ 
	By Birkhoff's Ergodic Theorem applied for $\chi_{\pi^{-1}(K_{\varepsilon,\delta(\varepsilon)})}$ we obtain that for $\hat{\mu}$-a.e $\hat{x}\in \widehat{X}$,
	\begin{equation*}
	\frac{1}{n} 	\sum_{i=0}^n\chi_{\pi^{-1}(K_{\varepsilon,\delta(\varepsilon)})}(\hat{f}^{-i}(\hat{x}))\underset{n\rightarrow\infty}\longrightarrow  \hat{\mu}(\pi^{-1}(K_{\varepsilon,\delta(\varepsilon)})).
	\end{equation*}
	Since for every $\hat{x}=(x,x_{-1},\ldots)\in\widehat{X}$, $\chi_{\pi^{-1} (K_{\varepsilon,\delta(\varepsilon)})}(\hat{f}^{-i}(\hat{x}))=\chi_{K_{\varepsilon,\delta(\varepsilon)}}(x_{-i})$, 
	it follows that for $\hat{\mu}$-a.e. $\hat{x}$
	\begin{equation}\label{eq2}
	\frac{1}{n} 	\sum_{i=0}^n\chi_{\pi^{-1}(K_{\varepsilon,\delta(\varepsilon)})}(\hat{f}^{-i}(\hat{x}))\underset{n\rightarrow\infty}\longrightarrow \mu(K_{\varepsilon,\delta(\varepsilon)}).
	\end{equation}
	For $\delta>0$, $\hat{x}\in \widehat{X}$ and $n\geq 1$ consider the sets $A(\varepsilon,\delta, \hat{x}, n)=\{i : 1\leq i\leq n, x_{-i}\in K_{\varepsilon,\delta} \}$ and  $R(\varepsilon,\delta,\hat{x}, n)=\{i : 1\leq i\leq n, x_{-i}\notin K_{\varepsilon,\delta} \}$, and let  
	$N(\varepsilon,\delta, \hat{x}, n)=\text{Card}(A(\varepsilon,\delta,\hat{x}, n))$. Notice that by (\ref{eq2}), for $\hat{\mu}$-a.e $\hat{x}$ we have 
	\begin{equation}\label{eq3}
	\lim_{n\rightarrow\infty}\frac{ N(\varepsilon,\delta(\varepsilon), \hat{x}, n)}{n} \ = \ \mu(K_{\varepsilon,\delta(\varepsilon)})>1-\varepsilon.
	\end{equation}
	Let $\widehat{X}(\hat{\mu}, \varepsilon)$ be the set of $\hat{x}$ from $\widehat{X}$ satisfying  (\ref{eq11}) and (\ref{eq3}). Then 
	\begin{equation}\label{-1}
	\hat{\mu}(\widehat{X}(\hat{\mu}, \varepsilon))=1.
	\end{equation}
	Let $\hat{x}\in\widehat{X}(\hat{\mu},\varepsilon) $ and $\delta<\delta(\varepsilon)$. Then there exists $n(\varepsilon,\hat{x})$ such that for $n\geq n(\varepsilon,\hat{x})$ we have
	\begin{equation}\label{0}
N(\varepsilon,\delta,\hat{x},n)\geq N(\varepsilon,\delta(\varepsilon),\hat{x},n)>n(1-\varepsilon), \ \text{and}
	\end{equation}
	\begin{equation}\label{jacfold}
	F_f(\mu)-\varepsilon < \frac {\log J_{f^n}(x_{-n})}{n} < F_f(\mu)+\varepsilon.
	\end{equation}
	Recall that we assumed that $J_f(\mu)$ is bounded. 
	Let $M>0$ such that $1\leq J_f(\mu)(z)<M$ for every $z\in X$. Then, for $\hat{x}\in \widehat{X}(\hat{\mu}, \varepsilon)$ and $y\in B_{n}(x_{-n},\delta)$ we have from (\ref{Kdelta}) and (\ref{jacfold})
	\begin{align*}
	J_{f^{n}}(\mu)(y)&=\prod_{i\in A(\varepsilon,\delta, \hat{x}, n)}J_f(\mu)(f^{n-i}y)\prod_{i\in R(\varepsilon,\delta, \hat{x}, n)}J_f(\mu)(f^{n-i}y)\leq  e^{N(\varepsilon,\delta, \hat{x}, n)\varepsilon} \prod_{i\in A(\varepsilon,\delta, \hat{x},n)}J_f(\mu)(x_{-i})  M^{n-N(\varepsilon,\delta, \hat{x}, n)}\\
	&\leq  e^{N(\varepsilon,\delta, \hat{x}, n)\varepsilon}\prod_{i\in A(\varepsilon,\delta, \hat{x},n)}J_f(\mu)(x_{-i})  M^{n-N(\varepsilon,\delta, \hat{x}, n)}\prod_{i\in R(\varepsilon,\delta, \hat{x},n)}J_f(\mu)(x_{-i})\\
	&=M^{n-N(\varepsilon,\delta, \hat{x}, n)}\cdot e^{N(\varepsilon,\delta, \hat{x}, n)\varepsilon}  \prod_{i=1}^nJ_f(\mu)(x_{-i}) = M^{n-N(\varepsilon,\delta, \hat{x}, n)} e^{N(\varepsilon,\delta, \hat{x}, n)\varepsilon}  J_{f^{n}}(\mu)(x_{-n}) \\
	&\leq M^{n-N(\varepsilon,\delta, \hat{x}, n)} e^{N(\varepsilon,\delta, \hat{x}, n)\varepsilon} e^{(F_f(\mu)+\varepsilon)n}.
	\end{align*}
	On the other hand since $M^{-1}J_f(z)\leq 1$ and $J_f(f^{n-i}z)\geq 1$ for every  $z\ \in X$ and for every $1 \leq i\leq n$, we have from (\ref{Kdelta}) and (\ref{jacfold}) that
	for $\hat{x}\in \widehat{X}(\hat{\mu}, \varepsilon)$ and $y\in B_{n}(x_{-n},\delta)$,
	\begin{align*}
	J_{f^{n}}(\mu)(y)&=\prod_{i\in A(\varepsilon,\delta, \hat{x}, n)}J_f(\mu)(f^{n-i}y)\prod_{i\in R(\varepsilon,\delta, \hat{x}, n)}J_f(\mu)(f^{n-i}y)\\
	&\geq   e^{-N(\varepsilon\delta, \hat{x}, n)\varepsilon}  \prod_{i\in A(\varepsilon,\delta, \hat{x},n)}J_f(\mu)(x_{-i})  M^{-(n-N(\varepsilon,\delta, \hat{x}, n)} \prod_{i\in R(\varepsilon,\delta, \hat{x},n)}J_f(\mu)(x_{-i})\\
	&=M^{N(\varepsilon,\delta, \hat{x}, n)-n} e^{-N(\varepsilon,\delta, \hat{x}, n)\varepsilon} \prod_{i=1}^n J_f(\mu)(x_{-i})\geq M^{N(\varepsilon,\delta, \hat{x}, n)-n} e^{-N(\varepsilon,\delta, \hat{x}, n)\varepsilon} e^{(F_f(\mu)-\varepsilon)n}.
	\end{align*}
	Since $\delta<\frac{\delta_0}{2}$, $f$ is injective on $B(x,\delta)$ for all $x\in X$; hence for all $n\ge 1$ and all  $x_{-n}\in X$ we obtain that $f^n$ is injective on $ B_{n}(x_{-n}, \delta)$. Thus  $\mu(f^n( B_{n}(x_{-n}, \delta) )=\int_{B_{n}(x_{-n}, \delta)}J_{f^n}(\mu)d\mu$. So for every $\hat{x}\in \widehat{X}(\hat{\mu},\varepsilon)$, for every $\delta<\delta(\varepsilon)<\frac{\delta_0}{2}$ and $n\geq n(\varepsilon,\hat{x})$ we obtain
	$$\varepsilon\log M +2\varepsilon +F_f(\mu)+ \frac{\log \mu(B_{n}(x_{-n}, \delta))}{n}\geq \frac{\log \mu(f^n( B_{n}(x_{-n}, \delta) )}{n}$$
	and   
	$$ -\varepsilon\log M -2\varepsilon +F_f(\mu)+ \frac{\log \mu(B_{n}(x_{-n}, \delta))} {n}\leq \frac{\log \mu(f^n( B_{n}(x_{-n}, \delta) )}{n}.$$
	As  $B_n^{-}(\hat{x},\delta)=f^n( B_{n}(x_{-n}, \delta)$, by Lemma \ref{mod} it follows that for every $\hat{x}\in\widehat{X}(\hat{\mu}, \varepsilon)$ we have
	\begin{align*}
	-\varepsilon\log M -2\varepsilon -F_f(\mu) + 
	h_f(\mu) & \leq 
	\lim_{\delta\rightarrow 0}\liminf_{n\rightarrow\infty}\frac{-\log \mu(B_n^-(\hat{x},\delta))}{n},\\
	\varepsilon\log M +2\varepsilon - F_f(\mu)
	+ h_f(\mu) &\geq  
	\lim_{\delta\rightarrow 0}\limsup_{n\rightarrow\infty}\frac{-\log \mu(B_n^{'-}(\hat{x},\delta))}{n}.
	\end{align*}
	Now for any $k\geq 1$, let $\varepsilon_k= 1/k $. Denote $\widehat{X}(\hat{\mu}, \varepsilon_k)$ by $\widehat{X}(\hat{\mu}, k)$. Then from (\ref{-1}) we obtain that $\widehat{X}(\hat{\mu}) = \bigcap_{k\geq 1} \widehat{X}(\hat{\mu}, k)$
	has full $\hat{\mu}$-measure, and for every $\hat{x}\in \widehat{X}(\hat{\mu})$ we have
	$$\lim_{\delta\rightarrow 0}\liminf_{n\rightarrow\infty}\frac{-\log \mu(B_n^-(\hat{x},\delta))}{n} =  
	\lim_{\delta\rightarrow 0}\limsup_{n\rightarrow\infty}\frac{-\log \mu(B_n^-(\hat{x},\delta))}{n}= h_f(\mu)- F_f(\mu).$$
	$\hfill\square$

\begin{proposition} 
	
	Let $M$ be a $C^2$ manifold, $X$ a compact subset of $M$ and $f: X\rightarrow X$ a $\mathcal{C}^2$ smooth map on a neighborhood of $X$. Let $\mu$ be an $f$-invariant and ergodic measure on $X$ such that all its Lyapunov exponents are strictly positive. Then $h^-_{f,B}(\mu)=0$. 
\end{proposition}
\begin{proof}
Since all the Lyapunov exponents of $\mu$ are positive, then there exists a measurable function on $\widehat{X}$, $\hat{x}\mapsto r(\hat{x})$ such that for $\hat{\mu}$-a.e. $\hat{x}\in \widehat{X}$ the local inverse of $f$ on $ B(x,r(\hat{x}))$ is defined and is denoted by $f^{-1} : B(x,r(\hat{x})) \rightarrow  B(x_{-1},r(\hat{f}^{-1}\hat{x}))$ and it is a contraction. Moreover we know that $ r(\hat{x})\leq  e^{\varepsilon} \cdot r(\hat{f}^{-1}\hat{x})$ for $\varepsilon$ small. 
Since the contraction is stronger than the subexponential growth of $r(\hat{x})$ on the inverse orbit, it follows that for $\hat{\mu}$-a.e. $\hat{x}\in \widehat{X}$ there exists $\delta(\hat{x})>0$ such that for any $0<\delta<\delta(\hat{x})$, $B(x,\delta)\subset B^-_n(\hat{x}, \delta)$ for $n$ sufficiently large. From this it  follows that $h^{-}_{f,B}(\mu)=0$.

\end{proof}


\subsection{Inverse entropy for measures with respect to partitions.}

Let $(X, \mathcal{B}, \mu)$ be a probability space and $f: X\rightarrow X$ be a measurable endomorphism such that $\mu$ is $f$-invariant and ergodic. 
Let $\mathcal{P}$ be a measurable partition of $X$, $\widehat{X}$  the inverse limit of $(X,f)$ and $\pi : \widehat{X} \rightarrow X$  the canonical projection. Then $\widehat{\mathcal{P}}= \{ \pi^{-1}(P) \ | P \in \mathcal{P} \}$
is a measurable partition of $\widehat{X}$. 
For $n\geq 1$ define 
$$\mathcal{P}_{n}= \bigvee_{i=0}^{n} {f}^{-i}(\mathcal{P}), \ \ \ \widehat{\mathcal{P}}_n^{\hat{f}^{-1}}= \bigvee_{i=0}^{n} \hat{f}^{i}(\widehat{\mathcal{P}}).$$
For $x\in X$ denote by $\mathcal{P}(x)$ (respectively $\mathcal{P}_{n}(x)$) the atom of $\mathcal{P}$ (respectively of $\mathcal{P}_{n}$) that contains $x$. For $\hat{x}=(x_{-i})_{i\geq 0}\in \widehat{X}$ with $x_0=x$, and $n \ge 1$, define the set
$$\mathcal{P}_{n}^{-}(\hat{x})=\{y\in X \ | \ \exists \  \hat{y}= (y_{-i})_{i\geq 0}\in \widehat X \text{ with }y_0=y \text{ and }y_{-i}\in \mathcal{P}(x_{-i}),  \text{ for }i=0,1,\ldots, n \}.$$

\begin{proposition}
	With the above notation we have
	$$\mathcal{P}_{n}^{-}(\hat{x})=f^{n}(\mathcal{P}_{n}(x_{-n})).$$
\end{proposition}

\begin{proof}
	Let $y\in \mathcal{P}_{n}^{-}(\hat{x})$. Hence there exists $\hat{y}=(y_{-i})_{i\geq 0} \in \widehat X$ a prehistory of $y$ with $y_0=y$, such that $y_{-i}\in \mathcal{P}(x_{-i})$ for $i=0,1,\ldots,n$. Thus $f^{i}(y_{-n})\in \mathcal{P}(f^i(x_{-n}))$ for $i=0,1,\ldots,n$. This shows that $y_{-n}\in P_n(x_{-n})$ and thus $y\in f^{n}(\mathcal{P}_{n}(x_{-n}))$. Now, if $y\in \mathcal{P}_{n}(x_{-n})$ then $f^{i}(y)\in \mathcal{P}(x_{-n+i})$ for $i=0,1,\ldots,n$. Let $\hat{z}\in \widehat{X}$ such that $z_{-i}=f^{n-i}(y)$, for $i=0,1,\ldots,n$. Then $f^n(y)= z_0 \in \mathcal{P}_{n}^{-}(\hat{x})$. 
\end{proof}

Let $\widehat{\mathcal{P}}_n^{-}(\hat{x})=\{ \hat{y}\in \widehat{X} \ | \hat{f}^{-i}(\hat{y}) \in  \widehat{\mathcal{P}} ( \hat{f}^{-i}(\hat{x})) \text{ for }i=0,1,\ldots, n \}=\{  \hat{y}\in \widehat{X}\ | \ y_{-i}\in \mathcal{P}(x_{-i}),  \text{ for }i=0,1,\ldots, n   \}.$ Then one infers the following:
\begin{proposition}\label{pr}
	With the above notation we have
	$$\hat{f}^{-n}(\widehat{\mathcal{P}}_{n}^{-}(\hat{x}))=\pi^{-1}(\mathcal{P}_{n}(x_{-n})).$$
\end{proposition}

\begin{theorem}[Modified Shannon-McMillan-Breiman Theorem]\label{sm}
	Let $\mathcal{P}$ be a measurable partition of $(X,\mu)$ such that $H_\mu(\mathcal{P})<\infty$. Then  for $\hat{\mu}$-a.e. $\hat{x}\in  \widehat{X}$,
	\begin{equation}\label{smb}
	\lim_{n\rightarrow\infty}\frac{-\log \mu(\mathcal{P}_{n}(x_{-n}))}{n}=h_f(\mu,\mathcal{P}).
	\end{equation}
\end{theorem}
\begin{proof}
From Proposition \ref{pr}, it follows that	
\begin{equation}\label{a1}
\hat{\mu}(\widehat{\mathcal{P}}_n^{-}(\hat{x}))=\hat{\mu}(\hat{f}^{-n}(\widehat{\mathcal{P}}_n^{-}(\hat{x})))=\mu((\mathcal{P}_{n}(x_{-n})).
\end{equation}If $H_\mu(\mathcal{P})<\infty$ then by Shannon-McMillan-Breiman Theorem we know that for $\hat{\mu}$-a.e. $\hat{x}\in\widehat{X}$,
\begin{equation}\label{a2}
\lim_{n\rightarrow\infty}\frac{-\log \hat{\mu}(\widehat{\mathcal{P}}_n^{-}(\hat{x}))}{n}=h_{\hat{f}^{-1}}(\hat{\mu},\widehat{\mathcal{P}}).
\end{equation} 
Since $h_{\hat{f}^{-1}} (\hat{\mu},\widehat{\mathcal{P}})=h_{\hat{f}}( \hat{\mu},\widehat{\mathcal{P}})=h_f(\mu,\mathcal{P}),$ the conclusion follows from (\ref{a1}) and (\ref{a2}).	
\end{proof}

Let $\mathcal{P}$ be measurable a partition of $X$. We recall that for $\hat{x}\in\widehat{X}$, $$h^-_{f, inf}(\mu,\mathcal{P},\hat{x})= \liminf_{n\rightarrow\infty}\frac{-\log \mu(\mathcal{P}_n^-(\hat{x}))}{n}, \ \  
h^-_{f,sup}(\mu, \mathcal{P}, \hat{x})= \limsup_{n\rightarrow\infty}\frac{-\log \mu(\mathcal{P}_n^-(\hat{x}))}{n}.$$
Furthermore, 
$$h^-_{f,inf}(\mu,\mathcal{P})= \int_{\widehat{X}}h^-_{ f,inf}(\mu,\mathcal{P}, \hat{x}) \ d \hat{\mu}(\hat{x}), \ \  h^-_{f,sup}(\mu,\mathcal{P})= \int_{\widehat{X}}h^-_{ f,sup}(\mu,\mathcal{P}, \hat{x})  \ d\hat{\mu}(\hat{x}).$$
 
The \textbf{lower, respectively the upper inverse partition entropy of $\mu$} are defined by
$$h^-_{f,inf}(\mu)=\sup\{ h^-_{f,inf}(\mu,\mathcal{P})\  : \mathcal{P} \text{ is a measurable partition with }  H_\mu(\mathcal{P})<\infty \}, $$
$$h^-_{f,sup}(\mu)=\sup\{ h^-_{f,sup}(\mu,\mathcal{P})\  : \mathcal{P} \text{ is a measurable partition with }  H_\mu(\mathcal{P})<\infty \}. $$
If $h^-_{f,inf}(\mu)=h^-_{f,sup}(\mu)$, the common value is called the \textbf{inverse partition entropy} of $\mu$ with respect to $f$ and is denoted by $h^-_{f}(\mu)$. In this case we say that $\mu$ has \textbf{inverse partition entropy}.

\begin{proposition}
	If $\mathcal{P}\leq \mathcal{Q}$ are measurable partitions with finite entropy with respect to $\mu$, then 
	$$h^-_{f,inf}(\mu,\mathcal{P})\leq h^-_{f,inf}(\mu,\mathcal{Q}) \text{ and } h^-_{f,sup}(\mu,\mathcal{P})\leq h^-_{f,sup}(\mu,\mathcal{Q}).$$	
\end{proposition}	

\begin{proposition}
	Let $(X,\mathcal{B},\mu)$ and $(Y,\mathcal{C},\nu)$ be two probability spaces and let $f: X\rightarrow X$ and $g: Y\rightarrow Y$ be two measurable endomorphisms such that $\mu$ is $f$-invariant and $\nu$ is $g$-invariant.
	\begin{itemize}
		\item [(i)] If $\phi: X\rightarrow Y$ is measurable such that $\phi \circ f=g\circ\phi$ $\mu$-a.e. and $\phi_*\mu=\nu$, then 
		$h^-_{f,inf}(\mu) \geq  h^-_{g, inf}(\nu)$ and $h^-_{f,sup}(\mu) \geq  h^-_{g, sup}(\nu).$
		\item[(ii)] If $\phi: X\rightarrow Y$ is an isomorphism of probability spaces such that $\phi\circ f=g\circ\phi$ $\mu$-a.e. and $\phi_*\mu=\nu$, then
		$h^-_{f,inf}(\mu) =  h^-_{g, inf}(\nu)$ and $h^-_{f,sup}(\mu) =  h^-_{g, sup}(\nu).$
		\item[(iii)] If $\phi: X\rightarrow Y$ is an isomorphism of probability spaces such that $\phi \circ f=g\circ\phi$ $\mu$-a.e., $\phi_*\mu=\nu$ and $f$ has inverse partition entropy, then $g$ has inverse partition entropy and $h^-_{f}(\mu) =  h^-_{g}(\nu).$
	\end{itemize}
\end{proposition}
\begin{proof} Let $\phi: X\rightarrow Y$ be a measure preserving map such that $\phi\circ f=g\circ\phi$ $\mu$-a.e. Notice that $\hat{\phi}:  \widehat{X}\rightarrow  \widehat{Y}$ defined by $\phi(\hat{x}) = \left(\phi({x_{-n}})\right)_{n\geq 0}$ is measure preserving from $(\widehat{X},\hat{\mu})$ to $(\widehat{Y},\hat{\nu})$.
	Let $ \mathcal{P}$ be a measurable partition of $Y$ of finite entropy. Then $\phi^{-1}(\mathcal{P})$ is a measurable partition of $(X, \mu)$ and it is easy to check that
	$h^-_{f,inf}(\mu,\phi^{-1}(\mathcal{P}),\hat{x})= h^-_{g,inf}(\nu,\mathcal{P},\hat{\phi}(\hat{x}))$. 
	Hence $$\int_{\widehat{X}} h^-_{f,inf}(\mu,\phi^{-1}(\mathcal{P}),\hat{x}) d\hat{\mu}(\hat{x})= \int_{\widehat{X}} h^-_{f,inf}(\nu,\mathcal{P},\hat{\phi}(\hat{x}))d\hat{\mu}(\hat{x})=\int_{\widehat{Y}} h^-_{g, inf}(\nu,\mathcal{P},\hat{y}))d\hat{\nu}(\hat{y})$$
	and then $h^-_{f,inf}(\mu) \geq  h^-_{g, inf}(\nu)$. Similarly  $h^-_{f,sup}(\mu) \geq  h^-_{g, sup}(\nu)$, thus proving (i). 
	Also (ii) and (iii) follow similarly.
\end{proof}	

\begin{definition}
	We say that a measurable partition $\mathcal{P}$ of $X$ is \textbf{normal} with respect to $(f, \mu)$,  if $f$ is injective on every atom $P\in \mathcal{P}$, $J_f(\mu)$ is bounded on every $P\in \mathcal{P}$, and $H_\mu(\mathcal{P})<\infty$.
\end{definition}

\begin{proposition}
	If $\mathcal{P}$ is a normal partition of $(X,\mu)$, then  for $\hat{\mu}$-a.e. $\hat{x}\in \widehat{X}$ we have
	\begin{itemize}
		\item [(i)] $h^-_{f,sup}(\mu, \mathcal{P})=h^-_{f,sup}(\mu, \mathcal{P}, \hat{x})\text{ and }h^-_{f,inf}(\mu, \mathcal{P})=h^-_{f,inf}(\mu, \mathcal{P}, \hat{x}).$
		\item [(ii)]  $h^-_{f,inf}(\mu, \mathcal{P}) \leq 	h^-_{f,sup}(\mu, \mathcal{P}) \leq h_{f}(\mu, \mathcal{P})$.
	\end{itemize}
	\end{proposition}

\begin{proof}	
	For $\hat{x}\in \widehat{X}$ we have 
	$$\mu(\mathcal{P}_{n}^-(\hat{f}(\hat{x})))\leq \mu( f(\mathcal{P}_{ n-1}^-(\hat{x})))= \int_{\mathcal{P}_{n-1}^-(\hat{x})}J_f(\mu) d\mu,$$
	because $f$ is injective on $\mathcal{P}_{ n-1}^-(\hat{x})$ as $\mathcal{P}_{n-1}^-(\hat{x}) \subset \mathcal{P}(x)$. Also, as the Jacobian $J_f(\mu)$ is bounded on $\mathcal{P}(x)$ it follows that 
	$$h^-_{f,sup}(\mu, \mathcal{P}, \hat{f}(\hat{x})) \geq h^-_{f,sup}(\mu, \mathcal{P}, \hat{x}).$$
	From the ergodicity of $\hat{\mu}$ it follows that we have equality almost everywhere. 
\end{proof}	

As any refinement with finite entropy of a normal partition is also a normal partition, we obtain: 
\begin{proposition}\label{np}
	If there exists a normal partition of $X$ with respect to $(f, \mu)$ then 
	$$h^-_{f,inf}(\mu)= \sup \{ h^-_{f,inf}(\mu, \mathcal{P}) :  \mathcal{P} \text{ normal partition}  \},$$
	$$h^-_{f,sup}(\mu)= \sup \{ h^-_{f,sup}(\mu, \mathcal{P}) :  \mathcal{P} \text{ normal partition}  \}.$$	
\end{proposition}

\begin{definition}
For $n\geq 1$, the projection $\pi_n :\widehat{X}\rightarrow X$ on the $n$-th coordinate is defined by $\pi_n(\hat{x})=x_{-n}$.	
\end{definition}

\begin{proposition}\label{a}
	Let $\mathcal{P}$ be a measurable partition of $(X,\mu)$ with finite entropy such that $f$ is injective on every $P\in \mathcal{P}$. Then, for $\hat{\mu}$-a.e. $\hat{x}\in\widehat{X}$ we have 
	\begin{equation}\label{c}
	h^-_{f,sup}(\mu, \mathcal{P}, \hat{x})	\leq  h_f(\mu, \mathcal{P}) -F_f(\mu),
	\end{equation}	
	and consequently, $h^-_{f,sup}(\mu,\mathcal{P})\leq  h_f(\mu)- F_{f}(\mu)$.
\end{proposition}
\begin{proof}
	Let $\delta>0$. For $k\geq 1$ define 
	$$\widehat{A}_{k,\delta}=\{\hat{x}=(x,x_{-1}, x_{-2},\ldots )\in \widehat{X} : \left| \frac{\log J_{f^n}(\mu)(x_{-n})}{n} - F_f(\mu)\right|<\delta  \text{ for all } n\geq k \}.$$
	Then, since $F_f(\mu)= \int_X \log J_{f}(\mu)d\mu$ and $\mu$ is ergodic (thus $\hat{\mu}$ ergodic), it follows by Birkhoff Ergodic Theorem applied to $\hat{f}$ on $(\widehat{X},\hat{\mu})$ and to $\log J_{f}(\mu)$, that
	$\lim_{k\rightarrow\infty} \hat{\mu}(\widehat{A}_{k,\delta})=1.$
	Let us fix $k\geq 1$ such that
	$\hat{\mu}(\widehat{A}_{k,\delta})>0$.
	For $n\geq k$, let the collection of sets from the partition $\mathcal{P}_n$
	$$\mathcal{K}_n^{\delta}=\left\{P\in \mathcal{P}_{n} : \mu(P\cap \pi_n(\widehat{A}_{k,\delta}))\geq \frac{1}{n^2}\cdot\mu(P) \right\},$$
	and define the measurable sets 
	$$D_n^{\delta}=\bigcup_{P\in \mathcal{K}_n^{\delta}}P  \ \text{and}  \ E_n^{\delta}=X\setminus D_n^{\delta}= \bigcup_{P\in \mathcal{P}_{n}\setminus \mathcal{K}_n^{\delta}}P.$$
	Notice that
	$$D_n^{\delta}=\left\{x\in X : \mu(\mathcal{P}_{n}(x)\cap \pi_n(\widehat{A}_{k,\delta}))\geq \frac{1}{n^2}\cdot\mu(\mathcal{P}_{n}(x))    \right\}.$$
	If $P\notin \mathcal{K}_n^{\delta}$ we have
	$\mu(P\cap \pi_n(\widehat{A}_{k,\delta}))<\frac{1}{n^2}\cdot\mu(P)$.
	Then $\mu(E_n^{\delta}\cap \pi_n(\widehat{A}_{k,\delta})) <\frac{1}{n^2}$.
		
	Since
	$$\pi_n^{-1}(E_n^{\delta}) \cap \widehat{A}_{k,\delta} \subseteq \pi_n^{-1}(E_n^{\delta} \cap \pi_n(\widehat{A}_{k,\delta})),$$
	we obtain that
	\begin{equation*}
	\hat{\mu}(\pi_n^{-1}(E_n^{\delta}) \cap \widehat{A}_{k,\delta}) \leq \hat{\mu}(\pi_n^{-1}(E_n^{\delta} \cap \pi_n(\widehat{A}_{k,\delta})))=\mu(E_n^{\delta}\cap \pi_n(\widehat{A}_{k,\delta}))<\frac{1}{n^2}.
	\end{equation*}
	If $\widehat{F}_n= \pi_n^{-1}(E_n^{\delta} \cap \pi_n(\widehat{A}_{k,\delta}))$,  then $\sum\limits_{n\geq k}\widehat{\mu}(\widehat{F}_n)<\infty$.  By Borel-Cantelli Lemma we have  
	$\hat{\mu}(\bigcap _{p\geq k} \bigcup_{n\geq p}\widehat{F}_n)=0$
	and therefore
	\begin{equation*}
	\hat{\mu}(\widehat{A}_{k,\delta})=\hat{\mu}((\widehat{X}\setminus\bigcap _{p\geq k} \bigcup_{n\geq p}\widehat{F}_n)\bigcap \widehat{A}_{k,\delta}) =\hat{\mu}((\bigcup\limits_{p\geq k} \bigcap_{n\geq p}\widehat{F}_n^c) \bigcap \widehat{A}_{k,\delta})=\hat{\mu}(\bigcup _{p\geq k} \bigcap_{n\geq p}(\widehat{A}_{k,\delta}\bigcap \widehat{F}_n^c)).
	\end{equation*}
	This shows that for $\hat{\mu}$-a.e $\hat{x}\in \widehat{A}_{k,\delta}$, there exists $p\geq k$ such that $\hat{x}\in \widehat{A}_{k,\delta}\bigcap \widehat{F}_n^c $ for all $n\geq p$. Notice that if $\hat{x}\in \widehat{A}_{k,\delta}\cap \widehat{F}_n^c$, then $\hat{x}\in \widehat{A}_{k,\delta}$ and $\pi_n(\hat{x})\in D_n^{\delta}$. Hence for $\hat{\mu}$-a.e $\hat{x}\in \widehat{A}_{k,\delta}$, there exists $p\geq k$ such that $\pi_n(\hat{x})\in D_n^{\delta}$ for all $n\geq p$. 
	We proved that 
	for $\hat{\mu}$-a.e. $\hat{x}\in \widehat{A}_{k,\delta}$ there exists $p\geq k$ such that
	\begin{equation}\label{above}
	\frac{1}{n^2}\cdot\mu(\mathcal{P}_{n}(x_{-n}))\leq \mu(\mathcal{P}_{n}(x_{-n})\cap \pi_n(\widehat{A}_{k,\delta})),  \text{ for all } n\geq p. 
	\end{equation}
	Let $\hat{x}\in \widehat{A}_{k,\delta}$ satisfying (\ref{above}) and the conclusion of Theorem \ref{sm}. For all $n\geq k$, since $f^n$ is injective on $\mathcal{P}_n(x_{-n})$, we have from (\ref{above}) that, 
	\begin{align*}
	\mu(\mathcal{P}_{n}^{-}(\hat{x}))&=\mu (f^n(\mathcal{P}_{n}(x_{-n})))= \int_{\mathcal{P}_{n}(x_{-n})}J_{f^n}(\mu)(y)\ d\mu(y)\\ &\geq \int_{\mathcal{P}_{n}(x_{-n}) \cap \pi_n(\widehat{A}_{k,\delta})}J_{f^n}(\mu)(y)\ d\mu(y)\geq e^{n(F_f(\mu)-\delta )}\cdot \mu (\mathcal{P}_{n}(x_{-n}) \cap \pi_n(\widehat{A}_{k,\delta}))\\ & \geq \frac{1}{n^2}\cdot e^{n(F_f(\mu)-\delta )}\cdot \mu (\mathcal{P}_{n}(x_{-n})),
	\end{align*}
	and then
	$$\frac{-\log \mu(\mathcal{P}_{n}^{-}(\hat{x}))}{n}\leq \frac{-\log \mu (\mathcal{P}_{n}(x_{-n}))}{n} +\frac{2\log n}{n}-F_f(\mu)+\delta.$$
	Hence, by using Theorem \ref{sm}, for $\hat{\mu}$-a.e. $\hat{x}\in \widehat{A}_{k,\delta}$ we obtain 
	\begin{equation}\label{b}
	\limsup_{n\rightarrow\infty}\frac{-\log \mu(\mathcal{P}_{n}^{-}(\hat{x}))}{n} \leq h_f(\mu, \mathcal{P}) -F_f(\mu)+\delta.
	\end{equation}
	Since the sequence of sets $(\widehat{A}_{k,\delta})_{k\geq 1}$ is increasing and $\hat{\mu}(\widehat{X}\setminus\bigcup\limits_{k\geq 1} \widehat{A}_{k,\delta})=0$, it follows that (\ref{b}) is satisfied for $\hat{\mu}$-a.e. $\hat{x}\in \widehat{X}$. But as $\delta>0$ was chosen arbitrarily, we conclude that  
	\begin{equation*}
	\limsup_{n\rightarrow\infty}\frac{-\log \mu(\mathcal{P}_{n}^{-}(\hat{x}))}{n} \leq  h_f(\mu, \mathcal{P}) -F_f(\mu), \ \text{ for }\hat{\mu}-\text{a.e. }\hat{x}\in \widehat{X}.
	\end{equation*}\end{proof}
\begin{corollary}\label{cor}
	If there exists a measurable partition $\mathcal{A}$ of $(X,\mu)$ with finite entropy such that $f$ is injective on every atom $A\in \mathcal{A}$, then	
	$$ h^-_{f,sup}(\mu)\leq  h_f(\mu)- F_{f}(\mu).$$
\end{corollary}
\begin{proof}
	Let $\mathcal{P}$ be a partition with finite entropy. Then the join partition $\mathcal{P} \vee\mathcal{A}$ has finite entropy. Also $f$ is injective on every atom of this partition. 
	The corollary then follows immediately from Proposition \ref{a} and the fact that 
	$(\mathcal{P}\vee{\mathcal{A}})^{-}_n(\hat{x}) \subset \mathcal{P}^{-}_n(\hat{x})$ for every $\hat{x}\in\widehat{X}$. 
\end{proof}

\begin{proposition}\label{aa}
	Assume that there exists  a finite partition $\mathcal{A}$ of $(X,\mu)$ such that $f$ is injective on every $A\in \mathcal{A}$. 
	Then: 
	\begin{itemize}
		\item [(i)] there exists a sequence $\mathcal{C}^1 \leq \mathcal{C}^2 \leq \mathcal{C}^3 \leq \ldots$ of normal partitions such that whenever $\mathcal{P}$ is an arbitrary partition with finite entropy and $\mathcal{P}_{(k)}:= \mathcal{P}\vee \mathcal{C}^k$, we have for $\hat{\mu}$-a.e $\hat{x}\in \widehat{X}$, 
		$$h^-_{f,inf}(\mu, \mathcal{P}_{(k)}, \hat{x})\geq  h_f(\mu, \mathcal{P}_{(k)})- F_f(\mu) -\frac{1}{2^{k-1}}.$$
		\item [(ii)] $\ h^-_{f,inf}(\mu)\geq  h_f(\mu)- F_{f}(\mu).$
	\end{itemize}
\end{proposition}
\begin{proof} (i) Consider $k\geq 1$ and define the countable measurable partition $\alpha^k$ of $X$ with elements
	\begin{equation}\label{alphaik}
	\alpha_i^k=J_f(\mu)^{-1}\left(\left[\frac{i^2}{4^k}, \frac{(i+1)^2}{4^k} \right) \right),  i\geq 2^k.
	\end{equation}
	Notice that for any $i\geq k$ and $A\in \mathcal{A}$, since $f$ is injective on every atom $A\in\mathcal{A}$ and $\frac{i^2}{4^k}\cdot \mu (\alpha_i^k\cap A)\leq \mu (f(\alpha_i^k\cap A))\leq  1$, we have
	$\mu(\alpha_i^k\cap A )\leq \frac{4^k}{i^2}.$ 
	Define the countable measurable partition of $X$
	$$\mathcal{C}^k:=\{\alpha_{i}^k \cap A \ | \ \mu (\alpha_i^k\cap A)>0, A\in \mathcal{A}, i\geq 2^k \}.$$ 
 Let $N$ be the number of elements of $\mathcal{A}$. Since
	$\sum_i -\frac{4^k}{i^2}\log \frac{4^k}{i^2}<\infty$ and the function $x\mapsto -x\log x$ is increasing on $(0,1/e)$ we have 
	$$-\sum_{A\in \mathcal{A}}\mu(\alpha_i^k\cap A )\log \mu(\alpha_i^k  \cap A)< -\sum_{i\geq 2^k}\frac{N\cdot  4^k}{i^2} \log \frac{ 4^k}{i^2}<\infty,$$ and then $H_\mu(\mathcal{C}^k)<\infty$.
	Thus since $f$ is injective on every atom of $\mathcal{A}$ it follows that $\mathcal{C}^k$ is a normal partition.
	Define the function $G: X \rightarrow \mathbb{R}$ by
	$$G(x) =\sum_{i\geq 2^k}\frac{(i+1)^2}{4^k}\cdot\chi_{\alpha^k_i}(x), \ \  x\in X.$$
	Then $G$ is a measurable function and $\log G$ is integrable with respect to $\mu$. As 
	$$\log \frac{i^2}{4^k}\leq \log J_f(\mu) (x)<\log \frac{(i+1)^2}{4^k} \text{ for all }x\in \alpha^k_i,$$
	and as $\log \frac{(i+1)^2}{4^k}-\log \frac{i^2}{4^k}=2\log \left( 1+\frac{1}{i} \right)\leq \frac{2}{i}$
	for every $i\geq 2^k$, we have from above and (\ref{alphaik}) that
	\begin{equation}\label{GJ}
	0\leq\log G(x) - \log J_f(\mu)(x) <\frac{2}{2^k},  \text{ for all }x\in X.
	\end{equation}
	Therefore
	\begin{equation}\label{int}
	0\leq \int_X (\log G - \log J_f(\mu)) \ d\mu \leq \frac{2}{2^k}.
	\end{equation}
	Let $\mathcal{P}$ be an arbitrary finite or countable measurable partition of $X$ with $H_\mu(\mathcal{P})<\infty$. Consider the join partition $\mathcal{P}_{(k)}:=\mathcal{P}\vee \mathcal{C}^{k}$. Now, by Birkhoff Ergodic Theorem we get 
	\begin{equation}\label{x1}
	\frac{1}{n} \sum_{i=1}^n \log G(x_{-i}) \underset{n\rightarrow\infty}\longrightarrow \int_X \log G(x) d\mu(x), \text{ for }\hat{\mu}-\text{a.e. }\hat{x}\in \widehat{X}.
	\end{equation}
	Define $\widehat{X}(\hat{\mu}, \mathcal{P}_{(k)})$ to be the set of all $\hat{x}\in \widehat{X}$ that satisfy (\ref{smb}) for $\mathcal{P}_{(k)}$ and also (\ref{x1}). Then $\hat{\mu}(\widehat{X}(\hat{\mu}, \mathcal{P}_{(k)}))=1$. If $y\in \mathcal{P}_{(k), n}(x_{-n})$ then $$f^{n-j}(y)\in  \mathcal{P}_{(k)}\left(f^{n-j} (x_{-n})\right)= \mathcal{P}_{(k)}\left(x_{-j}\right)\subset \mathcal{C}^k(x_{-j}),  \ \text{ for }j=1,2,\ldots, n.$$
	From (\ref{GJ}), $G$ is constant on any $\alpha_i^k$ and $J_f(\mu)(z)\leq G(z)$, for every $z\in \alpha_i^k$. 	
	Notice that $G(z)= G(x_{-j})  \text{ whenever }z\in  f^{n-j}(\mathcal{P}_{(k), n}(x_{-n})).$
	Thus by successive integration we obtain
	\begin{align*}
	\mu\left(f^n (\mathcal{P}_{(k),n}(x_{-n})) \right)&=\int_{f^{n-1} (\mathcal{P}_{(k),n}(x_{-n})) } J_f(\mu)(z)d\mu(z) \leq G(x_{-1})\cdot\mu\left(f^{n-1} (\mathcal{P}_{(k),n}(x_{-n})) \right)\\
	&\leq  G(x_{-1})\cdot G(x_{-2})\ldots G(x_{-n})\cdot \mu\left( \mathcal{P}_{(k). n}(x_{-n})\right).
	\end{align*}
	Thus for every $k\ge 1$ we have
	\begin{align*}
	\frac{-\log \mu(f^n (\mathcal{P}_{(k),n}(x_{-n})) )}{n}&\geq -\frac{1}{n }\sum_{i=1}^n \log G(x_{-i})- \frac{\log \mu( \mathcal{P}_{(k),n}(x_{-n}) )}{n}.
	\end{align*}
	But for every $\hat{x}\in \widehat{X}(\hat{\mu}, \mathcal{P}_{(k)})$, from Theorem \ref{sm}, we get
	$$\lim_{n\rightarrow\infty}\frac{-\log \mu( \mathcal{P}_{(k),n}(x_{-n}) )}{n} = h_f(\mu, \mathcal{P}_{(k)}).$$ 
	Then, by applying (\ref{int}) and (\ref{x1}), we obtain: 	
	\begin{align}\label{x}
	h^-_{f,inf}(\mu, \mathcal{P}_{(k)}, \hat{x})&=\liminf_{n\rightarrow\infty} \frac{-\log \mu(f^n (\mathcal{P}_{(k),n}(x_{-n})) )}{n} \geq h_{f}(\mu, \mathcal{P}_{(k)})- \int_X \log G \ d\mu\nonumber \\
	&\geq  h_f(\mu, \mathcal{P}_{(k)})- \int_X \log J_f(\mu) \ d\mu-\frac{1}{2^{k-1}} = h_f(\mu, \mathcal{P}_{(k)})- F_f(\mu) -\frac{1}{2^{k-1}}\nonumber, 
	\end{align}
	for every $\hat{x}\in \widehat{X}(\hat{\mu}, \mathcal{P}_{(k)})$. 
	Then by integration, we get: 
	$$h^-_{f,inf}(\mu, \mathcal{P}_{(k)})\geq  h_f(\mu, \mathcal{P}_{(k)})- F_f(\mu) -\frac{1}{2^{k-1}}.$$	
	(ii) Let $\{\mathcal{P}^{(m)}, m \ge 1\}$ be a sequence of measurable partitions of finite entropy such that $$
	\lim_{m\rightarrow\infty}h_f(\mu,\mathcal{P}^{(m)})=\sup_m h_f(\mu, \mathcal{P}^{(m)})= h_f(\mu).$$
	For $k,m\geq 1$, let $\mathcal{P}^{(m)}_{(k)} = \mathcal{P}^{(m)}\vee \mathcal{C}^k$. Since $\mathcal{P}^{(m)}_{(k)}$ is a refinement of $\mathcal{P}^{(m)}$, for all $k,m$ we have 
	$$h^-_{f,inf}(\mu, \mathcal{P}^{(m)}_{(k)})\geq  h_f(\mu, \mathcal{P}^{(m)}_{(k)})- F_f(\mu) -\frac{1}{2^{k-1}}\geq  h_f(\mu,\mathcal{P}^{(m)})- F_f(\mu) -\frac{1}{2^{k-1}}.$$
	Thus for every $k\geq 1$, we have
	$$h^-_{f,inf}(\mu)\geq \sup_{m,k} h^-_{f,inf}(f, \mathcal{P}^{(m)}_{(k)})\geq \sup_{m} h_f(\mu,\mathcal{P}^{(m)})- F_f(\mu) -\frac{1}{2^{k-1}}.$$ 
	Therefore
	$h^-_{f,inf}(\mu)\geq  h_f(\mu)- F_f(\mu).$
\end{proof}

Now \textbf{Theorem \ref{t1}} follows immediately from Corollary \ref{cor} and Proposition \ref{aa}.  Thus if there exists a finite partition $\mathcal{A}$ such that $f$ is injective on every atom of $\mathcal{A}$, then $\mu$ has inverse partition entropy and 
	$$h^-_{f}(\mu)=h_f(\mu)- F_f(\mu).$$
\begin{corollary}\label{corz}
Let $f: M\rightarrow M$ be a $C^r$, $r>1$ endomorphism on  a compact Riemannian manifold $M$ with no critical points and let $\mu$ be an ergodic $f$-invariant measure. Then $h^-_f(\mu)$  exists and 
$$h^-_f(\mu)\leq  - \sum_{i: \lambda_i(\mu)<0}\lambda_i(\mu),$$
where $\lambda_i(\mu)$ are the Lyapunov exponents of $\mu$ taken with their multiplicities. 
\end{corollary}
\begin{proof}
From \cite{LW} we know that 
$h_f(\mu)\leq F_f(\mu) - \sum_{i: \lambda_i(\mu)<0}\lambda_i(\mu).$
This inequality proved in \cite{LW} by Liao and Wang for $C^r$ maps, $r>1$ was conjectured  in \cite{R}. Since $f$ has no critical points, there exists a partition $\mathcal{A}$ such that $f$ is injective on each atom of $\mathcal{A}$, and then by Theorem \ref{t1}, it follows that $\mu$ has inverse partition entropy and 
$$h^-_f(\mu)=h_f(\mu) -F_f(\mu)\leq  - \sum_{i: \lambda_i(\mu)<0}\lambda_i(\mu).$$
\end{proof}


We have seen in the proof of Proposition \ref{aa} that if there exists a finite partition $\mathcal{A}$ such that $f$ is injective on every atom of $\mathcal{A}$, then there exists certain normal partitions. Also from Proposition \ref{np} and Theorem \ref{t1}, $h^-_f(\mu)$ can be computed using normal partitions. 

\begin{proposition}
Let $(X,\mathcal{B},\mu)$ and $(Y,\mathcal{C},\nu)$ be two probability spaces and let $f: X\rightarrow X$ and $g: Y\rightarrow Y$ be two measurable endomorphisms such that $\mu$ is ergodic with respect to $f$ and $\nu$ is ergodic with respect to $g$ and satisfy the conditions from Theorem \ref{t1}.
	\begin{itemize}
		\item [(i)] If $\mu$ is ergodic with respect to $f^k$ for some $k \ge 1$, then $\mu$ has inverse partition entropy with respect to $f$ and $f^k$, and $h^{-}_{f^k}(\mu)= k h^{-}_f(\mu)$.
		\item 	[(ii)] If $\mu\times \nu$ is ergodic, then $\mu\times \nu$ has inverse partition entropy, and $$h^{-}_{f\times g}(\mu\times \nu)= h^{-}_f(\mu) + h^{-}_g(\nu).$$
	\end{itemize}	
\end{proposition}
\begin{proof}
	For $(i)$, notice that since $\mu$ is $f$-invariant, we have $F_{f^k}(\mu)= \int_X \log J_{f^k}d\mu = k \cdot\int_X J_f(\mu)d\mu$. As $h_{f^k}(\mu)= k \cdot h_f(\mu)$ it follows from Theorem \ref{t1} that $h^{-}_{f^k}(\mu)= k\cdot h^{-}_f(\mu)$.     	
	For $(ii)$ we have $h_{f\times g}(\mu \times \nu)= h_f(\mu) + h_{g}(\nu)$ and it is easy to see that $F_{f\times g}(\mu\times\nu)=  F_f(\mu) + F_{g}(\nu)$. We then apply again Theorem \ref{t1} to conclude the proof.
\end{proof}

\begin{proposition}
	Let $(X, \mathcal{B}, \mu)$ be a probability space and  $f: X\rightarrow X$ be a measurable endomorphism as in Theorem \ref{t1} (hence $\mu$ has inverse partition entropy). If $\mathcal{P}_1 \leq  \mathcal{P}_2 \leq  \mathcal{P}_3\leq  \ldots $ is a sequence of partitions with finite entropy such that $\mathcal{B}$ is the $\sigma$-algebra generated by $\bigcup_{n\geq 1} P_n$, then there exists a sequence $\mathcal{P}_1' \leq  \mathcal{P}_2'\leq \mathcal{P}_3'\leq\ldots$ of normal partitions of $X$ with $\mathcal{P}_n\leq  \mathcal{P}_n'$ such that 
	\begin{equation*}			
	h^-_{f}(\mu) = \lim_{n\rightarrow\infty} h^-_{f,inf}(\mu, \mathcal{P}_n') =\lim_{n\rightarrow\infty} h^-_{f,sup}(\mu,\mathcal{P}_n'). 
	\end{equation*}
\end{proposition}
\begin{proof}
	As $\mathcal{P}_1 \leq  \mathcal{P}_2\leq \mathcal{P}_3\leq\ldots $ and $\mathcal{B}$ is the $\sigma$-algebra generated by $\bigcup_{n\geq 1} P_n$, we have 	
	$h_{f}(\mu) = \sup_n   h_{f}(\mu, \mathcal{P}_n)$. By Proposition \ref{aa} there exist normal partitions $\mathcal{C}^1 \leq \mathcal{C}^1 \leq \mathcal{C}^3 \leq \ldots $ such that if $\mathcal{P}_k':=\mathcal{P}_k \vee \mathcal{C}^k$, then for every $k \geq 1$, $h^-_{f,inf}(\mu, \mathcal{P}_k')\geq  h_f(\mu, \mathcal{P}_k')-  F_f(\mu) -\frac{1}{2^{k-1}}.$
	By Corollary \ref{cor} we have $h_f(\mu, \mathcal{P}_k')-  F_f(\mu) \geq h^-_{f,sup}(\mu, \mathcal{P}_k'), \forall k \ge 1$.
	Hence $h_f(\mu, \mathcal{P}_k')-F_f(\mu)\geq h^-_{f,sup}(\mu, \mathcal{P}_k')\geq h^-_{f,inf}(\mu, \mathcal{P}_k') \geq  h_f(\mu, \mathcal{P}_k')-  F_f(\mu) -\frac{1}{2^{k-1}}$ and so $\lim\limits_{n\rightarrow\infty} h^-_{f,inf}(\mu, \mathcal{P}_n') =\lim\limits_{n\rightarrow\infty} h^-_{f,sup}(\mu,\mathcal{P'}_n)=h^-_f(\mu)$. 
	\end{proof}
	
	We now give a class of ergodic measures for which the inverse entropy can be computed. Let $f: M\rightarrow M$ be a smooth $C^2$ map on a Riemannian manifold $M$ and let $\Lambda$ be a compact set which is $f$-invariant and such that $f$ is topologically transitive. We assume that  $\Lambda$ is a repellor;  by this we mean that there exists a neighborhood $U$ of $\Lambda$ such that $\Lambda=\bigcap_{n\in\mathbb{N}}f^{-n}(U)$ and $\overline{U}\subset f(U)$. If $\Lambda$ is connected and $f$ does not have critical points in $\Lambda$, then $Card (f^{-1}(x)\cap \Lambda)$ does not depend on $x\in \Lambda$ and is equal to some integer $d\geq 1$. There exists a neighbourhood $V$ of $\Lambda$ which is close enough to $\Lambda$ such that
	any point $y\in V$ has exactly $d^n$ $n$-preimages belonging to $U$, for $n\geq 1$ (see \cite{M4}).
	Then for any $z\in V\subset U$ one can consider the discrete  measures
	$$\mu_n^z=\frac{1}{d^n}\sum_{y\in f^{-n}z \cap U}\frac{1}{n}\sum_{i=0}^{n-1}\delta_{f^iy}, n\geq 1.$$
	It was proved in \cite{M4} that there exists a subset $A\subset V$, having full Lebesgue  measure in $V$ and a subsequence $(\mu^z_{n_k})_{k}$ that converges weakly to a unique measure $\mu^-$ for every $z\in A$, and this measure $\mu^-$ is called the \textbf{inverse SRB measure}. It was shown in \cite{M4} that $\mu^-:=\mu_s$, where $\mu_s$ is the equilibrium measure of the stable potential $\Phi^s(x)= \log |Df|_{E^{s}_x}|$, $x\in\Lambda$. Then a Pesin type formula involving the negative Lyapunov
	exponents can be derived for the measure $\mu^-$, namely: 
	\begin{theorem}\cite[Theorem 3]{M4}.
		Let $\Lambda$ be a connected hyperbolic repellor for a $C^2$ endomorphism $f: M\rightarrow M$ on a Riemannian manifold $M$; assume that $f$ is $d$-to-$1$ on $\Lambda$ and does not have critical points in $\Lambda$. Then there exists a unique
		$f$-invariant probability measure $\mu^-$ on $\Lambda$ satisfying an inverse Pesin entropy formula:
		$$h_f(\mu^-)=F_f(\mu^-) -\int_\Lambda\sum_{i: \lambda_i(x)<0}\lambda_i(x)d\mu^-(x)=\log d -\int_\Lambda\sum_{i: \lambda_i(x)<0}\lambda_i(x)d\mu^-(x),$$
		where the Lyapunov exponents $\lambda_i(x)$ are taken with their multiplicities. In addition the measure $\mu^-$ has absolutely continuous conditional measures on local stable
		manifolds.	
	\end{theorem}	
	\begin{proposition}\label{propz}
		Let $\Lambda$ be a connected hyperbolic repellor for a $C^2$ endomorphism $f: M\rightarrow M$ on a Riemannian manifold $M$. Assume that $f$ is $d$-to-$1$ on $\Lambda$ and $f$ does not have critical points in $\Lambda$, and let $\mu^-$ be the inverse SRB measure of $f$ on $\Lambda$. Then $h^-_f(\mu^-)$ and $h^-_{f,B}(\mu^-)$ exists and $$h^-_f(\mu^-)=h^-_{f,B}(\mu^-)=-\sum_{i: \lambda_i(\mu^-)<0}\lambda_i(\mu^-),$$
		where the Lyapunov exponents $\lambda_i(\mu^-)$ are taken with their multiplicities. 
	\end{proposition}
	\begin{proof}
		From the proof of the above theorem, we know that the Jacobian $J_f(\mu^-)(x)=d$ for $\mu^-$-a.e. $x\in\Lambda$ and from this it easily follows that $h^-_f(\mu^-)$ and $h^-_{f,B}(\mu^-)$ exist and $$h^-_f(\mu^-)=h^-_{f,B}(\mu^-)=-\sum_{i: \lambda_i(\mu^-)<0}\lambda_i(\mu^-),$$
		where $\lambda_i(\mu^-)$ are taken with their multiplicities. 
	\end{proof}

Now we study the relations between the inverse metric entropy (defined using inverse Bowen balls) and the inverse partition entropy. First we obtain an inequality in Proposition \ref{c1}. Next we study conditions when the inverse metric entropy is equal to the inverse partition entropy. Under these conditions we prove in Theorem \ref{zerob} that both the inverse metric entropy and the inverse partition entropy of an ergodic measure $\mu$ are equal to $h_{f}(\mu)- F_f(\mu)$.

\begin{proposition}\label{c1}
	Let $f$ be a continuous locally injective transformation of a compact metric space $X$ and $\mu$ be a probability measure $f$-invariant on $X$ which is ergodic. Then for $\hat{\mu}$-a.e. $\hat{x}\in\widehat{X}$ we have 
	$$h^-_{f,sup, B}(\mu)=\lim_{\varepsilon\rightarrow 0}\limsup_{n\rightarrow\infty} \frac{-\log \mu(B^-_{n}(\hat{x},\varepsilon))}{n}\leq h^-_{f,sup}(\mu) = h_f^-(\mu) =  h_{f}(\mu)- F_f(\mu).$$
\end{proposition}
\begin{proof}
Let a finite measurable partition $\mathcal{P}$ so that $\text{diam}(P)< \varepsilon$ and $f$ is injective on every $P\in \mathcal{P}$. Thus $\mathcal{P}_{n+1}^-(\hat{x})\subset B_n^-(\hat{x},\varepsilon)$, $\forall \hat{x}\in \widehat{X}$. Then by Theorem \ref{t1}, it follows that for $\hat{\mu}$-a.e. $\hat{x}\in \widehat{X}$, 
	$$\limsup_{n\rightarrow\infty}\frac{-\log \mu(B_n^-(\hat{x},\varepsilon)) }{n}\leq \limsup_{n\rightarrow\infty}\frac{-\log \mu(\mathcal{P}_{n+1}^-(\hat{x})) }{n} \leq h_f^-(\mu) = h_f(\mu)- F_f(\mu).$$
\end{proof}

\begin{definition}\label{defzerob}
	Let $f:X\rightarrow X$ be a continuous and locally injective transformation of the compact metric space $X$ and $\mu$ be a probability measure on $X$ which is $f$-invariant. We say that the measure $\mu$ satisfies the \textbf{zero boundary property} if for every $\varepsilon'>0$ there exists a finite measurable partition $\mathcal{P}$ such that $\mu(\partial \mathcal{P})=0$ and such that for $\mu$-a.e $x\in X$,
	$$e^{-\varepsilon'} <\left| \frac{J_f(\mu)(x)}{J_f(\mu)(y)} \right | < e^{\varepsilon'}, \ \text{for }\mu\text{-a.e. }y\in\mathcal{P}(x).$$
\end{definition}	

\textbf{Proof of Theorem \ref{zerob}.}
Let $\varepsilon'>0$ arbitrary. As $\mu$ satisfies the zero boundary condition, there exists a finite measurable partition $\mathcal{P}$ which depends on $\varepsilon'$ such that $\mu(\partial \mathcal{P})=0$ 
	and for $\mu$-a.e $x\in X$ 
	\begin{equation}\label{b0}
e^{-\varepsilon'} <\left| \frac{J_f(\mu)(x)}{J_f(\mu)(y)} \right | < e^{\varepsilon'}, \text{ for } \mu\text{-a.e } y\in\mathcal{P}(x).
	\end{equation}
	We can also assume that  $h_f(\mu, \mathcal{P})> h_f(\mu) -\varepsilon'$.
	For $\delta>0$ let 
	$$W_\delta(\mathcal{P})=\{x\in X :  B(x,\delta) \not\subset \mathcal{P}(x) \}.$$
	As $\bigcap\limits_{\delta>0} W_\delta(\mathcal{P})= \partial \mathcal{P}$ and $\mu(\partial \mathcal{P})=0$, it follows that $\mu (W_\delta(\mathcal{P}) )\rightarrow 0$ as $\delta \rightarrow 0$. Let $\varepsilon>0$ arbitrary. Then there exists $\delta_0(\varepsilon)>0$ such that $\mu(W_{\delta}(\mathcal{P})) <\varepsilon$ for any $0<\delta<\delta_0(\varepsilon)$. Let $0<\delta<\delta_0(\varepsilon)$ arbitrary.
	Denote by $N=N(\varepsilon')$ the number of elements of the above partition $\mathcal{P}$. Recall that $\hat{f}:\widehat{X}\rightarrow \widehat{X}$ is a homeomorphism which preserves the lift measure $\hat{\mu}$ and $\pi:\widehat{X}\rightarrow X$ is the canonical projection. By Birkhoff's Ergodic Theorem applied to $\hat{f}$ and  $\chi_{\pi^{-1}(W_\delta(\mathcal{P}))}$, we have $\frac{1}{n}\sum\limits_{i=0}^{n} \chi_{\pi^{-1}(W_\delta(\mathcal{P}))}(\hat{f}^{-i}(\hat{x})) \mathop{\longrightarrow}\limits_{n\rightarrow\infty}  \hat{\mu}(\pi^{-1}(W_\delta(\mathcal{P})))$ for $\hat{\mu}$-a.e $\hat{x}\in \widehat{X}$, and then for $\hat{\mu}$-a.e $\hat{x}\in \widehat{X}$,
	\begin{equation*}
	\frac{1}{n}\sum_{i=0}^{n} \chi_{W_\delta(\mathcal{P})}(x_{-i}) \mathop{\longrightarrow}\limits_{n\rightarrow\infty}  \mu(W_\delta(\mathcal{P})). 
	\end{equation*}
	Thus as $\mu(W_\delta(\mathcal{P})) <\varepsilon$, it follows that for $\hat{\mu}$-a.e $\hat{x}\in \widehat{X}$ there exists $n(\hat{x}, \varepsilon, \varepsilon')\geq 1$ such that
	\begin{equation}\label{b1} 
	\frac{1}{n}\sum_{i=0}^{n} \chi_{W_\delta(\mathcal{P})}(x_{-i}) <\varepsilon, \text{ for all }n\geq n(\hat{x},\varepsilon, \varepsilon').
	\end{equation}
	Now recall that for $n\geq 1$, $\mathcal{P}_n=\bigvee_{i=0}^{n}f^{-i}(\mathcal{P})$. By Theorem \ref{sm}, for $\hat{\mu}$-a.e. $\hat{x}\in \widehat{X}$ there exists $n'(\hat{x},\varepsilon, \varepsilon')\geq 1$ such that
	\begin{equation}\label{e2}
	\frac{-\log \mu(\mathcal{P}_{n}(x_{-n}))}{n}> h_f(\mu, \mathcal{P})-\varepsilon, \text{ for all }n\geq n'(\hat{x},\varepsilon, \varepsilon').
	\end{equation}
	Also since $F_f(\mu)=\int_X \log J_f(\mu)d\mu$, by Birkhoff Ergodic Theorem we obtain that for $\hat{\mu}$-a.e $\hat{x}\in \widehat{X}$ there exists $n''(\hat{x},  \varepsilon)\geq 1$ such that 
	\begin{equation}\label{e3}
	\left|\frac{1}{n}\log J_{f^{n}}(x_{-n}) -F_f(\mu) \right|<\varepsilon, \text{ for all }n\geq n''(\hat{x}, \varepsilon).
	\end{equation}
	For $\hat{x}\in \widehat{X}$ define the $\left(\mathcal{P}, n\right)$-name of $\hat{x}$ as $(\mathcal{P}(x) ,\mathcal{P}(x_{-1}),\ldots, \mathcal{P}(x_{-n}))$. Let $n\geq n(x,\varepsilon, \varepsilon')$. If $\hat{x}, \hat{y}\in\widehat{X}$ then the Hamming distance (see \cite{B}) between the   
	$(\mathcal{P}, n)$-name of $\hat{x}$ and the   
	$(\mathcal{P}, n)$-name of $\hat{y}$ is 
	$$\frac{1}{n+1} \cdot\text{Card}\{0\leq i\leq n :  \mathcal{P}(x_{-i})\neq \mathcal{P}(y_{-i})   \}.$$
	The $(\mathcal{P}, n)$-name of $\hat{x}$ can be interpreted as being the (forward) $(\mathcal{P}, n)$-name of $x_{-n}$. Notice that by (\ref{b1}), $B^{-}_{n}(\hat{x},\delta)$ is contained in the set of all $y\in X$ with the property that there exists $\hat{y}\in\widehat{X}$ such that the $(\mathcal{P},n)$-name of $\hat{y}$ is $\varepsilon$-close in the Hamming distance to the $(\mathcal{P},n)$-name of $\hat{x}$. But if $V_n$ denotes the number of $(\mathcal{P},n)$-names which are $\varepsilon$-close to the $(\mathcal{P},n)$-name of $x_{-n}$, then from \cite{B},
	\begin{equation}\label{b10}
	\lim_{n\rightarrow\infty}\frac{\log V_n}{n} = \varepsilon \log(N-1)-\varepsilon \log \varepsilon -(1-\varepsilon)\log (1-\varepsilon),
	\end{equation}
	where recall that $N=\text{card}(\mathcal{P})>1$ depends only on $\varepsilon'$. Hence there exists $N(\varepsilon, \varepsilon')$ such that for every $n\geq N(\varepsilon, \varepsilon')$, we have
	\begin{equation}\label{Vn}
	V_n < e^{C(\varepsilon, \varepsilon')n}, \ \text{where}
	\end{equation} 
	\begin{equation}\label{b8}
	 C(\varepsilon, \varepsilon')=\varepsilon \log(N-1)-\varepsilon \log \varepsilon -(1-\varepsilon)\log (1-\varepsilon) +\varepsilon.
	\end{equation}
	However notice that
	$
	C(\varepsilon, \varepsilon') > \varepsilon.
	$
	For any $k\geq N( \varepsilon, \varepsilon')$, define $$\widehat{R}_k(\varepsilon, \varepsilon')=\{ \hat{x}\in\widehat{X} : n(\hat{x},\varepsilon, \varepsilon')\leq k, n'(\hat{x},\varepsilon, \varepsilon')\leq k, n''(\hat{x},\varepsilon)\leq k   \}.$$
	Notice that $\left\{\widehat{R}_k(\varepsilon, \varepsilon')\right\}_k$ is an increasing sequence of Borel sets and 
	\begin{equation}\label{b2}
	\hat{\mu}(\widehat{R}_k(\varepsilon, \varepsilon')))\mathop{\longrightarrow}_{k\rightarrow\infty}1.	
	\end{equation}
	Fix $k\geq N(\varepsilon,\varepsilon')$. 
	It follows from (\ref{Vn}) that the total number of elements of $\mathcal{P}_{n}$ with measure greater than $e^{-(h_f(\mu, \mathcal{P})-2 C(\varepsilon, \varepsilon'))n}$ is at most $e^{(h_f(\mu, \mathcal{P})-2 C(\varepsilon, \varepsilon'))n}$, for all $n\geq N(\varepsilon, \varepsilon')$. Denote the set of these elements by $\Xi_n$. The total number $Q_n$ of elements of $\mathcal{P}_{n}$ belonging to the Hamming $\varepsilon$-neighborhood of $\Xi_n$ satisfies 
	\begin{equation}\label{b3}
	Q_n \leq  V_n\cdot e^{(h_f(\mu, \mathcal{P})  - 2C(\varepsilon, \varepsilon'))n}= e^{(h_f(\mu, \mathcal{P})  - C(\varepsilon, \varepsilon'))n}.	
	\end{equation}
	Recall that $k\geq N(\varepsilon, \varepsilon')$ is fixed. From these $Q_n$ elements of $\mathcal{P}_{n}$ consider those whose intersection with $\pi_{n}(\widehat{R}_k(\varepsilon, \varepsilon'))$ has positive measure and denote their union by $E_n(\varepsilon, \varepsilon')$. Then from the definition of  $\widehat{R}_k(\varepsilon, \varepsilon')$ and from (\ref{e2}) and (\ref{b3}), we have that, for all $n>N(\varepsilon, \varepsilon')$, 
	$$\mu(E_{n}(\varepsilon, \varepsilon'))\leq e^{(h_f(\mu, \mathcal{P})  - 2 C(\varepsilon, \varepsilon'))n} \cdot e^{(-h_f(\mu, \mathcal{P})+\varepsilon)n} = e^{(\varepsilon -C(\varepsilon, \varepsilon'))n}.$$
	Since $C(\varepsilon,\varepsilon')>\varepsilon$, there exists $k(\varepsilon, \varepsilon')$ such that for every $n\geq k(\varepsilon, \varepsilon')$ we have
	\begin{equation}\label{b4}
	\sum_{n\geq k(\varepsilon, \varepsilon')}\mu( E_n(\varepsilon,\varepsilon'))< \varepsilon.
	\end{equation} 
	Also by (\ref{b2}) we can assume that for every $k\geq k(\varepsilon,  \varepsilon')$ we have
	\begin{equation}\label{i1}
	\hat{\mu}(\widehat{X}\setminus \widehat{R}_k(\varepsilon,\varepsilon'))<\varepsilon.
	\end{equation}
	Since $\hat{\mu}(\pi_n^{-1}(E_n(\varepsilon,\varepsilon')))=\mu(E_n(\varepsilon,\varepsilon'))$, from (\ref{b4}) we have 
	\begin{equation}\label{i3}
	\hat{\mu}(\bigcup_{n\geq k(\varepsilon,\varepsilon')} \pi_n^{-1}(E_n(\varepsilon,\varepsilon'))) \leq \sum_{n\geq k(\varepsilon,\varepsilon')}\hat{\mu}(\pi_n^{-1}(E_n(\varepsilon,\varepsilon')))= \sum_{k\geq k(\varepsilon,\varepsilon')}\mu(E_n(\varepsilon,\varepsilon')) <\varepsilon.
	\end{equation} 
	For $k\geq k(\varepsilon, \varepsilon')$ define 
	\begin{equation}\label{i5}
	\widehat{Q}_k(\varepsilon,\varepsilon')=\widehat{R}_k(\varepsilon,\varepsilon')\setminus \bigcup_{n\geq k}\pi_n^{-1} (E_n(\varepsilon,\varepsilon')).
	\end{equation}
	Then by (\ref{i1}), (\ref{i3}) and (\ref{i5}), for any $k\geq k(\varepsilon, \varepsilon')$ we have 
	\begin{equation}\label{b7}
	\hat{\mu} (\widehat{Q}_k(\varepsilon,\varepsilon')) >1- 2\varepsilon.
	\end{equation}
	Let $\hat{x}\in \widehat{Q}_k(\varepsilon,\varepsilon')$. Hence $\hat{x} \in \widehat{R}_k(\varepsilon,\varepsilon')$ and $x_{-n}\notin E_{n}(\varepsilon,\varepsilon')$ for every $n\geq k$. Let 
	$\Gamma(\hat{x},n,\varepsilon)$ be the collection of elements from $\mathcal{P}_{n}$
	whose $(\mathcal{P},n)$-names are $\varepsilon$-close in Hamming distance to the $(\mathcal{P},n)$-name of $x_{-n}$, since $\hat{x}\in \widehat{Q}_k(\varepsilon,\varepsilon')$. If $P\in \Gamma(\hat{x},n,\varepsilon)$ and $y\in P$,  
	then the Hamming distance between the $(\mathcal{P},n)$-names of $y$ and $x_{-n}$ is less than $\varepsilon$. 	Let $M>0$ be such that $J_f(\mu)(x)<M$ for every $x\in X$. Thus $f^i(y) \in \mathcal{P}(f^{i}(x_{-n}))$ for at least $n-[(n+1)\varepsilon]$ of  indices $i\in \{1,\ldots, n\}$ and if  $f^i(y) \notin \mathcal{P}(f^{i}(x_{-n}))$ then $J_{f}(\mu)(f^{i}(y)) \leq M\leq 
	M J_{f}(\mu)(f^i(x_{-n}))$. Consequently from (\ref{b0}),
	\begin{equation}\label{b5}
	J_{f^{n}}(\mu)(y)< J_{f^{n}}(\mu)(x_{-n})e^{(n-[(n+1)\varepsilon])\varepsilon' }M^{[(n+1)\varepsilon]}\leq J_{f^{n}}(\mu)(x_{-n})e^{\varepsilon' n}M^{(n+1)\varepsilon}.	
	\end{equation}
	Since the atoms of $\mathcal{P}_n$ with measure greater then $e^{-(h_f(\mu, \mathcal{P})-2C(\varepsilon, \varepsilon'))n}$ together with their neighbors $\varepsilon$-close in Hamming distance were eliminated in the definition of $\widehat{Q}_k(\varepsilon,\varepsilon')$, it follows that for all $P\in \Gamma(\hat{x},n,\delta)$, $\mu(P)<e^{-(h_f(\mu, \mathcal{P})-2C(\varepsilon, \varepsilon'))n}$.
	Thus, for every $\hat{x}\in \widehat{Q}_k(\varepsilon,\varepsilon')$, every $n\geq k$ and $\delta<\delta(\varepsilon)$, we have from the discussion about $B^-_n(\hat{x},\delta)$ before (\ref{b10}) together with (\ref{Vn}) and (\ref{b5}) that
	\begin{align*}\mu(B^-_{n}(\hat{x},\delta)) &=\mu(f^{n}(B_n(x_{-n},\delta))) \leq \sum_{P\in \Gamma(\hat{x},n,\varepsilon)} \mu (f^{n}(P)) = \sum_{P\in \Gamma(\hat{x},n,\varepsilon)}  \int_{P} J_{f^{n}}(\mu) d\mu\\
	&\leq V_n\cdot  e^{(-h_f(\mu, \mathcal{P})+2C(\varepsilon, \varepsilon'))n}\cdot   J_{f^{n}}(\mu)(x_{-n})\cdot e^{\varepsilon'n}\cdot M^{\varepsilon (n+1)}\\
	&\leq e^{C(\varepsilon,\varepsilon')n}\cdot e^{(-h_f(\mu, \mathcal{P})+ 2C(\varepsilon, \varepsilon') )n} \cdot  e^{(F_f(\mu)+\varepsilon)n}\cdot e^{n\varepsilon'}\cdot M^{\varepsilon (n+1)}\\
	&\leq  e^{(-h_f(\mu, \mathcal{P})+F_f(\mu)+3C(\varepsilon, \varepsilon') +\varepsilon +\varepsilon')n} \cdot M^{\varepsilon (n+1)},
	\end{align*}
	and therefore
	\begin{equation}\label{i7}
	\lim_{\delta\rightarrow 0}\liminf_{n\rightarrow\infty}\frac{- \log \mu(B^-_n(
		\hat{x},\delta)) }{n} \geq h_f(\mu, \mathcal{P})- F_f(\mu)-3C(\varepsilon, \varepsilon')- \varepsilon -\varepsilon \log M- \varepsilon'.
	\end{equation}
	Now recall that $\varepsilon'$ is fixed. Then for any $p>1$ let $k_p>k(\frac{1}{2^p},\varepsilon')$ such that $k_{p+1}>k_p$ and define $\widehat{Q}_{k_p}(\varepsilon'):=\widehat{Q}_{k_p}(\frac{1}{2^p},\varepsilon')$. For $m>1$, let $\widehat{Q}_m(\varepsilon'):=\mathop{\bigcap}\limits_{p>m} \widehat{Q}_{k_p}(\varepsilon')$. From (\ref{b7}) it follows that  $\hat{\mu}(\widehat{Q}_{k_p}(\varepsilon'))>1-\frac{1}{2^{p-1}}$. Hence
	\begin{equation}\label{b6}
	\hat{\mu}(\widehat{Q}_m(\varepsilon'))>1-\sum_{p>m}\frac{1}{2^{p-1}}=1-\frac{1}{2^{m-1}}.	
	\end{equation} 
	From (\ref{b8}) we know that  
	$C(\varepsilon, \varepsilon')=\varepsilon \log(N-1)-\varepsilon \log \varepsilon -(1-\varepsilon)\log (1-\varepsilon) +\varepsilon,$
	and $\mathop{\lim}\limits_{\varepsilon\rightarrow 0}C(\varepsilon, \varepsilon')=0$. 
	Then from (\ref{i7}) we obtain for every $m\geq 1$ and every $\hat{x}\in \widehat{Q}_m(\varepsilon')$ that
	\begin{equation}\label{b9}
	\lim_{\delta\rightarrow 0}\liminf_{n\rightarrow \infty}\frac{- \log \mu(B^-_n(
		\hat{x},\delta)) }{n} \geq h_f(\mu, \mathcal{P})- F_f(\mu)-\varepsilon'\geq h_f(\mu)- F_f(\mu)-2\varepsilon'.
	\end{equation}
	Notice that $\widehat{Q}_m(\varepsilon')\subset \widehat{Q}_{m+1}(\varepsilon')$ for every $m\geq 1$. Let now 
	$\widehat{Q}(\varepsilon')=\mathop{\bigcup}\limits_{m>1}\widehat{Q}_m(\varepsilon')$. Then  from (\ref{b6}), it follows that $\widehat{Q}(\varepsilon')$ has $\hat{\mu}$-measure equal to $1$. Finally let $\widehat{Q}=\mathop{\bigcap}\limits_{q>1}\widehat{Q}(\frac{1}{2^q})$. Then $\hat{\mu}(\widehat{Q})=1$ and for every $\hat{x}\in \widehat{Q}$ we have 
	$$h^{-}_{f,inf, B}(\hat{x})=\lim_{\delta\rightarrow 0}\liminf_{n\rightarrow \infty}\frac{- \log \mu(B^-_n(
		\hat{x},\delta)) }{n} \geq h_f(\mu, \mathcal{P})- F_f(\mu).$$
	Then the conclusion of the theorem follows from Proposition \ref{c1}.	
	
	$\hfill\square$

Let $f:M\rightarrow M$ be a smooth (say $C^2$) non-invertible map  defined on a Riemannian manifold $M$, and let $\Lambda\subset M$ be a compact $f$-invariant set. We recall that $f$ is (uniformly) hyperbolic on $\Lambda$ if there exists a continuous splitting
of the tangent bundle over $\widehat{\Lambda}$ into stable and unstable directions. For every $\hat{x}\in\widehat{\Lambda}$ we have a stable space $E^s_{x}$ and an unstable space $E^{u}_{\hat{x}}$ and these subspaces are invariant under $Df$ (see \cite{Ru}). 
The above splitting gives birth, for some $\vp>0$, to local stable/unstable manifolds  $W_\varepsilon^s(x)$ and $W_\varepsilon^u(\hat{x})$ for every $\hat{x}\in \widehat{\Lambda}$,  where
$$W_\varepsilon^s(x)=\{ y\in X : d(f^nx, f^ny)<\varepsilon, \forall n\geq 0 \} \text { and}$$
$$W_\varepsilon^u(\hat{x})=\{ y\in X : \exists \text{ a prehistory }\hat{y}=(y_{-n})_{n\geq 0} \text{ of }y \text{ such that }d(x_{-n}, y_{-n})<\varepsilon, \forall n\geq 0 \}.$$


Let $f: M\rightarrow M$ be a $\mathcal{C}^2$ smooth  endomorphism defined on a compact Riemannian manifold. 
We will now study the inverse metric entropy of an $f$-invariant \textbf{hyperbolic measure} $\mu$. For background on hyperbolic ergodic measures the book of Barreira and Pesin \cite{BP} is a good reference. 
We recall that if $f: M\rightarrow M$ is a $\mathcal{C}^2$ smooth endomorphism on a compact Riemannian manifold $M$, and if $\mu$ is an $f$-invariant ergodic hyperbolic measure, then for $\varepsilon>0$ there exists a Pesin set $\widehat{R}_\varepsilon \subset \widehat M$ such that for every $\hat{x}\in\widehat{R}_\varepsilon$ there exists a local stable manifold $W^s_\varepsilon(x)$ and a local unstable manifold $W^u_\varepsilon(\hat{x})$ of size $\varepsilon$. One has also the estimates from \cite{BP} for the distances between the iterates of points in  $W^s_\varepsilon(x)$ and $W^u_\varepsilon(\hat{x})$. Moreover $\bigcup\limits_{\varepsilon>0} \widehat{R}_\varepsilon=\widehat{M}$ up to a set of zero $\hat\mu$-measure (see \cite{BP}).

\begin{definition}\label{preind}
	Let $f :M \rightarrow M$ be a $\mathcal{C}^2$ endomorphism defined on a compact Riemannian manifold. Assume that $\mu$ is an $f$-invariant ergodic measure on $M$.

 a) The measure $\mu$ is called hyperbolic if for $\mu$-a.e $x\in M$ all the Lyapunov exponents of $\mu$ at $x$ are different from zero.	 

b) A hyperbolic measure $\mu$ is called special (or prehistory independent) if for any $\vp \in (0, \vp_0)$ and every prehistories $\hat{x}$, $\hat{y} \in \widehat R_\vp$ with $x =y$, we have that $W^u_\varepsilon(\hat{x})$ is equal to $W^u_\varepsilon(\hat{y})$.

\end{definition}

\begin{remark}
Special hyperbolic endomorphisms, i.e endomorphisms whose unstable manifolds depend only on their base point (and not on the entire prehistory) will be presented in more detail in Section \ref{specialAnosov}. If $f$ is a special hyperbolic endomorphism on $\Lambda$, then clearly any $f$-invariant ergodic measure on $\Lambda$ is hyperbolic and special. 
\end{remark}



\textbf{Proof of Theorem \ref{thpreind}.} By Proposition \ref{c1} it is enough to prove that for $\hat \mu$-a.e $\hat x \in \widehat M$,
	$$\lim_{\varepsilon\rightarrow 0}\liminf_{n\rightarrow\infty} \frac{-\log \mu(f^{n}(B_{n}(x_{-n},\varepsilon)))}{n}\geq  h_{f}(\mu)- F_f(\mu).$$	
	For $\varepsilon>0$ let $\widehat{R}_\varepsilon\subset\widehat{M}$ be a Pesin regular set for $\hat{\mu}$ (see for example \cite{BP}) such that 
	\begin{equation}\label{Repsilon}
\hat{\mu}(\widehat{R}_\varepsilon)>1-\eta(\varepsilon), \text{ where  } \lim_{\varepsilon\rightarrow 0}\eta(\varepsilon)= 0.
	\end{equation}
	Let $\tau > 0$ and $\varepsilon>0$ and recall Lemma \ref{mod}. Let 
	$$\widehat{T}_\varepsilon(\tau)=\left\{\hat{x}\in \widehat{M} :	\left |\liminf_{n\rightarrow\infty}\frac{-\log \mu(B_{n}(x_{-{n}},\varepsilon))}{n} -h_f(\mu)\right|<\tau \right\}.$$ Hence there exists $0<\varepsilon(\tau)<\tau$ such that for every $0<\varepsilon\leq\varepsilon(\tau)$ we have $\hat{\mu}(\widehat{T}_\varepsilon(\tau))>1-\tau$. Let $\varepsilon\in (0, \varepsilon(\tau)]$. If $\hat{x}\in \widehat{R}_\varepsilon$ then there exist the local stable manifold $W^s_\varepsilon(x)$ of size $\varepsilon$ and the local unstable manifold $W^u_\varepsilon(\hat{x})$ of size $\varepsilon$. For any $m\geq 1$ let $\widehat{A}_m(\varepsilon,\tau)$ be the set of all $\hat{x}\in \widehat{R}_\varepsilon \cap \widehat{T}_\varepsilon(\tau)$ which satisfy the following three conditions: 
	\begin{equation}\label{amepsdelta2}
	\left|\frac{1}{n}\sum_{i=0}^{n-1} \chi_{\widehat{R}_\varepsilon\cap\widehat{T}_\varepsilon(\tau)}(f^{-i}(x))-\hat{\mu}(\widehat{R}_\varepsilon\cap \widehat{T}_\varepsilon(\tau))\right|<\varepsilon, \  \text{ for all }n\geq 2m,
	\end{equation}
	\begin{equation}\label{amepsdelta1}
	\left|\frac{\log J_{f^{n}}(\mu)(x_{-n})}{n}-F_f(\mu)\right|<\varepsilon, \text{ for all }n\geq m,
	\end{equation}
		\begin{equation}\label{amepsdelta3}
	\frac{-\log \mu(B_{n}(x_{-{n}},\varepsilon))}{n}> h_f(\mu)-\tau, \text{ for all }n\geq m.
	\end{equation}
	By Birkhoff Ergodic Theorem applied to $\hat{f}^{-1}$ on $(\widehat{M},\hat{\mu})$ and to the functions $\log J_f(\mu)\circ \pi :\widehat{M}\rightarrow \mathbb{R}$ and $\chi_{\widehat{R}_\varepsilon}:\widehat{M}\rightarrow \mathbb{R}$, and by (\ref{Repsilon}) and Lemma \ref{mod} we have
	\begin{equation}\label{limamn}
	\lim_{m\rightarrow\infty}\hat{\mu}(\widehat{A}_m(\varepsilon,\tau))=\hat{\mu}(\widehat{R}_\varepsilon \cap \widehat{T}_\varepsilon(\tau)) > 1-\eta(\varepsilon) -\tau.
	\end{equation}
	Let $\widehat{D}_m(\varepsilon,\tau)= \Big\{ \hat{x}\in \widehat{A}_m(\varepsilon,\tau):\mu(f^{n}(B_{n}(x_{-{n}},\varepsilon)))  > n^2\cdot \mu\Big(f^{n}\big(B_{n}(x_{-{n}},\varepsilon)\cap \pi_{n}(\widehat{A}_{m}(\varepsilon,\tau))\big)\Big)
	\text{ for } \\ \text{infinitely many }n   \Big\}$. 
By choosing $\varepsilon$ and $\tau$ sufficiently small, without loss of generality we can assume that 	
$\hat{\mu}(\widehat{R}_\varepsilon\cap \widehat{T}_\varepsilon(\tau))-\varepsilon >1/2$. Then it follows from (\ref{amepsdelta1}) that for every $\hat{x}\in \widehat{A}_{m}(\varepsilon,\tau)$ and $n\geq 2m$, the number of positive integers 
$k$ with $0\leq k<n$ such that $\hat{f}^{-k}(\hat{x})\in \widehat{R}_\varepsilon\cap \widehat{T}_\varepsilon(\tau)$ is larger than $m$. For $\hat{x}\in \widehat{A}_{m}(\varepsilon, \tau)$ and $n\geq 2m$, let $m_n$ be the largest integer smaller than $n$ (which depends on $\hat{x}$), such that $\hat{f}^{-m_n}(\hat{x})\in \widehat{R}_\varepsilon\cap \widehat{T}_\varepsilon(\tau)$. From above it follows that $m_n\geq m$ and from (\ref{amepsdelta2}) we have
\begin{equation} \label{Tepshat}
\frac{m_n}{n}\geq \hat{\mu}(\widehat{R}_\varepsilon\cap \widehat{T}_\varepsilon(\tau))-\varepsilon.
\end{equation}
Notice also that
	\begin{equation}\label{fnbn}
 f^n(B_n(x_{-n},\varepsilon))\subset f^{m_n}(B_{m_n}(x_{-m_n},\varepsilon)).
	\end{equation}
For $n\geq 2m$ define the measurable subset of $\widehat M$,
	\begin{equation}\label{Enepsd}
	\begin{aligned}
		\widehat{E}_n(\varepsilon,\tau)=\big\{ \hat{x}\in \widehat{D}_m(\varepsilon,\tau) &: \mu(f^{m_n}(B_{m_n}(x_{-{m_n}},\varepsilon)))> \\&> n^2\cdot \mu(f^{m_n}\big(B_{m_n}(x_{-m_n},\varepsilon)\cap \pi_{m_n}(\widehat{A}_{m}(\varepsilon,\tau))\big)) \big\}.
	\end{aligned}
	\end{equation}
	Now we want to cover the set $\pi(\widehat{E}_n(\varepsilon,\tau))$ with sets of the type $f^{m_n}( B_{m_n}(x_{-m_n}, \varepsilon))$. First we fix $y\in \pi\widehat{R}_\varepsilon$ and take the intersection $W^s_\varepsilon(y)\cap \pi(\widehat{E}_n(\varepsilon,\tau))$. Then for any $\hat{x}\in \widehat{E}_n(\varepsilon,\tau)\subset\widehat{R}_\varepsilon$,  
	 $W^s_\varepsilon(y)\cap f^{m_n}( B_{m_n}(x_{-m_n}, \varepsilon))$ is a small parallelepiped in $W^s_\varepsilon(y)$ of dimension equal to the dimension of $W^s_\varepsilon(y)$ whose sides are parallel to the stable tangent subspaces. Since $\hat{f}^{-m_n}(\hat{x})\in \widehat{R}_\varepsilon$, we apply the estimates on the distances between iterates of points from $W_\varepsilon^s(x_{-m_n})$ (see \cite{BP}), and the fact that the contraction along the stable manifolds is stronger than the subexponential oscillation of the size of local stable/unstable manifolds and of the multiplicative constant. 
	 
	Let us now cover the set $W^s_\varepsilon(y)\cap \pi(\widehat{E}_n(\varepsilon,\tau))$ with a family $\mathcal{F}$ of small parallelepipeds of type $f^{m_n}(B_{m_n}(x_{-m_n},\varepsilon))\cap W^s_\varepsilon(y)$. Given that these parallelepipeds have sides parallel to a finite set of stable directions, we can apply a version of Besicovitch Covering Theorem for this family $\mathcal{F}$. Thus there exists a constant $N$ (which depends only on the dimension of the manifold $M$) such that we can extract at most $N$ subfamilies $\mathcal{G}_1, \mathcal{G}_2,\ldots, \mathcal{G}_N$ of $\mathcal{F}$ such that each such family $\mathcal{G}_i$ consists of mutually disjoint parallelepipeds in $W^s_\varepsilon(y)$ and $\mathcal{G}_1 \cup  \mathcal{G}_2 \cup \ldots\cup  \mathcal{G}_N$
	 covers $W^s_\varepsilon(y)\cap \pi(\widehat{E}_n(\varepsilon,\tau))$. Let us denote by $\widetilde{G}_k$ the family of sets of type $f^{m_n}(B_{m_n}(x_{-{m_n}},\varepsilon))$, where $f^{m_n}(B_{m_n}(x_{-{m_n}},\varepsilon)) \cap  W^s_\varepsilon(y)\in G_k$, for $k=1,\ldots, N.$ Let also $\widetilde{G}(y)$ to be the union of all sets from the families $\widetilde{G}_1,\ldots, \widetilde{G}_N$. Now since we work on $\widehat{R}_\varepsilon$, and since $\mu$ is special, the local unstable manifolds depend only on their respective base points in $W^s_\varepsilon(y) \cap \pi(\widehat{E}_{n}(\varepsilon, \tau))$. Thus the sets in each family $\widetilde{G}_i$ are mutually disjoint, for $i=1,\ldots, N$. Therefore we obtain that $\widetilde{G}(y)$ covers the set $B(y,\varepsilon)\cap \pi(\widehat{E}_n(\varepsilon,\tau))$, and from (\ref{Enepsd}) it follows that for each set $f^{m_n}(B_{m_n}(x_{-{m_n}},\varepsilon))$ from $\widetilde{G}(y)$ we have

	 \begin{equation}\label{mufnbnam}
	 \mu(f^{m_n}(B_{m_n}(x_{-{m_n}},\varepsilon))\cap \pi_{m_n}(\widehat{A}_{m}(\varepsilon, \tau)))< \frac{1}{n^2}\cdot\mu(f^{m_n}(B_{m_n}(x_{-{m_n}},\varepsilon))).
	 \end{equation}
	 Let $K_\varepsilon$ be the minimum number of balls of radius $\varepsilon/2$ which cover $M$. Thus, since the sets in each family $\widetilde{G}_i$ are mutually disjoint for $i=1,\ldots, N$, we infer from (\ref{mufnbnam}) that
	 \begin{equation}\label{nkepsn}
	 \mu(\pi(\widehat{E}_n(\varepsilon,\tau)))< \frac{N\cdot K_\varepsilon}{n^2}.
	 \end{equation}
	But $\hat{\mu}(\widehat{E}_n(\varepsilon,\tau)) \leq \hat{\mu}\left( \pi^{-1}(\pi(\widehat{E}_n(\varepsilon,\tau)))\right)=\mu(\pi(\widehat{E}_n(\varepsilon,\tau)))$, and then from (\ref{nkepsn}) we obtain that $\sum\limits_{n=1}^\infty \hat{\mu}(\widehat{E}_n(\varepsilon,\tau)) <\infty$. Then, from Borel-Cantelli Lemma we get $$\hat{\mu}\big(\bigcap_{k\geq 1 }\bigcup_{n\geq k}\widehat{E}_n(\varepsilon,\tau)\big)=0.$$
	Hence for $\hat{\mu}$-a.e $\hat{x}\in \widehat{A}_m(\varepsilon,\tau)$, there is $n(\hat{x})\geq 2m$ such that  $\hat{x}\notin \widehat{E}_n(\varepsilon,\tau)$ for any $n\geq n(\hat{x})$, and thus
	\begin{equation}\label{muf}
	\mu(f^{m_n}(B_{m_n}(x_{-{m_n}},\varepsilon)) \leq  n^2\cdot \mu(f^{m_n}(B_{m_n}(x_{-m_n},\varepsilon))\cap \pi_{m_n}(\widehat{A}_m(\varepsilon,\tau))).
	\end{equation}
	Hence from  (\ref{amepsdelta1}), (\ref{fnbn}) and (\ref{muf}) and since $m_n\geq m$, it follows that for $\hat{\mu}$-a.e $\hat{x}\in \widehat{A}_m(\varepsilon,\tau)$ and every $n\geq n(\hat{x})$ we obtain,
	\begin{align*}
	\mu(f^{n}(B_{n}(x_{-n},\varepsilon)))&\leq  \mu(f^{m_n}(B_{m_n}(x_{-m_n},\varepsilon))) \leq    n^2\cdot \int_{B_{m_n}(x_{-m_n},\varepsilon) \cap \pi_{m_n}(\widehat{A}_m(\varepsilon,\tau))} J_{f^{m_n}}(\mu) \ d\mu
	\\ &\leq n^2\cdot  e^{m_n(F_f(\mu)+\varepsilon)}\cdot \mu(B_{m_n}(x_{-{m_n}},\varepsilon)).
	\end{align*}
Then, for $\hat{\mu}$-a.e. $\hat{x}\in \widehat{A}_m(\varepsilon,\tau)$ and every $n\geq n(\hat{x})$ we have	
		$$\frac{-\log \mu(f^{n}(B_{n}(x_{-{n}},\varepsilon)))}{n}\geq \frac{-\log \mu(B_{m_n}(x_{-{m_n}},\varepsilon))}{m_n}\cdot\frac{m_n}{n}- (
F_f(\mu)+\varepsilon)\cdot\frac{m_n}{n}-\frac{\log n^2}{n },$$
		and therefore 
		\begin{equation*}
	\liminf_{n\rightarrow\infty} \frac{-\log \mu(f^{n}(B_{n}(x_{-n},\varepsilon)))}{n} \geq
	\liminf_{n\rightarrow\infty} \left(\frac{-\log \mu(B_{m_n}(x_{-{m_n}},\varepsilon))}{m_n}\cdot\frac{m_n}{n}  - (F_f(\mu)+\varepsilon)\cdot\frac{m_n}{n} \right)
		\end{equation*}
But from (\ref{amepsdelta3}) and since $\frac{m_n}{n}\geq \hat{\mu}(\widehat{R}_\varepsilon\cap \widehat{T}_\varepsilon(\tau))-\varepsilon$ (see (\ref{Tepshat})), we obtain 		
\begin{equation}\label{infhmufmu}
\liminf_{n\rightarrow\infty} \frac{-\log \mu(f^{n}(B_{n}(x_{-n},\varepsilon)))}{n} \geq
(h_f(\mu)-\tau-F_f(\mu)-\varepsilon)\cdot (\hat{\mu}(\widehat{R}_\varepsilon\cap \widehat{T}_\varepsilon(\tau))-\varepsilon) 
\end{equation}	
Notice that from (\ref{limamn}), for every $\varepsilon\leq\varepsilon(\tau)$ and for $m$ larger than some number $m(\varepsilon,\tau)$,
$$\hat{\mu}\left(\widehat{A}_{m}(\varepsilon, \tau)\right)\geq \hat{\mu}\left(\widehat{R}_{\varepsilon}\cap \widehat{T}_{\varepsilon}(\tau)\right)-\tau>1-\eta(\varepsilon)-\tau.$$ 
Hence for any integer $p>1$, there exists $\kappa(p)\in \mathbb{N}$, $\varepsilon(p)>0$ and $\tau(p)>0$ such that if $\widehat{A}_p$ denotes the set $\widehat{A}_{\kappa(p)}(\varepsilon(p),\tau(p))$, then $\hat\mu(\widehat{A}_p)> 1-\frac{1}{2^p}.$
Let now $\widehat{A} =\mathop{\bigcup}\limits_{k\geq 1} \mathop{\bigcap}\limits_{p\geq k} \widehat{A}_p$. Then $\hat{\mu}(\widehat{A})=1$ and from (\ref{limamn}) and (\ref{infhmufmu}) we obtain that for every $\hat{x}\in \widehat{A}$,
	$$\lim_{\varepsilon\rightarrow 0}\liminf_{n\rightarrow\infty} \frac{-\log \mu(f^{n}(B_{n}(x_{-n},\varepsilon)))}{n}\geq
	h_f(\mu)- F_f(\mu).$$

$\hfill\square$

\section{Special Anosov endomorphisms on tori}\label{specialAnosov}

If $f:M \to M$ is a smooth ($\mathcal C^\infty$) map on a compact manifold, then recall that $f$ is called \textbf{Anosov endomorphism} if $f$ is hyperbolic (as an endomorphism) over the entire manifold $M$ (for eg \cite{KH}). Thus there exists a continuous splitting of the tangent bundle over the inverse limit $\widehat M_f$ into $Df$-invariant stable and unstable tangent subbundles $T_{\hat x}M = E^s(x) \bigoplus E^u(\hat x)$, and there exists $\alpha \in (0, 1)$
such that for any $\hat x = (x, x_{-1}, \ldots) \in \widehat M_f$ we have that   $Df|_{E^s(x)}$ contracts with a factor smaller than $\alpha$ and $Df|_{E^u(\hat x)}$ expands with a factor larger than $1/\alpha$. 

Let us recall now also some notions related to endomorphisms from \cite{AH} and \cite{S}.
Firstly, a continuous surjection $f: X \to X$ on a compact metric space $(X, d)$ is called a \textit{covering map} if $f$ is a local homeomorphism. A continuous surjection $f: X \to X$ is called \textit{c-expansive} (\textit{constant-expansive}) if there exists some constant $e>0$ such that if $\hat x, \hat y \in \widehat X_f$ and $d(x_i, y_i) \le e, i \in \mathbb Z$ (where for $i >0$ we let $x_i = f^i(x)$), then $\hat x = \hat y$. If $\delta >0$, then a  sequence of points $\{x_i, i \ge 0\}$ is called a $\delta$-\textit{pseudo-orbit} if $d(f(x_i), x_{i+1}) < \delta$ for all $i \ge 0$. If $\vp>0$ we say that a $\delta$-pseudo-orbit $\{x_i, i \ge 0\}$ is $\vp$-\textit{traced} by a point $x$ if $d(x_i, f^i(x)) < \vp, i \ge 0$.   Then $f$ has \text{POTP} (\textit{pseudo-orbit tracing property}) if for every $\vp>0$ there exists some $\delta>0$ such that every $\delta$-pseudo-orbit can be $\vp$-traced by a point. The continuous surjection $f:X \to X$ is called a \textbf{topological Anosov map} (or \textbf{TA-map}) if $f$ is $c$-expansive and has POTP.  For a TA-map $f$ on  $X$ and for any $\hat x \in \widehat X_f$, define the \textit{global unstable set} as  
$$ W^u(\hat x) := \{y_0 \in X: \exists \ \hat y \in \widehat X_f \ \text{with} \ \mathop{\lim}\limits_{i \to \infty} d(x_{-i}, y_{-i}) = 0\}.$$
Then the map $f$ is called \textbf{special} if $W^u(\hat x) = W^u(\hat y)$ for every $\hat x, \hat y \in \widehat X_f$ with $x_0 = y_0$. 

If $f:X \to X$ is a special TA-map, then any $f$-invariant ergodic measure $\mu$ is special according to  Definition \ref{preind}.
Clearly, if $f:M \to M$ is an Anosov endomorphism, then $f$ is $c$-expansive and has POTP (see \cite{KH}), thus it is a TA-map. If $f: \mathbb T^d \to \mathbb T^d$ is a linear hyperbolic endomorphism, then $f$ is special. However  there exist many Anosov endomorphisms on tori which are not special, in fact any Anosov endomorphism on $\mathbb T^d$ can be approximated with Anosov endomorphisms which are not special (\cite{Pr}). Also for an Anosov endomorphism $f:\mathbb T^d \to \mathbb T^d$ without critical points,  the number of $f$-preimages of any point is constant (say equal to $D$), and we call this number the \textbf{degree} of $f$, so $$D = Card (f^{-1}(x)), \ \forall x \in \mathbb T^d.$$
 
If $f:\mathbb T^d \to \mathbb T^d$ is an Anosov endomorphism for $d \ge 2$, then $f$ is homotopic to a hyperbolic linear endomorphism $f_L: \mathbb T^d \to \mathbb T^d$ called the \textbf{linearization} of $f$. The integer-valued matrix of $f_L$ is determined by the induced homomorphism $f_*: \pi_1(\mathbb T^d) \to \pi_1(\mathbb T^d)$, where we recall that the fundamental group $\pi_1(\mathbb T^d)$ is equal to $\mathbb Z^d$. By extending a previous result from \cite{AH}, Sumi  proved in \cite{S} that any special TA-covering self-map on $\mathbb T^d$ (i.e covering map which is TA and special) is topologically conjugate to its linearization.

\begin{theorem*}(Linearization Theorem for special TA-covering maps, \cite{S}).
Let $f:\mathbb T^d \to \mathbb T^d$ be a special TA-covering map, and $f_L:\mathbb T^d \to \mathbb T^d$ be its linearization.  Then $f_L$ is a hyperbolic toral endomorphism and $f$ is topologically conjugate to $f_L$. 
\end{theorem*}

Given a $\mathcal C^\infty$ Anosov endomorphism $f:M \to M$ without critical points, one has the SRB (Sinai-Ruelle-Bowen) measure $\mu^+_f$ on $M$ which describes the asymptotic distribution of forward iterates of Lebesgue-a.e point $x\in M$ (see for eg \cite{Si}, \cite{Bo-carte}, \cite{Pe}, \cite{QZ},  \cite{Y}), and the inverse SRB measure $\mu^-_f$ introduced in \cite{M4} which describes the asymptotic distribution of the $n$-preimage sets of Lebesgue-a.e point $x\in M$. Recall that $\mu^+_f$ is the unique $f$-invariant probability measure absolutely continuous on the local unstable manifolds of $f$, while $\mu^-_f$ is the unique $f$-invariant probability measure absolutely continuous on the local stable manifolds of $f$. Moreover, the inverse SRB measure $\mu^-_f$ is the equilibrium measure of the stable potential $\log |\text{det}(Df|_{E^s(x)})|$ (see \cite{M4}).

In \cite{An} it was shown that if $f: \mathbb T^2 \to \mathbb T^2$ is a non-invertible Anosov endomorphism, then $f$ is special if and only if every periodic point admits the same Lyapunov exponent on the stable bundle, i.e 
$$ \lambda^s_f(p) = \lambda^s_{f_L}, \forall p \in Per(f),$$
where $\lambda^s_{f_L}$ is the Lyapunov exponent on the stable bundle for the linearization $f_L$ of $f$.   However, notice that the conjugacy above is only topological, not necessarily smooth ($\mathcal C^\infty$), and then the unstable Lyapunov exponents of $f$ at periodic points may be different from the unstable Lyapunov exponent $\lambda^u_{f_L}$ of $f_L$. This is the problem of \textbf{rigidity} in dynamics, namely when can we obtain a stronger conjugacy  (smooth) from a weaker conjugacy (topological). This is a difficult problem in general, since  the topological conjugacy obtained in \cite{AH} and \cite{S} is at most H\"older continuous (see \cite{KH}), but  it may be nowhere differentiable. The rigidity problem was studied in many cases for Anosov diffeomorphisms and Anosov endomorphisms, for eg by \cite{RL}, \cite{An}, \cite{Mi}. 
If $f, g:M \to M$ are Anosov endomorphisms on a manifold $M$ and if $\Phi$ is a smooth conjugacy with $\Phi\circ f = g \circ \Phi$, then $Df(x) = (D\Phi^{-1}\circ Dg \circ D\Phi) (x), x \in M$ as matrices. In this case the upper/lower Lyapunov exponents of $f$ and $g$ coincide at corresponding points.  

We now prove the \textbf{entropy rigidity} from Theorem \ref{clasAnendos}, namely that special Anosov endomorphisms on $\mathbb T^2$ can be classified up to smooth conjugacy by using the inverse entropy of their inverse SRB measure and the entropy of their (forward) SRB measure. 
If $f$ is a special Anosov endomorphism on $\mathbb T^2$, then unstable spaces depend only on their base points, and denote  $$Df_u(x):= Df|_{E^u_f(x)}, Df_s(x):= Df|_{E^s_f(x)}, \  x \in \mathbb T^2.$$ 

\

\textbf{Proof of Theorem \ref{clasAnendos}.} \ 
a) Since $f, g$ are Anosov endomorphisms without critical points, it follows that $f, g$ are TA-covering maps on $\mathbb T^2$. Since $f, g$ are special, it follows from the above Linearization Theorem of \cite{S} that $f, g$ are topologically conjugate to their respective linearization; but as $f_L = g_L$, there exists a topological conjugacy $\Phi: \mathbb T^2 \to \mathbb T^2$ between $f$ and $g$, i.e. $\Phi \circ f = g \circ \Phi$. 

Let $\mu^+_f$ be the SRB measure of $f$ and $\mu_f^-$ be the inverse SRB measure of $f$. Denote $\nu:= \Phi_*\mu_f^+$. Since $\Phi$ is a topological conjugacy, $h_g(\nu) = h_f(\mu_f^+)$. 
As $f$ is a special endomorphism, the unstable space $E^u_f(x)$ depends only on the base point for any  $x \in \mathbb T^2$. 
Also from Pesin formula, $h_f(\mu_f^+) = \chi_u(\mu_f^+) = \int \log |Df_u| d\mu_f^+$.  
Thus from our assumption and since $\nu = \Phi_* \mu_f^+$, we obtain $$h_g(\nu) = h_f(\mu_f^+) = \int \log|Dg_u|\circ \Phi d\mu_f^+ = \int \log |Dg_u| d\nu.$$
But the SRB measure $\mu_g^+$ is the only $g$-invariant probability measure whose entropy is equal to its unstable Lyapunov exponent (\cite{QZ}, \cite{Y}). Therefore, 
\begin{equation}\label{SRBmnug}
\mu_g^+ = \nu = \Phi_* \mu_f^+.
\end{equation}

On the other hand, denote by $\rho:= \Phi_*\mu_f^-$ which is a $g$-invariant ergodic measure on $\mathbb T^2$. Then since $\Phi$ is a topological conjugacy, $h_g^-(\rho) = h_f^-(\mu_f^-)$. Now denote by $D(x)$ the cardinality of the set $f^{-1}(x)$ for $x \in \mathbb T^2$; since $f$ does not have critical points, it follows that $D(\cdot)$ is constant on $\mathbb T^2$ and denote this constant by $D$. Since $f$ and $g$ are topologically conjugate, then $D$ is the cardinality of the set $g^{-1}(x), \forall x\in \mathbb T^2$. It was proved in \cite{M4} that $\mu_f^- = \mathop{\lim}\limits_{n \to \infty} \frac{1}{D^n}  \mathop{\sum}\limits_{z\in f^{-n}(y)} \frac 1n \mathop{\sum}\limits_{i=0}^{n-1} \delta_{f^iz}$, for any $y$ from a set $A\subset \mathbb T^2$ of full Haar measure. This implies that $\rho = \mathop{\lim}\limits_{n \to \infty} \frac{1}{D^n}  \mathop{\sum}\limits_{z'\in g^{-n}(y')} \frac 1n \mathop{\sum}\limits_{i=0}^{n-1} \delta_{g^iz'}$, for $y' \in \Phi(A)$, where $\rho(\Phi(A)) = 1$. Since $g$ has no critical points, it follows that for any set $B\subset \mathbb T^2 $ of sufficiently small diameter, $g$ is injective on $B$. Thus from the above convergence of measures to $\rho$, we obtain that, if $B$ has small diameter and its boundary has $\rho$-measure zero, then $\rho(g(B)) = D \rho(B)$. Hence $J_g(\rho) = D$, thus $F_g(\rho) = \log D$. 
Hence from our assumption in the statement of Theorem \ref{clasAnendos}, and using Theorem \ref{t1} and the above formula for $F_g(\rho)$,  we infer  that $$h_f^-(\mu_f^-) = h_g^-(\rho) = h_g(\rho) - \log D = - \int \log|Dg_s|\circ \Phi \ d\mu_f^-,$$
and therefore since $\rho = \Phi_*\mu_f^-$, we obtain
$$h_g(\rho) = \log D - \int \log |Dg_s| \ d\rho.$$
But then from the uniqueness property of the inverse SRB measures (Theorem 3 of \cite{M4}), we obtain that $\rho = \Phi_* \mu_f^- = \mu_g^-$.  Now if $f$ is Anosov special on $\mathbb T^2$ and not expanding, then $f$ is strongly special (Remark 5.4.1 of \cite{AH}). Hence since $\Phi_*\mu_f^- = \mu_g^-$ and $\Phi_*\mu_f^+ = \mu_g^+$ (from (\ref{SRBmnug})), we apply Theorem A  of \cite{Mi}, to conclude that $\Phi$ is in fact a smooth conjugacy between $f$ and $g$. 

b) Now let $f$ be a special $\mathcal C^\infty$  Anosov endomorphism without critical points on $\mathbb T^2$ and $g = f_L$ be its linearization. Then $Dg_s = \lambda_s, Dg_u = \lambda_u$, where $\lambda_s, \lambda_u$ are the stable/unstable eigenvalues of the matrix of $f_L$. Assuming that $h_f(\mu_f^+) = \log |\lambda_u|$ and $h_f^-(\mu_f^-) = -\log |\lambda_s|$, we obtain that the above conditions of a) are satisfied. Thus from a) it follows that $f$ and $f_L$ are smoothly conjugated. 

$\hfill\square$

\section{Links with inverse topological pressure}\label{linkstop}

In \cite{MU2} there was introduced and studied a notion of  inverse topological entropy (and inverse topological pressure) which is defined using coverings with inverse Bowen balls of type $B^-_n(\hat{x},\varepsilon)$, $x\in\widehat{X}$. In the current Section  we define a generalization of this inverse topological entropy by using covers with inverse Bowen balls along subsets of prehistories in $\widehat{X}$. We prove Theorem \ref{newtop} and Theorem \ref{t2} which relate the inverse entropy of an ergodic measure to the generalized inverse topological entropy. We prove a Partial Variational Principle for this generalized inverse topological entropy in Theorem \ref{varpr}; and a more precise result for special endomorphisms in Corollary \ref{specialcor}. Moreover, in Theorem \ref{VPtor} we establish a Full Variational Principle for inverse entropy for special TA-covering maps of tori (in particular for special Anosov endomorphisms on tori). 

Let again $X$ be a compact metric space, $f:X\rightarrow X$ a continuous map on $X$, $\widehat{X}$ the inverse limit of $(X,f)$, and $\widehat{\mathcal{A}}\subset \widehat{X}$ an arbitrary set of prehistories. For $\varepsilon>0$ define
$$\mathcal{B}^-(\widehat{\mathcal{A}},\varepsilon)=\{B_m^-(\hat{x},\varepsilon) :  m\geq 1, \hat{x}\in\widehat{\mathcal{A}} \}.$$
For a set $B^-=B_m^-(\hat{x},\varepsilon)$ denote $m$ by $n(B)$. Let a subset $Y\subset X$. For $\lambda\geq 0, N\geq 1 $ and $\varepsilon>0$ define
$$m_N^-(\lambda, Y,\widehat{\mathcal{A}},\varepsilon)=\inf\left\{\sum_{B^-\in \mathcal{F}}
e^{-\lambda n(B^-)} :  \mathcal{F}\subset \mathcal{B}^-(\widehat{\mathcal{A}},\varepsilon), n(B^-)\geq N,  \forall B^-\in \mathcal{F}, Y\cap \pi(\widehat{A})\subset  \bigcup_{B^-\in \mathcal{F}} B^-\right\}.$$ 
When $N$ increases, the set of acceptable covers $\mathcal{F}$ becomes smaller and therefore the infimum increases in the above expression. Hence
the limit
$\lim_{N\rightarrow\infty}m_N^-(\lambda, Y,\widehat{\mathcal{A}},\varepsilon)$
exists and will be denoted by $m^-(\lambda, Y,\widehat{\mathcal{A}},\varepsilon)$. Now, let 
$$h^-( Y,\widehat{\mathcal{A}},\varepsilon): =\inf\{\lambda : m^-(\lambda, Y,\widehat{\mathcal{A}},\varepsilon)=0  \}.$$
If $\varepsilon$ decreases to zero, $h^-( Y,\widehat{\mathcal{A}},\varepsilon)$ increases, so the limit $\lim_{\varepsilon\rightarrow 0} h^-( Y,\widehat{\mathcal{A}},\varepsilon)$ exists and is denoted by $h^-(Y,\widehat{\mathcal{A}})$. If $\widehat{Y}\subset\widehat{X}$ and $Y=\pi(\widehat{Y})$ then we denote $h^-(Y, \widehat{Y})$ simply by  $h^-(\widehat{Y})$; or if we want to emphasize the transformation $f$ we write $h^-_f(\widehat{Y})$. 

\begin{proposition}\label{properties}
Let $f: X\rightarrow X$ be a continuous  map on the compact metric space $X$, $\widehat{X}$ be the inverse limit of $(X,f)$. Then the following properties hold: 
\begin{itemize}
	\item [(i)] If $Y\subset X$ and $\widehat{\mathcal{B}}\subset \widehat{\mathcal{A}}\subset \widehat{X}$,
	then $h^-(Y,\widehat{\mathcal{A}})\leq h^-(Y,\widehat{\mathcal{B}})$.
 \item[(ii)] If $Y_1\subset Y_2	\subset X$ and $\widehat{\mathcal{A}}\subset \widehat{X}$, then $h^-(Y_1,\widehat{\mathcal{A}})\le h^-(Y_2,\widehat{\mathcal{A}})$.
 \item[(iii)] If $Y = \bigcup_{p\geq 1} Y_p$ is a countable union of subsets of $X$ and $\widehat{\mathcal{A}}\subset \widehat{X}$, then $h^-(Y,\widehat{\mathcal{A}})= \sup\limits_{p} h^-(Y_p,\widehat{\mathcal{A}})$.
\end{itemize}
\end{proposition}
\begin{proof}We prove only $(iii)$, since the others are straightforward. From $(ii)$, it follows that $h^-(Y,\widehat{\mathcal{A}}) \geq \sup\limits_p h^-(Y_p, \widehat{\mathcal{A}})$. Let $\lambda > \sup\limits_p h^-(Y_p, \widehat{\mathcal{A}})$ and $\varepsilon>0$. Let now $\alpha>0$ be such that $\lambda -\alpha> \sup\limits_p h^-(Y_p, \widehat{\mathcal{A}}, \varepsilon)$. Hence $m^-(\lambda-\alpha, Y_p, \widehat{\mathcal{A}}, \varepsilon)=0$. As $m^-_N(\lambda-\alpha, Y_p, \widehat{\mathcal{A}}, \varepsilon)$ grows with $N$, we have that 
$$m^-_N(\lambda-\alpha, Y_p, \widehat{\mathcal{A}}, \varepsilon)=0 \text{ for every }N>0.$$
If $N$ is fixed, for every $p$ there exists 
$\mathcal{F}_p\subset \mathcal{B}^-(\widehat{\mathcal{A}},\varepsilon)$ such that $n(B^-)\geq N,  \text{ for every } B^-\in \mathcal{F}_p$, $Y_p\subset \bigcup_{B^-\in \mathcal{F}_p} B^-$ and 
$$\sum_{B^-\in \mathcal{F}_p}
e^{-(\lambda - \alpha) n(B^-)} < \frac{1}{2^p}.$$
If $\mathcal{F}=\cup_p \mathcal{F}_p\subset \mathcal{B}^-(\widehat{\mathcal{A}},\varepsilon)$, then $n(B^-)\geq N, \text{ for every } B^-\in \mathcal{F}$. Also $Y= \bigcup_p Y_p \subset  \bigcup_{B^-\in \mathcal{F}} B^-$ and 
$$\sum_{B^-\in \mathcal{F}}
e^{-(\lambda - \alpha) n(B^-)} < 1.$$ 
Hence $m^-_N(\lambda, Y, \widehat{\mathcal{A}}, \varepsilon)<e^{-\alpha N}$.
Thus $m^-(\lambda, Y, \widehat{\mathcal{A}}, \varepsilon)=0$, so $h^-(Y,\widehat{\mathcal{A}}, \varepsilon) \leq\lambda,  \forall \varepsilon>0$. This holds for every $\varepsilon>0$, thus $\lambda\geq h^-(Y,\widehat{\mathcal{A}})$. As $h^-(Y,\widehat{\mathcal{A}}) \leq\lambda$ for every $\lambda >\sup\limits_p h^-(Y_p, \widehat{\mathcal{A}})$, we conclude that   $h^-(Y,\widehat{\mathcal{A}}) = \sup\limits_p h^-(Y_p, \widehat{\mathcal{A}})$. 	
\end{proof}

	
\textbf{Proof of Theorem \ref{newtop}. }
	Recall that $\hat{\mu}$ is $\hat{f}$-invariant and ergodic on $\widehat{X}$. Let us assume that 
	$$\lim_{\delta\rightarrow 0}\left( \sup \{h^-(\widehat{\mathcal{A}}) : \widehat{\mathcal{A}}\subset \widehat{Y},   \hat{\mu}(\widehat{Y}\setminus \widehat{\mathcal{A}})<\delta\}\right)<\alpha.$$	
	Let us choose $\beta' <\alpha$ and $\delta'\in (0, \hat{\mu}(\widehat{Y}))$ such that for every $0<\delta<\delta'$
	\begin{equation}\label{eqq1}
	\sup \{h^-(\widehat{\mathcal{A}}) : \widehat{\mathcal{A}}\subset \widehat{Y},  \hat{\mu}(\widehat{Y}\setminus \widehat{\mathcal{A}})<\delta\} < \beta'.
	\end{equation}
	 Now consider $\beta\in (\beta',\alpha)$ and $\delta\in (0, \delta')$. Define
	$$\widehat{Y}_{k,p}=\left\{\hat{x}\in \widehat{Y} :  \frac{-\log \mu (B^-_n(\hat{x},\varepsilon))}{n}>\beta, \text{ for all }{n\geq p}    \text{ and  for all }\varepsilon \in (0, 1/k]\right\}.$$
	Since $h^-_{f, inf, B}(\mu,\hat{x})\geq \alpha>\beta$ for every $\hat{x}\in \widehat{Y}$, it follows that  $\widehat{Y}=\bigcup\limits_{k,p\geq 1}\widehat{Y}_{k,p}$. Hence for every $\hat{x}\in \widehat{Y}_{k,p}$,
	\begin{equation}\label{eqq2}
	\mu (B^-_n(\hat{x},\varepsilon))<e^{-n\beta}  \text{ for all }{n\geq p}    \text{ and  for all }\varepsilon \in (0, 1/k].
		\end{equation}
	Let $Y_{k,p}=\pi(\widehat{Y}_{k,p})$. Since $\widehat Y$ is defined as an increasing union, we can choose $k$ and $p$ large enough such that $ 
	\hat{\mu}(\widehat{Y}_{k,p})> \hat\mu(\widehat{Y})-\delta$ and $$0<\frac{1}{2}\hat{\mu}(\widehat{Y})<\hat{\mu}(\widehat{Y}_{k,p})\leq\mu (Y_{k,p}).$$
	Then for $\varepsilon>0$ we have from (\ref{eqq1}) that 
	$h^-(Y_{k,p}, \widehat{Y}_{k,p}, \varepsilon)\leq h^-(\widehat{Y}_{k,p})<\beta'. $
	Let $\varepsilon<1/k$ and $N\geq p$. Then there exists a cover $\mathcal{F}_N$ of $Y_{k,p}$ with inverse Bowen balls $B^-$ with $n(B^-)\geq N$ such that 
	$\sum_{B^-\in \mathcal{F}_N} e^{-\beta' n(B^-)}<1$. 
	Then using (\ref{eqq2}) and since $\beta'<\beta$, we obtain
	\begin{align*}
	0<\frac{1}{2}\hat\mu(\widehat{Y})<\mu(Y_{k,p})<&\sum_{B^-\in \mathcal{F}_N}\mu(B^-)< \sum_{B^-\in \mathcal{F}_N}e^{-\beta n(B^-)}=\sum_{B^-\in \mathcal{F}_N} e^{-\beta' n(B^-)+ (\beta'-\beta)n(B^-)}\\
	&\leq \sum_{B^-\in \mathcal{F}_N} e^{-\beta' n(B^-)- (\beta-\beta')N}\leq e^{-(\beta-\beta')N}.
	\end{align*} 
	But this is a contradiction since  $e^{-(\beta-\beta')N}\rightarrow 0$ when $N\rightarrow\infty$. 
	
	$\hfill\square$

We will now prove a covering lemma in the context of an ergodic hyperbolic measure which is special (see Definition \ref{preind}). If $B=B(x,r)$ is a ball in a metric space we denote by $5B$ the ball $B(x,5r)$. It is well known (Besicovitch Covering Theorem) that every family $\mathcal{F}$ of balls of uniformly bounded radius in a compact metric space contains a disjointed subcollection $\mathcal{G}\subset \mathcal{F}$ such that $\bigcup\limits_{B\in \mathcal{F}}B\subset \bigcup\limits_{B\in \mathcal{G}} 5B$. 
If $f:M\rightarrow M$ is $C^2$ endomorphism we will prove a similar result for a family of inverse Bowen balls. 
Before proving the general situation, we give a proof in a more restrictive setting, where we assume that $f$ is conformal on the local stable manifolds of the (non-uniformly) hyperbolic measure. 
 
\begin{lemma}\label{l32}
Let $f: M\rightarrow M$ be a $\mathcal{C}^2$ smooth endomorphism defined on a compact Riemannian manifold and let $\mu$ be an $f$-invariant ergodic hyperbolic and special measure  on $M$ and assume that that $f$ is conformal on the local stable manifolds with respect to $\mu$. For $\varepsilon>0$ let $\widehat{R}_\varepsilon\subset \widehat{M}$ be a Pesin set and $\widehat{Y}\subset \widehat{R}_\varepsilon$ be an arbitrary Borel set. Let $\mathcal{F}$ be a family of inverse Bowen balls $B_n^{-}(\hat{x}, \varepsilon)$ with $\hat{x}\in\widehat{R}_\varepsilon$ such that $\hat{f}^{-n}(\hat{x})\in \widehat{R}_\varepsilon$ for some integers $n\geq 1$ and assume that $\mathcal{F}$ covers $Y=\pi(\widehat{Y})$. Then there exists a subfamily $\mathcal{G}\subset\mathcal{F}$ of mutually disjoint sets such that 
	$$Y\subset \bigcup_{B\in \mathcal{G} } 5B^-, \text{ where  }5B^-= B_n^{-}(\hat{x}, 5\varepsilon)\text{ if }B^-=B_n^{-}(\hat{x}, \varepsilon).$$ 
\end{lemma}
\begin{proof}
	For any inverse Bowen ball $B_n^{-}(\hat{x}, \varepsilon)$ with $\hat{x}\in\widehat{R}_\varepsilon$ and $\hat{f}^{-n}(\hat{x})\in \widehat{R}_\varepsilon$ define  
	$$d_s(B_n^{-}(\hat{x},  \varepsilon))=\text{diam}\left(B_n^{-}(\hat{x},  \varepsilon)\cap W_\varepsilon^s(x)\right),$$ as being the stable diameter of $B_n^{-}(\hat{x},\varepsilon)$. Let $\Omega$ be the collection of subfamilies $\zeta$ of mutually disjoint sets from $\mathcal{F}$ with the following property: if $B_n^{-}(\hat{x}, \varepsilon)\in \mathcal{F}$ intersects a set from $\zeta$, then it intersects a set $B_m^{-}(\hat{y}, \varepsilon)$ from $\zeta$ with 
	$d_s(B_m^{-}(\hat{y},\varepsilon)>\frac{1}{2}d_s(B_n^{-}(\hat{x},\varepsilon))$. Clearly $\Omega$ is partially ordered by inclusion. Let $\mathcal{C}\subset \Omega$ be a totally ordered collection of subfamilies of $\mathcal{F}$. Then $\tilde{\zeta}:= \bigcup_{\zeta\in\mathcal{C}}\zeta$ belongs to $\Omega$. Thus, by Zorn's Lemma there exists a maximal subfamily $\mathcal{G}$ of $\tilde{\zeta}$. Since $d_s(B^-)\leq C\varepsilon$, $B^-\in\mathcal{F}$, for $C>0$ a fixed constant, notice that $\Omega\neq \emptyset$ since $\{B_n^-(\hat{x},\varepsilon )\}\in \Omega$ whenever $B_n^-(\hat{x},\varepsilon )$ has stable diameter larger than $\frac{1}{2}\sup\{ d_s(B^-), B^-\in \mathcal{F}\}<\infty$.

	Now we want to prove that every set from $\mathcal{F}$ intersects at least a set from $\mathcal{G}$. If this were not true, then we can find a set $B^*\in \mathcal{F}$ that does not intersect any set of $\mathcal{G}$ and which has 
	$$d_s(B^*)> \frac{1}{2}\sup \{d_s(B'), \ B'\in \mathcal{F}, \text{ such that }B'\cap B''=\emptyset, \text{ for all }B''\in \mathcal{G}  \}.$$
	Consider $\mathcal{G}'=\mathcal{G}\cup \{ B^*\}$. We want to show that $\mathcal{G}'\in \Omega$. Let an arbitrary element  $\widetilde{B}^-\in \mathcal{F}$. If $\widetilde{B}^-$ intersects some element of $\mathcal{G}$, then since $\mathcal{G}\in\Omega$, we know that there exists $B'\in \mathcal{G}$ such that $d_s(B')>\frac{1}{2} d_s(\widetilde{B}^-)$. Otherwise, $\widetilde{B}^-$ does not intersect any element of $\mathcal{G}$. 
	But then, from the definition of $B^*$ we have $d_s(B^*)>\frac{1}{2}d_s(\widetilde{B}^-)$.	Hence  $\mathcal{G}'=\mathcal{G}\cup \{ B^*\}\in \Omega$, which contradicts the maximality of $\mathcal{G}$. 
	
	Let next $z\in \pi \widehat{R}_\varepsilon$. Since $f$ is conformal on the local stable manifold $W^s_\varepsilon(z)$, it follows that for any inverse Bowen ball $B^-\in\mathcal{G}$, $B^-\cap W^s_\varepsilon(z)$ is a (usual) ball in $W^s_\varepsilon(z)$. Thus from the definition of $\Omega$ and the maximality of $\mathcal{G}$, it follows that $\{5B^- \cap W^s_\varepsilon(z), \  B^-\in \mathcal{G} \}$ covers $W^s_\varepsilon(z) \cap Y$. But now we use the fact that there exist local stable and unstable manifolds of size $\varepsilon$ for any  $\hat y \in \widehat{R}_\varepsilon$ and the fact that the local unstable manifolds depend only on base points since $\mu$ is special. Therefore $\{5B^-, \  B^-\in \mathcal{G} \}$ covers $Y$.
	\end{proof}

\begin{theorem}\label{T33}
	Let $f: M\rightarrow M$ be a $\mathcal{C}^2$ smooth endomorphism defined on a compact Riemannian manifold and let $\mu$ be an $f$-invariant ergodic hyperbolic measure on $M$. Assume that $\mu$ is special and $f$ is conformal on the local stable manifolds with respect to $\mu$. Let $\widehat{Y}\subset \widehat{M}$ be a Borel set and $Y:=\pi(\widehat{Y})$. If $h^{-}_{f,sup,B}(\mu,\hat{x}) \leq \alpha$ for every $\hat{x}\in \widehat{Y}$, then there exists $\widehat Z\subset \widehat Y$ with $\mu(\widehat Y\setminus \widehat Z)=0$ such that $h^{-}(\widehat{Z})\leq \alpha$. 
\end{theorem}

\begin{proof}
	Let $\beta>\alpha$ and $\hat{x}\in \widehat{Y}$.
	Since the expression $\mathop{\limsup}\limits_{n\rightarrow \infty} \frac{-\log\mu (B_n^{-}(\hat{x},\varepsilon))}{n}$ increases to $h^-_{f,sup, B}(\mu,\hat{x})$ when $\varepsilon$ decreases to $0$, and since by assumption $h^-_{f,sup, B}(\mu,\hat{x})\leq \alpha$, we have that
	\begin{equation}\label{sdd}
	\limsup_{n\rightarrow \infty} \frac{-\log (B_n^{-}(\hat{x},\varepsilon))}{n}\leq \alpha<\beta, \text{ for every }\varepsilon>0.
	\end{equation}
	For $\varepsilon>0$ we consider a Pesin set $\widehat{R}_\varepsilon\subset\widehat{M}$. Let $\widehat{Q}_\varepsilon$ be the set of all $\hat{x}\in \widehat{R}_\varepsilon$ with the property that $\hat{f}^{-n}(\hat{x})\in \widehat{R}_\varepsilon$ for infinitely many $n\geq 1$. By Poincar\'e Recurrence Theorem we have that $\hat{\mu}(\widehat{R}_\varepsilon\setminus \widehat{Q}_\varepsilon)=0$. For $\varepsilon >0$ and $p\geq 1$ define the Borel set
	\begin{equation}\label{Ydelta}
	\widehat{Y}_{p}(\varepsilon)=\left\{\hat{x}\in \widehat{Y}: \ \frac{-\log \mu(B_n^{-}(\hat{x},\varepsilon))}{n}< \beta \text{ for all }n\geq p \right\}
	\end{equation}  
	and let 
	$$\widehat{Z}_{p}(\varepsilon)=\widehat{Y}_{p}(\varepsilon)\cap \widehat{Q}_\varepsilon.$$	
Clearly $Z_{p}(\varepsilon)=\pi(\widehat{Z}_{p}(\varepsilon))$ is a Borel set in $M$. Let now $N\geq p$. Notice that for every $\hat{x}\in\widehat{Z}_{p}(\varepsilon) \subset \widehat{Q}_\varepsilon$, $\hat{f}^{-n}(\hat{x})\in \widehat{R}_\varepsilon$ for infinitely many $n\geq 1$. Let $\mathcal{F}$ be the set of all inverse Bowen balls of the form $B_n^-(\hat{x},\varepsilon)$, where $\hat{x}\in \widehat{Z}_{p}(\varepsilon)$ and $\hat{f}^{-n}(\hat{x})\in \widehat{R}_\varepsilon$. Clearly $\mathcal{F}$ covers $Z_{p}(\varepsilon)$. By Lemma \ref{l32} we obtain then a subcollection $\mathcal{G}\subset \mathcal{F}$ of mutually disjoint sets of the form $B_n^-(\hat{x},\varepsilon)$ with $n\geq N$ and $\hat{x}\in \widehat{Z}_{p}(\varepsilon)$ such that $\{5B^- : B^-\in\mathcal{G}\}$ covers $Z_{p}(\varepsilon)$, where  $5B^-= B_n^{-}(\hat{x}, 5\varepsilon)$ if $B^-=B_n^{-}(\hat{x},\varepsilon)$. For a set $B^-$ of the form $B_n^{-}(\hat{x},\varepsilon)$ we denote $n$ by $n(B^-)$. 
	Then, from (\ref{Ydelta}) and since $\mathcal{G}$ is a disjoint family, we obtain
	\begin{equation}\label{betan}
	\sum_{B^-\in \mathcal{G}} e^{-\beta n(B^-)}\leq \sum_{B^-\in \mathcal{G}}\mu(B^-)\leq 1.
	\end{equation}
	Let now an arbitrary $\beta'$ with $\beta'>\beta$. Then from (\ref{betan}),
	\begin{align*}
	\sum_{B^-\in \mathcal{G}} e^{-\beta' n(B^-)} & = \sum_{B^-\in \mathcal{G}} e^{-\beta n(B^-)} e^{-(\beta'-\beta) n(B^-)} \leq \sum_{B^-\in \mathcal{G}} e^{-\beta n(B^-)} e^{-(\beta'-\beta) N}\leq  e^{-(\beta'-\beta) N} \underset{N\rightarrow\infty}\longrightarrow 0.
	\end{align*}
	Thus $h^-(Z_{p}(\varepsilon), \widehat{Z}_p(\varepsilon), 5\varepsilon)\leq \beta'$ for every $p\geq 1$ and every $\beta'>\beta>\alpha$.  Notice that $\widehat{Y}_p(\varepsilon) \subset \widehat{Y}_{p+1}(\varepsilon)$ and by (\ref{sdd}) we have $\widehat{Y}=\bigcup_{p=1}^\infty \widehat{Y}_{p}(\varepsilon)$. Let $\widehat{Z}= \bigcup_{\varepsilon>0}(\widehat{Y}\cap \widehat{Q}_\varepsilon)=\bigcup_{\varepsilon>0}\bigcup_{p\geq 1}\widehat{Z}_p(\varepsilon)$. Then, since $\widehat{Z}_p(\varepsilon) \subset \widehat{Z}$, by Proposition \ref{properties}, we have
	\begin{equation}\label{vhpe}
		h^-(Z_p(\varepsilon), \widehat{Z}, 5\varepsilon)\leq h^-(Z_{p}(\varepsilon), \widehat{Z}_p(\varepsilon), 5\varepsilon)\leq \beta',
	\end{equation}
	 for every $p\geq 1$ and every $\varepsilon>0$. Notice that $\widehat{Y}_p(\varepsilon) \subset \widehat{Y}_{p+1}(\varepsilon)$ and by (\ref{sdd}) we have $\widehat{Y}=\bigcup_{p=1}^\infty \widehat{Y}_{p}(\varepsilon)$. 	 
	  Since $\widehat{R}_\varepsilon\subset \widehat{R}_{\varepsilon'}$ if $\varepsilon'<\varepsilon$ it follows that $\widehat{Q}_\varepsilon\subset \widehat{Q}_{\varepsilon'}$ if $\varepsilon'<\varepsilon$. As $\widehat{Y}=\bigcup_{p=1}^\infty\widehat{Y}_p(\varepsilon)$ for every $\varepsilon>0$, it follows that for every $n\geq 1$, $$\widehat{Z}=\bigcup_{q=n}^\infty(\widehat{Y}\cap  \widehat{Q}_\frac{1}{q})=\bigcup_{q=n}^\infty(\bigcup_{p=1}^\infty\widehat{Y}_p(\frac{1}{q})\cap  \widehat{Q}_\frac{1}{q})=\bigcup_{q=n}^\infty \bigcup_{p=1}^\infty \widehat{Z}_p(\frac{1}{q}).$$
	  Now as $\bigcup_\varepsilon \widehat{R}_\varepsilon= \widehat{M}$ up to a set of $\hat{\mu}$-measure zero and $\hat{\mu}(\widehat{R}_\varepsilon\setminus \widehat{Q}_\varepsilon)=0$, it follows that  
	 $\hat{\mu}(\widehat{Y}\setminus \widehat{Z})=0$. Hence, if we define $Z:=\pi(\widehat{Z})$ then $Z=\bigcup\limits_{q=n}^\infty\bigcup\limits_{p=1}^\infty Z_{p}(\frac{1}{q})$, for every $n\geq 1$. Thus by Proposition \ref{properties} and (\ref{vhpe}) we obtain for every $n\geq 1$,  
	$$
	 h^-\big(Z, \widehat{Z}, \frac{5}{n}\big) = \sup \{ h^-\big(Z_p(\frac{1}{q}), \widehat{Z}, \frac{5}{n}\big), \ p\geq 1, q\geq n \}  \leq \sup\{ h^-\big(Z_p(\frac{1}{q}), \widehat{Z}, \frac{5}{q}\big), p\geq 1, q\geq n \} \leq  \beta'.	 
	 $$
	 Then since $\beta,\beta'$ are arbitrary with $\beta'>\beta>\alpha$ it follows from above that $h^-(Z, \widehat{Z}, \frac{5}{n})\leq \alpha$, for every $n\geq 1$. Therefore $h^-(Z,\widehat{Z}) = h^-(\widehat Z) \leq \alpha$.
\end{proof}

Now we study the case when $f$ is not necessarily conformal on local stable manifolds. Let $f: M\rightarrow M$ be a $\mathcal{C}^2$ smooth endomorphism on a compact Riemannian manifold and let $\mu$ be an $f$-invariant ergodic hyperbolic and special measure on $M$. Let us assume now that the endomorphism $f$ is not necessarily conformal on local stable manifolds over $M$. Recall that $\hat{\mu}$ denotes the unique ergodic $\hat{f}$-invariant measure on $\widehat{M}$ such that $\pi_*\hat{\mu}=\mu$. Since $\hat{\mu}$ is ergodic, from Oseledec Theorem there exists a set $\hat{E}\subset \widehat{M}$ of $\hat{\mu}$-measure equal to $1$ and $k\geq 1$, such that for every  $\hat{x}\in\hat{E}$ there exist vector subspaces $E^s_{i,x_{-n}}$, $1\leq i\leq k$ and Lyapunov exponents of $\mu$, $\lambda_{s,i}<0$, $1\leq i\leq k$ (since $\hat{\mu}$ is ergodic). Without loss of generality we assume that there exists only two negative Lyapunov exponents, so $k=2$. Thus for any $\delta>0$ and $p\geq 1$ define the Borel set 
\begin{equation}\label{eq1}
\widehat{D}^s_p(\delta)=\left\{\hat{x}\in \widehat{M}:  e^{(\lambda_{s,i}-\delta)n} < \left\|Df^{n} |_{E^s_{i, x_{-n}}} \right\| <  e^{(\lambda_{s,i}+\delta)n} , \forall   i\in \{1,2\},\forall n\geq p\right\}.
\end{equation}
We have  $\lim\limits_{p\rightarrow\infty}\hat{\mu}( \widehat{D}^s_p(\delta))=1.$
Since $\mu$ is hyperbolic, for any $\varepsilon>0$ there exists a Pesin set $\widehat{R}_\varepsilon\subset\widehat{M}$ for $\mu$, (see \cite{BP}). If a set of the form $B_n^-(\hat{x},\varepsilon)\cap W^s_\varepsilon(y)$, with $y\in \pi \widehat{R}_\varepsilon
$ and such that $\hat{x}\in \widehat{D}^s_p(\delta)$ and $\hat{f}^{-n}(\hat{x})\in \widehat{R}_\varepsilon$ for some $n\geq p$ is nonempty, then this set is almost a rectangle whose sides in the two stable directions are $l_1(\hat{x},n, \varepsilon)$ and $l_2(\hat{x},n, \varepsilon)$ and there exists a constant $C(\varepsilon)>0$ independent of $\hat{x}$ and $n$, such that
$$ \frac{\varepsilon}{C(\varepsilon)}\cdot e^{(\lambda_{s,i}-\delta)n} < l_i(\hat{x}, n, \varepsilon) <  C(\varepsilon)\varepsilon\cdot e^{(\lambda_{s,i}+\delta)n}, \  i=1,2.$$
Let 	$$ K(\delta): =\left( \frac{-\lambda_{s,1}+\delta}{-\lambda_{s,1}-\delta} + \delta \right)^{-1}.$$
Since $\lim\limits_{\delta\rightarrow 0} K(\delta)=1$ for all $\delta$ sufficiently small we have $K(\delta)>1/2$. Let also
\begin{equation}\label{Adelta}
A(\delta):= K(\delta)\cdot\left( \min\left\{ \frac{-\lambda_{s,i}-\delta}{-\lambda_{s,i}+\delta} :  i=1,2 \right \} -\delta\right) .
\end{equation}
Clearly $0<A(\delta)< 1$ and  $\lim\limits_{\delta\rightarrow 0} A(\delta)=1$. With this notation we have: 
\begin{lemma}\label{l1}	
	Let $\delta, \varepsilon>0$ arbitrarily small and $p(\delta,\varepsilon)>1$ such that $p(\delta,\varepsilon)\cdot( -\lambda_{s,i}+\delta)\delta>\log 2+ 2\log C(\varepsilon)$ and then let $m,n,p$ arbitrary with $m,n>2p$ and $p\geq p(\delta,\varepsilon)$. Let $y\in \pi \widehat{R}_\varepsilon$ and $\hat{x}, \hat{z}\in \widehat{D}^s_p(\delta)$ such that $\hat{f}^{-n}(\hat{x}), \hat{f}^{-m}(\hat{z})\in \widehat{R}_\varepsilon$. If
	$B_n^-(\hat{x},\varepsilon)\cap W^s_\varepsilon(y)$ 
	intersects $B_m^{-}(\hat{z},\varepsilon)\cap   W^s_\varepsilon(y)$ and 
	$l_1(\hat{z},m, \varepsilon)>\frac{1}{2}\cdot 
	l_1(\hat{x},n, \varepsilon)$ then  $B_n^{-}(\hat{x},\varepsilon) \cap W^s_\varepsilon(y)\subset B^-_{[m A(\delta)]}(\hat{z}, 5\varepsilon)\cap  W^s_\varepsilon(y)$, where $[m A(\delta)]$ denotes the integer part of $m A(\delta)$. 
\end{lemma}	
\begin{proof}	
	With the above notation, for $i=1,2$ we have
	$$ \frac{\varepsilon}{C(\varepsilon)}\cdot e^{(\lambda_{s,i}-\delta)n} < l_1(\hat{x}, n, \varepsilon) <  C(\varepsilon)\varepsilon\cdot e^{(\lambda_{s,i}+\delta)n},$$
	$$ \frac{\varepsilon}{C(\varepsilon)}\cdot e^{(\lambda_{s,i}-\delta)m} < l_1(\hat{z}, m, \varepsilon) <  C(\varepsilon)\varepsilon\cdot e^{(\lambda_{s,i}+\delta)m}.$$
	Now, since $l_1(\hat{z},m, \varepsilon)>\frac{1}{2}\cdot 
	l_1(\hat{x},n, \varepsilon)$, we have
	$C(\varepsilon)e^{(\lambda_{s,1}+\delta)m}> \frac{1}{2C(\varepsilon)}e^{(\lambda_{s,1}-\delta)n}$,
	and then 
	$$(-\lambda_{s,1}-\delta)m < (-\lambda_{s,1}+\delta)n + \log 2 +2\log C(\varepsilon).$$
	Hence, since $m,n>2p(\delta,\varepsilon)$ and $p(\delta,\varepsilon)$ is chosen as in the statement, we obtain
	$$\frac{m}{n}< \frac{-\lambda_{s,1}+\delta}{-\lambda_{s,1}-\delta}+ \frac{2\log C(\varepsilon)+\log 2}{n( -\lambda_{s,1}+\delta)}< \frac{-\lambda_{s,1}+\delta}{-\lambda_{s,1}-\delta} + \delta.$$
	Therefore 
	$m\cdot K(\delta)< n$. As $K(\delta)>1/2$, for $i=1,2$ we have
	\begin{align*}
	(-\lambda_{s,i}+\delta)&\cdot A(\delta)m \leq     \left( \frac{-\lambda_{s,i}-\delta}{-\lambda_{s,i}+\delta}-\delta\right) \cdot(-\lambda_{s,i}+\delta) K(\delta) m\\&=   (-\lambda_{s,i}-\delta)\cdot m K(\delta) - \delta (-\lambda_{s,i}+\delta)m K(\delta)< (-\lambda_{s,i}-\delta)n+\log 2-  2\log C(\varepsilon).
	\end{align*}
	Thus
	$-\log C(\varepsilon)+ (\lambda_{s,i}-\delta)\cdot A(\delta)m > (\lambda_{s,i}+\delta)n-\log 2+\log C(\varepsilon)$			
	and then 
	$$ \frac{\varepsilon}{C(\varepsilon)}\cdot e^{(\lambda_{s,i}-\delta)A(\delta)m}>  \frac{1}{2}\cdot  C(\varepsilon)\varepsilon \cdot e^{(\lambda_{s,i}+\delta)n}, \ \ \  i=1,2.$$
	Using the assumption $n>2p$ from the statement, it follows that 
	$$B^-_n(\hat{x},\varepsilon)\cap W^s_\varepsilon(y)\subset B^-_{[mA(\delta)]}(\hat{z},5\varepsilon)\cap W^s_\varepsilon(y).$$	
\end{proof}

Recall that $\mu$ is a hyperbolic and special measure on $\Lambda$. With the same notation as above we prove now the following lemma: 

\begin{lemma}\label{lx}
	Let $\widehat{Y}\subset \widehat{R}_\varepsilon$ be a Borel set. Let $\mathcal{F}$ be a family of sets of the form $B_n^{-}(\hat{x}, \varepsilon)$ with $\hat{x}\in \widehat{D}^s_p(\delta)\cap \widehat{R}_\varepsilon$ and such that $\hat{f}^{-n}(\hat{x})\in\widehat{R}_\varepsilon$ for some $n\geq 2p$ (where $p\geq p(\delta,\varepsilon)$ as in Lemma \ref{l1}) and assume that $\mathcal{F}$ covers $Y=\pi(\widehat{Y})$. Then there exists a subfamily $\mathcal{G}\subset\mathcal{F}$ of mutually disjoint sets such that  
	\begin{equation*}
		Y\subset \bigcup\limits_{B^- \in \mathcal{G} }5B^-, \text{ where  }5B^-= B_{[nA(\delta)]}^{-}(\hat{x}, 5\varepsilon)\text{ if  }B^-=B_n^{-}(\hat{x}, \varepsilon).
		\end{equation*}
\end{lemma}
\begin{proof}
	We recall that for any set of type $B_n^{-}(\hat{x}, \varepsilon)$ with $\hat{x}\in \widehat{D}^s_p(\delta)\cap \widehat{R}_\varepsilon$ and such that $\hat{f}^{-n}(\hat{x})\in\widehat{R}_\varepsilon$ we denoted by $l_1(\hat{x}, n, \varepsilon)$ and 	$l_2(\hat{x}, n, \varepsilon)$ the diameters in the stable directions of $B_n^{-}(\hat{x},\varepsilon)\cap W^s_\varepsilon (x)$. For the inverse ball $B^-= B_n^{-}(\hat{x}, \varepsilon)$ we also denote $l_i(\hat{x}, n, \varepsilon)$ by $l_i(B^-)$, for $i=1,2$. Let $\Omega$ be the collection of subfamilies $\zeta$ of mutually disjoint sets from $\mathcal{F}$ with the following property: if $B_n^{-}(\hat{x}, \varepsilon)\in \mathcal{F}$ intersects a set from $\zeta$, then it intersects a set $B_m^{-}(\hat{y}, \varepsilon)$ from $\zeta$ with 
	$l_1(\hat{y}, m, \varepsilon)> \frac{1}{2}l_1(\hat{x}, n, \varepsilon)$. Clearly $\Omega$ is partially ordered by inclusion. Let $\mathcal{C}\subset \Omega$ be a totally ordered collection of subfamilies of $\mathcal{F}$. Then $\tilde{\zeta}= \bigcup_{\zeta\in\mathcal{C}}\zeta$ belongs to $\Omega$. Thus, by Zorn's Lemma there exists a maximal subfamily $\mathcal{G}$ of $\tilde{\zeta}$. Notice that $\Omega\neq \emptyset$ since $\{B^-=B_n^-(\hat{x},\varepsilon )\}\in \Omega$ whenever
	$l_1(B^-)> \frac{1}{2}\sup\{ l_1(B^-), B^-\in \mathcal{F}\}$.

	Now, we want to prove that every set from $\mathcal{F}$ intersects at least a set from $\mathcal{G}$. If this were not true, then we can find a set $B^*\in \mathcal{F}$ that does not intersect any set of $\mathcal{G}$ and which has 	$$l_1(B^*)> \frac{1}{2}\sup \{l_1(B'), \ B'\in \mathcal{F}, \text{ such that }B'\cap B''=\emptyset, \text{ for all }B''\in \mathcal{G}  \}.$$
	Consider $\mathcal{G}'=\mathcal{G}\cup \{ B^*\}$. We want to show that $\mathcal{G}'\in \Omega$. Let an arbitrary element $\widetilde B^-\in\mathcal{F}$. If $\widetilde B^-$ intersects some element of $\mathcal{G}$, then since $\mathcal{G}\in\Omega$, we know that there exists $B'\in \mathcal{G}$ such that $l_1(B')>\frac{1}{2} l_1(\widetilde B^-)$. Otherwise, $\widetilde B^-$ does not intersect any element of $\mathcal{G}$. 
	But then, from the definition of $B^*$ we have $l_1(B^*)>\frac{1}{2}l_1(\widetilde B^-)$.
	Hence  $\mathcal{G}'=\mathcal{G}\cup \{ B^*\}\in \Omega$, which contradicts the maximality of $\mathcal{G}$. 

Let now $z\in \pi \widehat{R}_\varepsilon$. Hence for any set $B^-\in\mathcal{G}$, $B^- \cap   W_\varepsilon^s(z)$ is almost a rectangle with sides $l_1(\hat{x},n, \varepsilon)$ and $l_2(\hat{x},n, \varepsilon)$ in the two stable directions. Thus from the definition of $\Omega$ and the maximality of $\mathcal{G}$ it follows that $\{5B^- \cap W^s_\varepsilon(z), \  B^-\in \mathcal{G} \}$ covers $W^s_\varepsilon(z) \cap Y$. Here we use Lemma \ref{l1} and the fact that $A(\delta)$ is necessary for the estimate of $l_2(B^-)$. But now we use that there exist local stable and unstable manifolds of size $\varepsilon$ over $\widehat{R}_\varepsilon$ and that the local unstable manifolds depend only on their respective base points, since $\mu$ is special. Thus $\{5B^-, \  B^-\in \mathcal{G} \}$ covers $Y$.
\end{proof}


 \textbf{Proof of Theorem \ref{t2}. } 
As in the previous lemma, we assume without loss of generality, that there exists only two negative Lyapunov exponents $\lambda_{s,1}$ and $\lambda_{s,2}$.  Let $\delta>0$ be such that $A(\delta), K(\delta)>1/2$ for every $0<\delta<\delta_0$. Let  $p\geq 1$. Recall that 
\begin{equation*}
\widehat{D}^s_p(\delta)=\left\{\hat{x}\in \widehat{M}:  e^{(\lambda_{s,i}-\delta)n} < \left\|Df^{n} |_{E^s_{i, x_{-n}}} \right\| <  e^{(\lambda_{s,i}+\delta)n} , \forall   i\in \{1,2\},\forall n\geq p\right\}.
\end{equation*}
Then since the lift $\hat{\mu}$ of $\mu$ is ergodic, we have
\begin{equation}\label{Ddelta}
\widehat{D}^s_p(\delta)\subset \widehat{D}^s_{p+1}(\delta)\text{ and } \hat\mu(\widehat{M}\setminus\bigcup_{p\geq 1}\widehat{D}^s_p(\delta))=0.
\end{equation}
Let  
$$\widehat{Y}(\delta)=\bigcup_{p\geq 1}(\widehat{D}^s_p(\delta) \cap \widehat{Y}).$$
Then by (\ref{Ddelta}), we have $\hat{\mu}(\widehat{Y}\setminus\widehat{Y}(\delta))=0$. Let $\beta>\alpha$.
Since the expression $\mathop{\limsup}\limits_{n\rightarrow \infty} \frac{-\log\mu (B_n^{-}(\hat{x},\varepsilon))}{n}$ increases to $h^-_{f,sup, B}(\mu,\hat{x})$ when $\varepsilon$ decreases to $0$, and since by assumption $h^-_{f,sup, B}(\mu,\hat{x})\leq \alpha$, we obtain that for every $\hat{x}\in \widehat{Y}$,
\begin{equation}\label{sd}
\limsup_{n\rightarrow \infty} \frac{-\log \mu (B_n^{-}(\hat{x},\varepsilon))}{n}<\beta, \text{ for every }\varepsilon>0.
\end{equation}
	Recall that for $\varepsilon>0$,  $\widehat{R}_\varepsilon\subset\widehat{M}$ is a Pesin set. Let $\widehat{Q}_\varepsilon$ be the set of all $\hat{x}\in \widehat{R}_\varepsilon$ with the property that $\hat{f}^{-n}(\hat{x})\in \widehat{R}_\varepsilon$ for infinitely many $n\geq 1$. By Poincar\'e Recurrence Theorem we have that $\hat{\mu}(\widehat{R}_\varepsilon\setminus \widehat{Q}_\varepsilon)=0$. For $\varepsilon >0$ and $p\geq 1$ define the Borel set
\begin{equation}\label{Ypdelta}
\widehat{Y}_{p}(\delta, \varepsilon)=\left\{\hat{x}\in \widehat{D}^s_p(\delta) \cap \widehat{Y} \ | \ \frac{-\log \mu(B_n^{-}(\hat{x},\varepsilon))}{n}<\beta, \text{for all }n\geq p \right\}.
\end{equation}
Notice that $\widehat{Y}_p(\delta, \varepsilon) \subset \widehat{Y}_{p+1}(\delta,\varepsilon)$ and by (\ref{sd}), we have that $\widehat{Y}(\delta)=\bigcup\limits_{p=1}^\infty\widehat{Y}_{p}(\delta,\varepsilon)$. Let
$$\widehat{Z}_{p}(\delta, \varepsilon)=\widehat{Y}_{p}(\delta,\varepsilon)\cap\widehat{Q}_\varepsilon.$$Clearly, $Z_{p}(\delta,\varepsilon):=\pi(\widehat{Z}_{p}(\delta,\varepsilon))$ is a Borel set in $M$. Let now $N\geq p$. Notice that for every $\hat{z}\in\widehat{Z}_{p}(\delta,\varepsilon) \subset \widehat{Q}_\varepsilon$, we have $\hat{f}^{-n}(\hat{z})\in \widehat{R}_\varepsilon$ for infinitely many $n\geq 1$. Then, by Lemma \ref{lx}, we obtain a collection $\mathcal{G}$ of disjoint sets of the form $B_n^-(\hat{x},\varepsilon)$, $\hat{x}\in \widehat{Z}_{p}(\delta,\varepsilon)$ with $n\geq 2N+2$, such that $\{5B^- : B^-\in\mathcal{G}\}$ covers  $Z_{p}(\delta,\varepsilon)$, where  $5B^-= B_{[nA(\delta)]}^{-}(\hat{x}, 5\varepsilon)$ if $B^-=B_n^{-}(\hat{x},\varepsilon)$. For a set $B^-$ of the form $B_n^{-}(\hat{x},\varepsilon)$ we denote $n$ by $n(B^-)$; thus $n(5B^-)=[nA(\delta)]$. Notice that $\beta n(5B^-)/A(\delta)=\beta[A(\delta)n(B^-)]/A(\delta)> \beta  n(B^-)- 2\beta$, since $A(\delta)>\frac{1}{2}$. Then from (\ref{Ypdelta}), and since $\mathcal{G}$ is a disjoint family, we obtain
\begin{equation}\label{ebeta}
\sum_{B^-\in \mathcal{G}} e^{-\beta n(5B^-)/A(\delta)}\leq \sum_{B^-\in \mathcal{G}} e^{-\beta n(B^-)} \cdot e^{2\beta}\leq
\sum_{B^-\in \mathcal{G}}\mu(B^-)\cdot  e^{2\beta}\leq  e^{2\beta}.
\end{equation}
Since  $A(\delta)>1/2$, we have  $n(5B^-)>N$ for every $B^-\in \mathcal{G}$. Consider now $\beta'$ arbitrary with $\beta'>\beta$. Then from (\ref{ebeta}),
\begin{align*}
\sum_{B^-\in \mathcal{G}} e^{-\beta' n(5B^-)/A(\delta)} & = \sum_{B^-\in \mathcal{G}} e^{-\beta n(5B^-)/A(\delta)} e^{-(\beta'-\beta) n(5B^-)/A(\delta)} \\ &\leq  e^{-(\beta'-\beta) N}\cdot\sum_{B^-\in \cdot\mathcal{G}} e^{-\beta n(5B^-)/A(\delta)}\leq e^{2\beta}\cdot e^{-(\beta'-\beta) N} \overset{N\rightarrow \infty}\longrightarrow 0.
\end{align*}
Thus $h^-(Z_{p}(\delta,\varepsilon), \widehat{Z}_p(\delta,\varepsilon), 5\varepsilon)\leq \beta'/A(\delta)$ for every $p\geq 1$ and every $\beta'>\beta>\alpha$. Let $\widehat{Z}(\delta)= \bigcup_{\varepsilon>0}(\widehat{Y}(\delta)\cap \widehat{Q}_\varepsilon)$. 
Then, since $\widehat{Z}_p(\delta,\varepsilon) \subset \widehat{Z}(\delta)$, by Proposition \ref{properties}, for every $p$ and $\varepsilon>0$, we have 
\begin{equation}\label{hzp}
h^-(Z_p(\delta,\varepsilon), \widehat{Z}(\delta), 5\varepsilon)\leq h^-(Z_{p}(\delta,\varepsilon), \widehat{Z}_p(\delta,\varepsilon), 5\varepsilon)\leq \beta'/A(\delta).	
\end{equation}

 If $0<\varepsilon'<\varepsilon$, then since $\widehat{R}_\varepsilon\subset \widehat{R}_{\varepsilon'}$, it follows that $\widehat{Q}_\varepsilon\subset \widehat{Q}_{\varepsilon'}$. As   $\widehat{Y}(\delta)=\bigcup_{p=1}^\infty\widehat{Y}_p(\delta,\varepsilon)$ for every $\varepsilon>0$, it follows that for every $n\geq 1$, $$\widehat{Z}(\delta)=\bigcup_{q=n}^\infty(\widehat{Y}(\delta)\cap  \widehat{Q}_\frac{1}{q})=\bigcup_{q=n}^\infty(\bigcup_{p=1}^\infty\widehat{Y}_p(\delta,\frac{1}{n})\cap  \widehat{Q}_\frac{1}{q})=\bigcup_{q=n}^\infty \bigcup_{p=1}^\infty \widehat{Z}_p(\delta,\frac{1}{q}).$$ Now as $\bigcup_{\varepsilon>0} \widehat{Q}_\varepsilon= \widehat{M}$ up to a $\hat{\mu}$-null set, and  $\hat{\mu}(\widehat{Y}\setminus\widehat{Y}(\delta))=0$, it follows that  
$\hat{\mu}(\widehat{Y}\setminus \widehat{Z}(\delta))=0$. Hence if  $Z(\delta):=\pi(\widehat{Z}(\delta))$, then 	 $Z(\delta)=\bigcup\limits_{q=n}^\infty\bigcup\limits_{p=1}^\infty Z_p(\delta,\frac{1}{q})$, $\forall n\geq 1$. So by  Proposition \ref{properties} and (\ref{hzp}), for $\forall n\geq 1$,  
\begin{align*}
h^-\big(Z(\delta), \widehat{Z}(\delta), \frac{5}{n}\big) &= \sup \{ h^-\big(Z_p(\delta,\frac{1}{q}), \widehat{Z}(\delta), \frac{5}{n}\big), \ p 
\geq 1, q\geq n \} \\ &\leq \sup \{ h^-\big( Z_p(\delta,\frac{1}{q}), \widehat{Z}(\delta), \frac{5}{q}\big), \ p\geq 1, q\geq n \} \leq  \beta'/A(\delta).
\end{align*}	
Then since $\beta,\beta'$ are arbitrary with $\beta'>\beta>\alpha$, it follows from above that $h^-(Z(\delta), \widehat{Z}(\delta), \frac{5}{n})\leq \alpha/A(\delta)$, for every $n\geq 1$. Thus $h^-(\widehat{Z}(\delta))=h^-(Z(\delta),\widehat{Z}(\delta))\leq\alpha/A(\delta)$, where $\widehat{Z}(\delta)\subset \widehat{Y}$ and $\hat{\mu}(\widehat{Y}) =\hat{\mu}(\widehat{Z}(\delta))$.
Hence $\inf\limits_{\delta>0}  h^-(\widehat Z(\delta))\le \alpha$ and the conclusion of the theorem follows.


$\hfill\square$

\noindent \textbf{Theorem 1.9.}\textit{(Partial Variational Principle for inverse entropy).
	Let $f: M\rightarrow M$ be a $\mathcal{C}^2$ smooth endomorphism on a manifold $M$. Then,	
	\begin{align*}
		\sup \{ \inf \{h^-(\widehat Z), \    \hat\mu(\widehat Z)=1 \}, \  & \mu \text{ hyperbolic ergodic and special} \} \
		\leq\sup \{h^-_{f,inf, B}(\mu),   \mu \text{ ergodic} \} \\
		& \leq\lim_{\delta\rightarrow 0} \left(\sup\{h^-(\widehat{\mathcal{A}}), \ \hat{\mu}(\widehat{\mathcal{A}})>1-\delta, \ \mu \text{  ergodic} \}\right).
	\end{align*}}
	\begin{proof}
Let $\mu$ be an arbitrary  probability measure on $M$ which is $f$-invariant ergodic hyperbolic and special. Since $\mu$ is ergodic there exists a Borel set $\widehat{Y}\subset\widehat{M}$ with $\hat{\mu}(\widehat{Y})=1$ and such that $h^{-}_{f,sup,B}(\mu,\hat{x})= h^{-}_{f,sup,B}(\mu)$ for every $\hat{x}\in \widehat{Y}$. Then by Theorem \ref{t2}, we have
$\inf \{ h^-(\widehat Z), \  \widehat Z \subset \widehat Y,  \hat \mu (\widehat Z ) =\hat \mu (\widehat Y )\}\le h^{-}_{f,sup,B}(\mu)$ and thus $\inf \{ h^-(\widehat Z), \  \hat{\mu}(\widehat Z)=1 \}\le h^{-}_{f,sup,B}(\mu)$.
From Theorem \ref{thpreind} we know that if $\mu$ is ergodic hyperbolic and special, then $h^-_{f,inf, B}(\mu)=h^-_{f,sup, B}(\mu)$. 
Therefore,  
\begin{align*}
\sup \Big\{ \inf \{h^-(\widehat{Z}), \ \hat\mu(\widehat{Z})=1 \}, \  & \mu \text{ ergodic and special} \Big\} \leq \\
\leq\sup \{h^-_{f,sup, B}(\mu),   \mu \text{ ergodic and special} \} &=\sup \{h^-_{f,inf, B}(\mu),   \mu \text{ ergodic and special}\}.
\end{align*}
Next it is clear that,
$\sup \{h^-_{f,inf, B}(\mu),  \mu \text{ ergodic and special} \}\leq \sup \{h^-_{f,inf, B}(\mu),   \mu \text{ ergodic} \}.$
Hence from the above displayed inequalities it follows that,
\begin{align}\label{varpppp}
	\sup \Big\{ \inf \{h^-(\widehat{Z}),  \ & \hat{\mu}(\widehat{Z})=1 \}, \   \mu \text{ ergodic and special} \Big\} \leq \\
	&\leq\sup \{h^-_{f,inf, B}(\mu),   \mu \text{ ergodic} \}.\nonumber
\end{align}
Let $\nu$ be an arbitrary  probability measure on $M$ which is $f$-invariant and ergodic. Let $\widehat{Y}\subset\widehat{M}$ be such that $\hat{\nu}(\widehat{Y})=1$ and $h^{-}_{f,inf,B}(\nu,\hat{x})= h^{-}_{f,inf,B}(\nu)$ for every $\hat{x}\in \widehat{Y}$. Then by Theorem \ref{newtop} we obtain:
\begin{align*}
h^-_{f,inf, B}(\nu) & \leq\lim_{\delta\rightarrow 0} \left(\sup\{h^-(\widehat{\mathcal{A}}), \widehat{\mathcal{A}}\subset \widehat{Y},  \hat{\nu}(\widehat{\mathcal{A}})>1-\delta\}\right) \leq \lim_{\delta\rightarrow 0} \left(\sup\{h^-(\widehat{\mathcal{A}}),   \hat{\nu}(\widehat{\mathcal{A}})>1-\delta\}\right) \\
& \leq\lim_{\delta\rightarrow 0} \left(\sup\{h^-(\widehat{\mathcal{A}}), \ \hat{\mu}(\widehat{\mathcal{A}})>1-\delta, \ \mu \text{  ergodic} \}\right),
\end{align*}
where the last supremum is taken over all ergodic measures $\mu$. 
As $\nu$ was arbitrary, we conclude that
\begin{align*}
\sup \{h^-_{f,inf, B}(\mu),   \mu \text{ ergodic} \} 
\leq\lim_{\delta\rightarrow 0} \left(\sup\{h^-(\widehat{\mathcal{A}}), \ \hat{\mu}(\widehat{\mathcal{A}})>1-\delta, \ \mu \text{  ergodic} \}\right).
\end{align*}
Therefore from the last displayed inequality and (\ref{varpppp}), we obtain the conclusion of the Theorem. 
\end{proof}

In case $f$ is a special hyperbolic endomorphism on $\Lambda$ (Definition  \ref{defspecial}), then any $f$-invariant ergodic measure on $\Lambda$ is hyperbolic and special. Examples of special endomorphisms are toral endomorphisms, certain skew product endomorphisms, etc (see \cite{MF, Mi}).

\begin{corollary}\label{specialcor}
Let $f$ be a $\mathcal{C}^2$ endomorphism on a Riemannian manifold $M$, so that $f$ is hyperbolic and special on a  compact invariant set $\Lambda\subset M$. Then, 
\begin{align*}
		\sup \Big\{ \inf \{h^-(Z,\widehat{\Lambda}),   \ Z\subset \Lambda,   \mu(Z)=1 \}, \   \mu \text{ ergodic on $\Lambda$} \Big\}  
		\leq\sup \{h^-_{f,inf, B}(\mu),   \mu \text{ ergodic on $\Lambda$} \}= \\
		= \sup \{h^-_{f,sup, B}(\mu),   \mu \text{ ergodic on $\Lambda$} \}  \leq\lim_{\delta\rightarrow 0} \left(\sup\{h^-(\widehat{\mathcal{A}}), \ \hat{\mu}(\widehat{\mathcal{A}})>1-\delta, \ \mu \text{ ergodic on $\Lambda$} \}\right).
	\end{align*}
\end{corollary}

In the case of special TA-covering maps on tori, we obtain in Theorem \ref{VPtor} a Full Variational Principle for inverse entropy. In particular, this holds for Anosov endomorphisms without critical points on tori.

\textbf{Proof of Theorem \ref{VPtor}.}

From the Theorem of Sumi from \cite{S} stated in Section \ref{specialAnosov} it follows that there exists a topological conjugacy $\Phi: \mathbb T^d \to \mathbb T^d$ between $f$ and $f_L$. Thus by Proposition \ref{conjtopol} the inverse topological entropy of $f$  satisfies the relation $$h^-_f(\widehat{\mathbb T^d}_f) = h^-_{f_L}(\widehat{\mathbb T^d}_{f_L}) =  - \sum_{i: |\lambda_i|<1} \log |\lambda_i|,$$ where $\lambda_i$ are the eigenvalues of the matrix of $f_L$. Moreover $\mu$ is an $f$-invariant ergodic measure if and only if $\Phi_*\mu$ is an $f_L$-invariant ergodic measure. But if $m$ 
	is the Haar measure on $\mathbb T^d$, then $m$ is $f_L$-invariant ergodic and $h^-_{f_L}(m) =  - \sum_{i: |\lambda_i|<1}\log |\lambda_i|$ (see subsection \ref{toralendo}); for any other $f_L$-invariant ergodic measure $\mu$ we have $h^-_{f_L}(\mu) \le  - \sum_{i: |\lambda_i|<1}\log |\lambda_i|.$  So the measure $\rho := \Phi^{-1}_*m$ is $f$-invariant ergodic and $h^-_f(\rho) = h^-_{f_L}(m) =  - \sum_{i: |\lambda_i|<1}\log |\lambda_i|$. If $\nu$ is any other $f$-invariant ergodic measure, then $h^-_f(\nu) = h^-_{f_L} (\Phi_*\nu) \le  - \sum_{i: |\lambda_i|<1}\log |\lambda_i|.$ Thus we have a Full Variational Principle for inverse entropy if $f:\mathbb T^d \to \mathbb T^d$ is a special TA-covering map.
	$\hfill\square$

\section{Classes of Examples}\label{classesexp}


\subsection{Examples using expanding maps}

(1)	Let $\sigma: \Sigma_m^+\rightarrow \Sigma_m^+$ be the shift on the one-sided symbolic space $$\Sigma_m^+=\{(x_0,x_1,x_2,\ldots ) :  x_i\in\{1,2,\ldots, m\}, i\geq 0\}.$$ Then $h^-_{\sigma}(\mu)=0$ for any $\sigma$-invariant probability measure $\mu$, in particular for any Bernoulli measure $\mu_p$ corresponding to a probability vector $p=(p_1,p_2,\ldots, p_m)$.
 
(2) Let $f: S^1\rightarrow S^1$, $f(x)=dx \ (\text{mod  }1)$. Then $h^-_{f}(m)=0$, where $m$ is the Haar measure on $S^1$.


\subsection{Linear toral endomorphisms}\label{toralendo}
Let $A$ be a $p\times p$ hyperbolic matrix with $|\text{det}(A)|>1$ and with integer entries and non-zero eigenvalues $\lambda_1, \lambda_2, \ldots\lambda_p$. Let $f_A:\mathbb{T}^p\rightarrow \mathbb{T}^p$ be the associated toral endomorphism. Let $m$ be the normalized  Lebesgue measure (Haar measure) on $\mathbb{T}^p$. Then
	 $$h^-_{f_A, B}(m)=-\sum_{\{i: |\lambda_i|<1\}}\log|\lambda_i| , \  h_{f_A}(m)=\sum_{\{i: |\lambda_i|>1\}}\log|\lambda_i|. $$

	Indeed, for an arbitrary prehistory $\hat{z}$ and arbitrary $n\ge 1$ we have $C_1 \cdot\varepsilon^p \prod_{i : |\lambda_i|<1} \lambda_i^n \leq m(B_n^-(\hat{z},\varepsilon))\leq C_2 \cdot\varepsilon^p \prod_{i : |\lambda_i|<1} \lambda_i^n$, where $C_1$ and $C_2$ are positive constants independent on $n$ and on $\hat{z}$. Then 
	$$\lim_{n\rightarrow\infty}\frac{-\log m(B_n^-(\hat{z},\varepsilon))}{n}=-\sum_{\{i: |\lambda_i|<1\}}\log|\lambda_i|,$$
	and therefore 
	$h^{-}_{f_A, B}(m)=-\sum_{\{i: |\lambda_i|<1\}}\log|\lambda_i|$. 
	Note that, for $m$-a.e. $z\in \mathbb{T}^p$, $J_{f_A}(m)(z)=\prod_{i=1}^p | \lambda_i|$
	and so by Theorem \ref{t1} or Theorem \ref{zerob}, 
\begin{equation}\label{inventLeb}	
	h^-_{f_A, B}(m)=h^-_{f_A}(m)=h_{f_A}(m)-F_{f_A}(m)=-\sum_{\{i: |\lambda_i|<1\}}\log|\lambda_i|.
	\end{equation}
	One can cover $\mathbb {T}^p$ with $N_{n,\varepsilon}= ( \varepsilon^p \prod_{i : |\lambda_i|<1} \lambda_i^n)^{-1} $ $(n,\varepsilon) $-inverse Bowen balls with mutually disjoint interiors. It then follows that the inverse topological entropy of $f_A$ satisfies $h^-_{f_{A}}(\widehat{\mathbb {T}}^p)	= - \sum_{i: |\lambda_i(\mu)|<1}\log |\lambda_i|.$
	If $\mu$ is any ergodic $f$-invariant measure, then from Corollary \ref{corz} we know that $h^-_{f_A}(\mu)$ exists and 
	$$h^-_{f_A}(\mu)\leq  - \sum_{i: |\lambda_i(\mu)|<1}\log |\lambda_i|,$$
	where $\lambda_i$ are the the eigenvalues of $A$ with their multiplicities. 
	In this case Theorem \ref{VPtor} applies. 



\begin{example}\label{toralex}
	\normalfont Inverse measure-theoretic entropy can \textbf{distinguish between isomorphism classes }of measure preserving endomorphisms,  when they have the same forward entropy. 
	Let the matrices
	\begin{equation*}
		A_1=\left(\begin{array}{ccc}
			8 & 1 & 4\\
			0 & 3 & 1\\
			0 & 2 & 1
		\end{array}
		\right) \text {and }A_2=\left(\begin{array}{ccc}
			4 & 0 & 0\\
			3 & 6 & 2\\
			5 & 4 & 2
		\end{array}
		\right).
	\end{equation*}
	The eigenvalues of $A_1$ are $8, 2+\sqrt{3}$ and $2-\sqrt{3}$ and the eigenvalues of $A_2$ are $4, 4+2\sqrt{3}$ and $4-2\sqrt{3}$. Then $f_{A_1}$ and $f_{A_2}$ have the same (forward) entropy for the Haar measure $m$ on $\mathbb{T}^3$, namely $\log 8+\log (2+\sqrt{3})$. On the other hand, we have
	$$h^-_{f_{A_1}}(m)=-\log(2-\sqrt{3}) \text{ and }h^-_{f_{A_2}}(m)=-\log(4-2\sqrt{3}),$$  thus  $(\mathbb{T}^3, f_{A_1}, m)$ and $(\mathbb{T}^3, f_{A_2}, m)$ are not isomorphic. 
\end{example}

\subsection{Fat baker's transformations}\label{fatbaker} 

Let $K=[-1,1]\times[-1,1]$. For $0<\beta<1$		
consider as in \cite{AY} the transformation $T_\beta: K\rightarrow K$ defined by
$$T_\beta(x,y)=\begin{cases}
(\beta x+(1-\beta), 2y-1)&   y\geq 0, \\
(\beta x-(1-\beta), 2y+1)&   y< 0.
\end{cases}$$
The maps obtained for $\beta\in\left(\frac{1}{2},1\right)$ are called fat
baker's transformations. In this case the attractor is the whole square $K$. Let $z=(x,y)\in K$. Notice that we have overlaps between the images of the two branches of $T_\beta$. There exists a $T_\beta$-invariant ergodic probability measure $\mu_{SRB}^\beta$ on $K$ (called the Sinai-Ruelle-Bowen measure of $T_\beta$), such that for Lebesgue a.e. point $(x,y) \in K$ the measures
$\frac{1}{n}\sum\limits_{i=0}^{n-1}\delta_{T_\beta^i(x,y)}$
converge weakly to $\mu_{SRB}^\beta$.
Consider the iterated function system $\mathcal{S}_\beta=\{S_1^\beta, S_2^\beta\}$ consisting of $S_1^\beta(x)=\beta x+(1-\beta),  S_2^\beta(x)=\beta x-(1-\beta)$, $x\in [-1,1]$ and let $\pi_\beta:\Sigma_2^+\rightarrow [-1,1]$ be the canonical projection to the limit set of $\mathcal{S}_\beta$. Let $\mu_{(\frac{1}{2}, \frac{1}{2})}$ be the Bernoulli measure on $\Sigma^+_2$ corresponding to the vector $\left(\frac{1}{2},\frac{1}{2}\right)$ and $\nu_\beta:=\pi_{\beta*}\mu_{(\frac{1}{2}, \frac{1}{2})}$. The Sinai-Ruelle-Bowen measure $\mu_{SRB}^\beta$ for $T_\beta$ is equal to $\nu_\beta\times m$ (see \cite{AY}), where $m$ is the normalized Lebesgue measure on $[-1,1]$. 
Let now $z=(x, y)\in K$ with $y \ne \pm 1$, and $\hat{z}$ be a $T_\beta$-prehistory of $z$. 
As $T_\beta$ is contracting with constant factor $\beta$ in first coordinate and expanding in second coordinate, it follows that for $\vp>0$ small, 
$B_n^-(\hat{z},\varepsilon)=B(x,\beta^n\varepsilon)\times B(y,\varepsilon)$ for all $n \ge 1$. Thus
\begin{equation}\label{invbaker}
\mu_{SRB}^\beta(B_n^-(\hat{z},\varepsilon))=\nu_\beta(B(x,\beta^n\varepsilon))\cdot\varepsilon.
\end{equation}
 Let
$$\underline{\delta}(\nu_\beta)(x)=\liminf_{r\rightarrow 0}\frac{\log \nu_\beta(B(x,r))}{\log r}\text{ and }\overline{\delta}(\nu_\beta)(x)=\limsup_{r\rightarrow 0}\frac{\log \nu_\beta(B(x,r))}{\log r},$$
be the lower, respectively upper pointwise dimension of $\nu_\beta$ at $x$. It follows from \cite{FH} (or \cite{M}) that $\nu_\beta$ is exact dimensional and thus $\underline{\delta}(\nu_\beta)(x)=\overline{\delta}(\nu_\beta)(x)$ for $\nu_\beta$-a.e. $x\in [-1,1]$ and the common value $\delta(\nu_\beta)(x)$ is constant for $\nu_\beta$-a.e $x$; we denote this constant by $\delta(\nu_\beta)$. 
Then from (\ref{invbaker}) and since $\delta(\nu_\beta)=\lim\limits_{r\rightarrow 0}\frac{\log \nu_\beta(B(x,r))}{\log r}$ for $\nu_\beta$-a.e. $x$, it follows that $h^-_{T_\beta,B}(\mu_{SRB}^\beta)$ exists and 
\begin{equation}\label{invb}
h^-_{T_\beta,B}(\mu_{SRB}^\beta)= |\log\beta|\cdot\delta(\nu_\beta).
\end{equation}
In \cite{MU4} Mihailescu and Urba\'nski introduced the topological overlap number; in our case denote the topological overlap number of $S_\beta$ by $o(\beta)$. Then from \cite{M} it follows that 
\begin{equation}\label{invc}
\delta(\nu_\beta)=HD(\nu_\beta)=\frac{\log 2- \log o(\beta)}{|\log\beta|}.
\end{equation}
From (\ref{invb}) and (\ref{invc}), for any $\beta\in \left(\frac{1}{2},1\right)$ the inverse entropy of the SRB measure of $T_\beta$ satisfies, 
\begin{equation}\label{fatformula}
h^-_{T_\beta, B}(\mu_{SRB}^\beta) = h^-_{T_\beta}(\mu_{SRB}^\beta) = \log 2- \log o(\beta).
\end{equation}

\subsection{Tsujii endomorphisms} \label{tsujii}
Let the family of endomorphisms $T: S^1\times \mathbb{R}\rightarrow  S^1\times \mathbb{R}$, 
\begin{equation}\label{TTs}
T(x,y)=(lx, \lambda y+f(x)),
\end{equation}
where $l\geq 2$ is an integer, $1/l < \lambda < 1$ and $f:S^1\rightarrow \mathbb{R}$ is a $C^2$ function, defined in \cite{Ts}. The map $T$ is a skew product over the expanding map $\tau : x\mapsto lx$, having uniform contraction in the fibre direction,  hence $T$ is an Anosov endomorphism. Let $\mu_{SRB}$ denote the SRB measure for $T$.
If $\lambda l<1$, the SRB measure $\mu_{SRB}$ is totally singular with respect to the Lebesgue measure on $S^1 \times \mathbb R$ because $T$ contracts area. If $\lambda l> 1$ as in our setting, then the situation is more interesting: in some cases the SRB measure is totally singular with respect to Lebesgue measure and in other cases it is absolutely continuous with respect to Lebesgue measure.


In \cite{Ts} Tsujii proved that for a generic map $T$ the corresponding SRB measure is absolutely continuous with respect to the Lebesgue measure. 
Fix an integer $l\geq 2$. Let $\mathcal{D}\subset (0,1)\times C^2(S^1, \mathbb{R})$ be the set of all pairs $(\lambda,f)$ for which the SRB measure is absolutely continuous with respect to the Lebesgue measure on $S^1\times \mathbb{R}$. Then $\mathcal{D}$ contains an open and dense subset of $(1/l, 1)\times C^2(S^1, \mathbb{R})$.
Consider the map $$\tau : S^1\rightarrow S^1, \ \  \tau(x)=lx \ \ \text{mod } 1.$$
Let $\Sigma_l^+$ be the one sided shift space on $l$ symbols, i.e the set $\{\omega=(\omega_1,\omega_2,\ldots)   :  \omega_i\in \{1,\ldots,l\}, i\geq 1 \}$. Let $\pi: S^1\rightarrow \{1,2,\ldots l \}$ be defined by $\pi(x)=j$ where 
$x\in 
\left[ \frac{j-1}{l}, \frac{j}{l}\right)$. If $\omega\in \Sigma_l^+$ let $\pi(x)\omega$ be the element of $\Sigma_l^+$ obtained by putting $\pi(x)$ in front of $\omega$. For $(x,\omega)\in S^1\times \Sigma_l^+$ let 
$$S(x, \omega)=f(x_1(\omega))+ \lambda f(x_2(\omega))+\lambda^2 f(x_2(\omega))+\cdots,$$
where $x_1(\omega)$ is the unique point $y$ from $\left[ \frac{\omega_1-1}{l}, \frac{\omega_1}{l}\right)$ such that $\tau(y)=x$, $x_2(\omega)$ is the unique point $y$ from $\left[ \frac{\omega_2-1}{l}, \frac{\omega_2}{l}\right)$ such that $\tau(y)=x_1(\omega)$, and so on. Let us consider the maps
\begin{align}\label{maps}
\Psi: S^1\times \Sigma_l^+\rightarrow S^1\times \mathbb{R},\ \ \  
\Psi(x,\omega)=(x, S(x,\omega)), \text  { and  }\\
\Theta: S^1\times \Sigma_l^+\rightarrow S^1\times \Sigma_l^+,\ \ \ \Theta(x,\omega)=(\tau(x), \pi(x)\omega) \nonumber.
\end{align}
Then
\begin{align*} 
\Psi \circ \Theta(x,\omega)&=\Psi(\tau(x), \pi(x)\omega) = (\tau(x), S(\tau(x), \pi(x)\omega)
=(\tau(x), f(x)+\lambda f(x_1(\omega)) + \lambda^2 f(x_2(\omega))+\cdots)\\ &= (\tau(x), f(x)+\lambda(f(x_1(\omega)) + \lambda f(x_2(\omega))+\cdots))=T(\tau(x), S(x,\omega))= T\circ\Psi (x,\omega).
\end{align*} 

Fix an integer $l\geq 2$ and $\lambda$ such that $l^{-1}<\lambda<1$. Let $g:S^1\rightarrow \mathbb{R}$ a $\mathcal{C}^2$ function, $k\geq 1$, 
and $\varphi_i:S^1\rightarrow \mathbb{R}$, $1\leq i\leq k$ be $\mathcal{C}^\infty$. For $t=(t_1, t_2, \ldots, t_k)\in \mathbb{R}^k$,  consider the family of functions
$$f_t(x)=g(x)+\sum_{i=1}^k t_i\varphi_i(x): S^1\rightarrow \mathbb{R}.$$ 

Then we obtain the corresponding family of maps
\begin{equation}\label{Tt}
T_t: S^1\times\mathbb{R}\rightarrow  S^1\times\mathbb{R} \ \ \ \ T_t(x,y)=(lx,\lambda y+f_t(x)).
\end{equation}

\begin{theorem}\label{Tsujiestimateinv}
	There exists a family $\varphi_i: S^1\rightarrow \mathbb{R}$, $1\leq i\leq k$ of $C^\infty$  functions such that, if  $\{T_t\}_{t\in\mathbb{R}^k}$ are the maps  defined in (\ref{Tt}),
	then for Lebesgue a.e. $t\in\mathbb{R}^k$ the following estimates for the inverse entropy of the SRB measure $\mu_{SRB}^t$ of $T_t$ hold:
	\begin{equation}\label{invts}
	\frac{1}{2}|\log \lambda|\leq h^-_{T_t,inf, B}(\mu_{SRB}^t) \leq h^-_{T_t,sup, B}(\mu_{SRB}^t)|\leq h^-_{T_t}(\mu_{SRB}^t)=|\log \lambda|.
	\end{equation} 
\end{theorem}
\begin{proof}
In general, let $T$ be an Anosov endomorphism defined as in (\ref{TTs}) and denote by $\mu_{SRB}$ its SRB measure (which exists since $T$ is Anosov). Recall the definition of the map $\Psi$ from (\ref{maps}). Let $\mu_x=\Psi_*(\delta_x\times \nu)$ for $x\in S^1$, where $\delta_x$ is the point mass at $x$ and $\nu$ is the Bernoulli measure on $\Sigma_l^+$ associated to the probability vector $(1/l, \ldots, 1/l)$. The measures $\mu_{x}$, $x\in S^1$, form a canonical family of conditional measures of the SRB measure $\mu_{SRB}$ with respect to the partition of $S^1\times \mathbb{R}$ into fibres $\{x\}\times \mathbb{R}$, $x\in S^1$ (see \cite{Ro}, \cite{Ts}).
In general, for a finite Borel measure $\rho$ on $\mathbb{R}$ and any $r>0$  define 
$$\|\rho\|_r=\left(\int_{\mathbb{R}} \left(\rho(B(z,r))\right)^2dz\right)^{\frac{1}{2}}.$$
For $r>0$,  define
$I(r):=r^{-2}\int_{S^1}\|\mu_{x}\|_r^2 \ dx.$
We want to see when is the following condition satisfied:
\begin{equation}\label{em}
\liminf_{r\rightarrow 0} I(r)<\infty.
\end{equation}
We will show that condition (\ref{em}) is satisfied for a large class of maps $T_t$ (defined in  (\ref{Tt})). With the  notations from \cite{Ts}, assume now that $\limsup_{q\rightarrow \infty}d(q)<\lambda l$. Thus there exists some $\beta>0$ such that $\limsup\limits_{q\rightarrow \infty}d(q)<\beta<\lambda l$. Let $M=\sup d(q)$ and let $q_0\geq 1$ be such that $d(i)\leq \beta$ for $i\geq q_o$. But from Proposition 9 of \cite{Ts}, we know that $e(q)\leq \prod_{i=1}^q d(i)$ for $q\geq 1$. Hence 
$$e(q)\leq \prod_{i=1}^q d(i)\leq M^{q_0}\beta^{q-q_0}\leq(\lambda l)^q,$$
for $q$ sufficiently large. Then Proposition 8 of \cite{Ts} implies that  $(\lambda,f)$ is in $\mathcal{D}^o$ and $\liminf\limits_{r\rightarrow 0}I(r)<\infty$. Recall the definition of $T_t$ in (\ref{Tt}) and denote by $\mu_{SRB}^t$ its SRB measure. 
From the proof of Theorem 1 of \cite{Ts} it follows that we can choose  the functions $\varphi_i$, $1\leq i\leq k$ such that for Lebesgue a.e. $t\in\mathbb{R}^k$, we have $\limsup\limits_{q\rightarrow \infty}d(q)<\beta<\lambda l$ and thus (\ref{em}) holds as shown above. Let us then fix the functions $\varphi_i$, $1\leq i\leq k$ and $t\in\mathbb{R}^k$ as above. Hence
$\mu^t_{SRB}$ is absolutely continuous with respect to the Lebesgue measure $m$ on $S^1\times \mathbb{R}$, and from (\ref{em}) its Radon-Nikodym derivative $\zeta_t$ is in $L^2(S^1\times \mathbb{R})$.
For $z\in S^1 \times  \mathbb{R}$ and  $r>0$  small enough, $T_t$ is injective on $B(z,r)$. 
Thus,
$$\mu_{SRB}^t(B(z,r))=\int_{B(z,r)}\zeta_t\ dm, \ \ \mu_{SRB}^t(T_t(B(z,r)))=\int_{T_t(B(z,r))} \zeta_t\ dm, $$
and therefore
\begin{equation}\label{der}
\frac{\mu_{SRB}^t(T_t(B(z,r)))}{\mu_{SRB}^t(B(z,r))}=\frac{\int_{T_t(B(z,r))} \zeta_t \ dm  }{  \int_{B(z,r)}\zeta_t \ dm} = \frac{ \int_{T_t(B(z,r))} \zeta_t \ dm}{m(T_t(B(z,r)))  } \cdot \frac{m (B(z,r))} {  \int_{B(z,r)}\zeta_t \ dm}\cdot \frac{ m(T_t(B(z,r))) }{m(B(z,r))}. 
\end{equation}
Now let $r\rightarrow 0$; then from (\ref{der}) and Lebesgue Density Theorem, we obtain for Lebesgue a.e. $z\in S^1\times\mathbb{R}$,
$J_{T_t}(\mu_{SRB}^t)(z)=\frac{\zeta_t(T_tz)}{\zeta_t(z)}\cdot \lambda l$.
Thus the folding entropy of the SRB measure $\mu_{SRB}^t$ of $T_t$ satisfies
$$F_{T_t}(\mu_{SRB}^t)=\int \log J_{T_t}(\mu_{SRB}^t)d\mu_{SRB}^t= \int \log \zeta_t (T_tz)d\mu_{SRB}^t(z)-\int \log \zeta_t(z)d\mu_{SRB}^t(z) +\log (\lambda l).$$ 
From the $T_t$-invariance of $\mu_{SRB}^t$ and the last formula, it follows that 
\begin{equation}\label{fol}
F_{T_t}(\mu_{SRB}^t)=\log (\lambda l). 
\end{equation}
From Proposition \ref{c1} we have $h^-_{T_t,sup,B}(\mu_{SRB}^t)\leq h_{T_t}(\mu_{SRB}^t)-F_{T_t}(\mu_{SRB}^t)$. On the other hand, $$h_{T_t}(\mu_{SRB}^t)=\log l,$$ as $T_t$ is given by $lx$ in first coordinate and contracts in second coordinate, and since $\mu_{SRB}^t$ is absolutely continuous on the unstable manifolds of $T_t$ for the chosen $t$. So from  above and (\ref{fol}) we obtain that
\begin{equation}\label{eq7}
h^-_{T_t,sup,B}(\mu_{SRB}^t)\leq |\log \lambda|.
\end{equation} 

However recall that $\mu_{SRB}^t(A)=\int_A \zeta_t\  dm$ for every measurable set $A\subset S^1\times \mathbb{R}$, and in our case for the parameter $t$ chosen above we have $\zeta_t\in L^2(S^1\times\mathbb{R})$. Hence, 
\begin{equation}\label{zetatma}
\mu_{SRB}^t(A)\leq \|\zeta_t\|_2\cdot  m(A)^{\frac{1}{2}}.
\end{equation}
Let $\hat{z}$ be an arbitrary prehistory of $z=(x,y)\in S^1\times \mathbb{R}$ with respect to the endomorphism $T_t$. Since $B^-_n(\hat{z}, \varepsilon)= T_t^n(B_n( z_{-n}, \varepsilon)) \subset B(x,\varepsilon) \times B(y, \lambda^n\varepsilon)
$,
we have $m(B^-_n(\hat{z}, \varepsilon))\leq \lambda^n \cdot\varepsilon^2$. Then, from (\ref{zetatma}),
$$\liminf_{n\rightarrow\infty}\frac{-\log \mu_{SRB}^t(B^-_n(\hat{z}, \varepsilon)) }{n}\geq \liminf_{n\rightarrow\infty} \frac{-\log m(B^-_n(\hat{z}, \varepsilon)) - 2\log \|\zeta_t\|_2}{2n}  \geq -\lim_{n\rightarrow\infty}\frac{\log (\lambda^n \varepsilon^2)}{2n}=\frac{1}{2}|\log \lambda|.$$
Therefore $h^-_{T_t,inf, B}(\mu_{SRB}^t)\geq\frac{|\log \lambda|}{2} $. Also by Theorem \ref{t1}, since $h_{T_t}(\mu^t_{SRB})= \log l$ and by (\ref{fol}), it follows that the inverse partition entropy $h^-_{T_t}(\mu_{SRB}^t)$ exists and  $$h^-_{T_t}(\mu_{SRB}^t)=|\log \lambda|.$$
Hence for Lebesgue a.e $t\in \mathbb{R}^k$, the lower/upper inverse metric entropies, and the inverse partition entropy of the SRB measure $\mu_{SRB}^t$ of $T_t$ satisfy 
\begin{equation}
\frac{1}{2}|\log \lambda|\leq h^-_{T_t,inf, B}(\mu_{SRB}^t) \leq h^-_{T_t,sup, B}(\mu_{SRB}^t)\leq  h^-_{T_t}(\mu_{SRB}^t)=|\log \lambda|.
\end{equation} 
\end{proof}

\smallskip

Eugen Mihailescu, 
		Institute of Mathematics of the  Romanian Academy, Calea Grivitei 21,  Bucharest,  Romania. \ \ \ \  \ \ \ \ 
	\email{Eugen.Mihailescu@imar.ro \ \ \ \ \ \   \  \   www.imar.ro/$\sim$mihailes}

\smallskip 

R. B. Munteanu,  Department of Mathematics, University of Bucharest, 14 Academiei St., 010014,  Bucharest, Romania.  \ \ \ \ \ \ 
			\email{radu-bogdan.munteanu@g.unibuc.ro}
\end{document}